\documentclass[12pt]{article}
\usepackage{personal}
\usepackage{realoracles}

\title{Defining Real Numbers as Oracles}
\jtauthor
\date{May 17, 2023}

\begin{document}\maketitle
\begin{abstract}
A real number is a rule that, when provided with a rational interval, answers Yes or No depending on if the real number ought to be considered to be in the given interval. Since the goal is to define the real numbers, this can only motivate the definition of which rules should be considered a real number. The rule must satisfy five properties and any rule that does so we call an oracle. Three of the properties ensure that we do not have multiple oracles representing the same real number. The other two properties ensure that the oracle does narrow down to a single real number. The most important property is the Separating property which ensures that if we divide a Yes interval into two parts, then one part is a Yes interval while the other is a No interval; the exception is if the division point, which is a rational number, happens to be the desired real number in which case both intervals are Yes intervals. We explore various examples and algorithms in using oracles in addition to establishing that the oracles do, in fact, form the field of real numbers. The concept of a Family of Overlapping, Notionally Shrinking Intervals is defined and found to be an essential tool in working with oracle arithmetic. Mediant approximations, which are related to continued fraction representations, naturally arise from an oracle perspective. We also compare and contrast with other common definitions of real numbers, such as Cauchy sequences and Dedekind cuts, in which the conclusion is that the oracle perspective is somewhat of a master map to the other definitions. We do an explicit example to contrast oracle arithmetic with decimal arithmetic and continued fraction arithmetic. 
\end{abstract}

\tableofcontents

\section{The Oracle of \texorpdfstring{$r$}{r}}\label{sec:ora}

Our goal is to give a new definition of real numbers. This definition strikes a balance between having a solid definition of a real number and being useful in computing. Most definitions of real numbers, such as decimal expansions and Cauchy sequences, seem to require explicit computations to even define the number. This is problematic given the infinite nature of real numbers. Other definitions, such as the Dedekind cut, do not require computations to state, but they seem to offer no useful guidance in computing the number. 

What we want to provide is an object that can be defined without computation, but provides the underlying mechanism for computing the real number to any precision one wants. In some sense, we are producing the framework for producing approximations to real numbers given the particulars of a given real number.  This is analogous to, in functional programming,  having a generic sorting algorithm that can be used in a variety of situations by supplying a comparison function to make a pairwise sorting decision while the framework handles the actual details of the global sorting.

Since this is a mechanism that provides answers to our questions, it seems reasonable to call it an oracle. We assert that the best way to understand a real number is that it is an oracle that reveals itself by answering, under repeated questioning, whether it is in various rational intervals. 

Informally, the Oracle of $r$ is a rule which, given two rational numbers, will return Yes if $r$ ought to be in the inclusive rational interval defined by the two rational numbers and returns No otherwise. If the answer was a Yes, then we say the interval is an $r$-Yes interval; if the answer is a No, then we say the interval is an $r$-No interval.  If we are just speaking of one oracle, we may simply say a Yes-interval or No-interval. 

There is a very short teaser version of this paper \cite{taylor23teaser} as well as a short overview paper \cite{taylor23over}. They serve as a more gentle introduction into these ideas. 

\subsection{Interval Notation}

We will define a \textbf{rational interval} $a:b$ as a binary rule that, for a given rational $q$, yields $1$ (Yes) if $q$ is between $a$ and $b$, inclusive, and yields $0$ (No) if not. We say that $q$ is in $a:b$ if the interval rule yielded an affirmative answer.  We identify $b:a$ as the same rule. The interval $a:b$ contains the interval $c:d$ if being in the interval $c:d$ implies being in the interval $a:b$. We will write $a\lt b$ if we want to indicate the ordering relation between the rationals and that they are strictly not equal. We may also require knowing the ordering, but still allow equality. In that case, we use the notation $a \lte b$ to indicate $a \leq b$.  We will also write $a:b < c:d$ to indicate the situation in which $a$ and $b$ are less than $c$ and $d$ which implies that all rationals $q$ in $a:b$ are less than all rationals $r$ in $c:d$.

Please take note that the intervals not only have rational endpoints but they consist entirely of rational points as we do not have real numbers. Throughout, unless we say otherwise, intervals always include their endpoints. If we had the real numbers, these would be closed intervals with rational endpoints.  The symbol $:$ was chosen to represent the two endpoints with something in between. Being dotted is to be a reminder that it is just the discrete rational interval under discussion. 

We define the term \textbf{singleton} to be intervals of the form $a:a$, namely the only number in the interval is the rational $a$. We also define the term \textbf{rational neighborhood} or a \textbf{neighborly interval} to be intervals which are not singletons, that is, $a \lt b$ is its form. When we write $a:b$, we do allow for $a=b$.

A rational $q$ is \textbf{strictly contained} in $a\lt b$ if $a < q < b$. An interval $c:d$ is \textbf{strictly contained} in $a\lt b$ if both $c$ and $d$ are strictly contained in $a \lt b$. We say that $c \lte d$ is \textbf{nested in} $a\lte b$ if $a \leq c \leq d \leq b$; $a \lte b$ then is also said to \textbf{contain} $c \lte d$. 

The length of an interval $|a \lte b |$ is $b-a$. 

Assume for this paragraph that we have $a \lt b$, $c \lt d$, and $a \leq c$. We say that $a\lt b$ and $c \lt d$ are \textbf{strictly overlapping} if $c \leq b$ and $b < d$. Two intervals are \textbf{overlapping} if they are either strictly overlapping or one is nested inside the other. An interval is nested inside itself. The two intervals are \textbf{disjoint} if $b < c$. If the two intervals are disjoint, then we define the \textbf{separation} $S$ to be $c-b$; if they are not disjoint, then the separation is 0.  The \textbf{distance} $D$ between the two intervals is the maximum of the difference between the endpoints. The length of an interval is the same as the distance of the interval to itself. Note that the non-negative difference between two points, one from each interval, will never be less than the separation and never greater than the distance. That is, if $q$ is in $a\lt b$ and $p$ is in $c \lt d$, then $S \leq |p-q| \leq D$. We have that  $S = 0$ exactly when the two intervals are overlapping. They are strictly overlapping when, in addition to $S=0$,  the distance between the two intervals is larger than both the lengths. If the distance is smaller than the longer length, then the shorter one is strictly contained in the other. If the distance is the same as the lengths, then they are the same interval. These statements are easy to establish by cases as is the fact that the distance of the intervals is the same as the separation plus the lengths for disjoint intervals. 

Note that we do have transitivity of interval inequality, that is, if $a:b < c:d$ and $c:d < e:f$, then $a:b < e:f$. This follows immediately from the transitivity of inequalities on rationals.

\subsection{Defining the Oracle}

The Oracle of $r$ is a rule $R$ defined on all rational intervals, including singletons, that returns values of 1 and 0, and satisfies: 
\begin{enumerate}
    \item Consistency. If $c:d$ contains $a:b$ and $R(a:b) = 1$, then $R(c:d) = 1$.
    \item Existence. $R(a:b) = 1$ for some rational interval $a:b$.
    \item Separating. If $R(a:b)=1$, then for a given $c \neq a, b$ in $a:b$, either $R(c:c) = 1$ or $R(a:c) \neq R(c:b)$. 
    \item Rooted. There is at most one $c$ such that $R(c:c) =1$.
    \item Closed. If $c$ is contained in all $R$-Yes intervals, then $R(c:c) = 1$.
\end{enumerate}

Note that in the above $a$, $b$, $c$, and $d$ are all rationals; we do not apply this to irrationals since we are defining them here. We generally obey the convention that $r$ represents naming the oracle as a number and $R$ is the rule. There is not a distinction, in fact, between $r$ and $R$, but the notation helps make the connection to how we customarily think of real numbers versus the oracle version.\footnote{As it becomes increasingly clear that the oracles do fit a model of the real numbers, we will move away from this distinction and also start to use Greek letters, such as $\alpha$, to denote oracles.}

Consistency tells us that if $c:d$ contains $a:b$ and $R(c:d) = 0$, then $R(a:b) = 0$. For if $R(a:b)=1$, then consistency implies $R(c:d)=1$ which contradicts our starting assumption. We will refer to this as part of the Consistency property. 

For further elaboration on what each of these properties means:

\begin{enumerate}

    \item Consistency says that Yes propagates upwards to larger intervals while No propagates downwards to smaller intervals. The Oracle never contradicts itself. 
    
    In our examples, this is usually trivially true by definition of the rule, corresponding to this not being of much practical use. In some sense, this property is here to ensure that there is a single oracle for a given real number.  

    \item The existence requirement says that there is definitely an interval which the Oracle confirms being in. It is required to avoid the trivial $R(a:b) = 0$ for all $a:b$ which would satisfy all the other conditions, including the separating property due to it only applying to Yes-intervals. 
    
    The existing interval also gives us a place to start in our approximation schemes that we discuss later. 
    
    It is usually easy to verify in any practical example as any Yes-interval will do and we can usually find a very bad approximate interval that is easy to establish.  

    \item Separation is needed to ensure that the Yes will continue to propagate downwards so that we can make progress on narrowing in on $r$. This is crucial to the use of the Oracle in approximating what we take to be a single number $r$. The possibility of $R(c:c) = 1$ occurs exactly when the real number $r$ corresponds to the rational number $c$. 
    
    Proving that a given oracle is, in fact, a rational can be impossible, mirroring difficulties with other real number approaches. 

    Separation could alternatively be phrased as insisting that if $R(a:b) = 1$, then at least one of $R(a:c)$ or $R(c:b)$ is 1 and if they are both 1, then $R(c:c) = 1$. We chose our formulation to emphasize that being able to separate the intervals is a key feature we desire to implement and that it fails only if the overlapping endpoint is the number of interest. 

    We will call this Interval Separation when we need to distinguish from other separation properties, as discussed later. 

    We also want to note that while we did exclude $c$ from being the endpoints, this is not because that would yield a problem, but rather because it is a vacuous statement. Indeed, if $c=a$, then either $R(a:a)=1$ or $R(a:a)=0$. If it is the former, then we have satisfied the singleton part of the separation property. If $R(a:a)=0$, then since $R(a:b)=1$ by assumption, we have $R(a:a) \neq R(a:b)$ which is the other part of the property. We therefore concentrate our discussion solely on $c$ strictly between $a$ and $b$.
    
    \item Being Rooted ensures that there is just one single number under discussion. Without this, we could have the rule, for example, $R(a:b) = 1$ for all $a:b$ that contain a subinterval of a fixed given interval $c:d$. This is compatible with all the properties except for being rooted. In particular, the separation property is fine since given an interval $a:b$ overlapping  $c:d$, and any number $e$ that is in both $a:b$ and in $c:d$, then $R(e:e) = 1$ and hence $R(a:e)=R(e:b) = 1$ is perfectly allowed even though this applies for multiple rational numbers $e$.
    
    In the arithmetic operations, being rooted is proven by taking two rational singletons and then exhibiting an interval of sufficiently small length that the two can be distinguished. 

    \item An important benefit of the Closed property is that the arithmetic of rationals is unchanged. As we shall see, oracle arithmetic involves interval arithmetic. With the presence of the singletons, we can do arithmetic with them as we do with fractions. 
    
    It also helps ensure the uniqueness of an oracle. For example, we could define a $0^+$ oracle by being all the intervals of the form $0\lt a$ and another oracle $0^-$ as intervals of the form $a \lt 0$. These two proto-oracles satisfy the other properties. Both could be said to represent $0$, but they are distinct. They obviously fail to satisfy the Closed property. If we add in $0:0$ to them to satisfy the Closed Property, then they fail Consistency since we now need to include all intervals that include $0:0$. This is how we achieve uniqueness. 
    
    Most representations of real numbers have this issue with rationals. Decimals have the issue of repeating 9s. Dedekind cuts have the question of whether to include the rational or not in its own cut. Continued fraction representations are unique on irrational numbers, but they have two representatives for rationals. 
    
    In our constructions, we often need to add this in explicitly. Establishing that a rational real is rational can be difficult, possibly impossible. This seems to be one of the essential difficulties with real numbers, to varying degrees.

    We will sometimes use the fact that $R(a:a)=0$ implies the existence of a Yes interval $b:c$ that does not contain $a$. To see this, if $b:c$ did not exist, then $a$ would be in every Yes interval and thus $R(a:a)=1$ by the Closed property. 

\end{enumerate}

This definition does not avoid all of the common downsides of the other definitions of real numbers, but it does reduce them to a context that reflects how real numbers get used in practice. It illustrates those downsides as fundamental to the nature of real numbers while avoiding inessential downsides that other approaches have. See Section \ref{sec:others} for a discussion of other approaches.  

A \textbf{singleton oracle }is an oracle whose Yes intervals includes a singleton. These will eventually be identified with the rationals. A \textbf{neighborly oracle} is an oracle whose Yes intervals do not include a singleton. These are the irrationals.

\section{Basic Properties of an Oracle}

We establish some basic facts of oracles which we will find useful throughout. 

We start with the extremely helpful properties involving intersections and unions. We then discuss the uniqueness of oracles, with the related notions of determining ordering and compatibility. In particular, we define what we mean by $r = s$, $r<s$, $r>s$, and $r?s$ for oracles $r$ and $s$. 

The relation $r?s$ means that they have not been established to be equal, but they are compatible with that hypothesis. Typically, we will want to specify an interval of compatibility, say $a:b$, which is the smallest known interval tested that they both reported Yes on. To denote that, we can use $a:r?s:b$ 

\subsection{Intersections and Unions}

Let $a \leq b \leq c \leq d$ with all of them being rationals. The intersection of $a\lte c$ and $b\lte d$ is $b\lte c$ while the union is $a\lte d$. The two intervals $a\lte b$ and $c\lte d$ are disjoint if $b < c$. 

One can either take those as statements from the usual set definitions or take these as the definitions of those terms. It is immediate from the Consistency property that Yes-intervals are closed under unions and that No-intervals are closed under intersection. We will prove that Yes (No) intervals are closed under pairwise intersections (unions). 

We start with establishing that the intersection of two $R$-Yes intervals is an $R$-Yes interval.

\begin{proposition}\label{pr:inter}
Let $R$ be an oracle and $a \leq b \leq c \leq d$. If $R(a:c) = 1 = R(b:d)$, then $R(b:c) = 1$.
\end{proposition}

\begin{proof}
  
  This follows from the Separation and Consistency properties. From Consistency, we have $R(a:d) = 1$ since it contains an $R$-Yes interval (two in fact, but we just need one). Then apply separation to $a:c$ and $c:d$. If $R(c:c) = 1$, then $R(b:c) = 1$ since it contains $c:c$ and we are done. Otherwise, Separation tells us $R(a:c) \neq R(c:d)$. By assumption, we know $R(a:c) = 1$ so that implies $R(c:d) = 0$. 
  
  We can then consider what follows from $R(b:d) = 1$.  From Separation based on $c$, we have $R(b:c) \neq R(c:d)$ since we are specifically in the case that $R(c:c) \neq 1$. As we do have $R(c:d) = 0$, we must have $R(b:c) = 1$.
  
  We have established that Yes-intervals are closed under pairwise intersection. 
\end{proof}

Let us now establish that two overlapping No-intervals are closed under union. 

\begin{proposition}\label{pr:union}
Let $R$ be an oracle and $a \leq b \leq c \leq d$.  If $R(a:c) = 0$ and $R(b:d) = 0$, then $R(a:d) = 0$. 
\end{proposition}

\begin{proof}
    Let's assume that $R(a:d) = 1$ and find a contradiction. With that assumption, $R(a:b) = 1$ from Separation applied to $a:b$ and $b:d$ with the fact that $b:d$ was a No-interval. But this contradicts $R(a:c)= 0$ since it contains $a:b$ and therefore Consistency would demand $a:c$ must be a Yes-interval. We can conclude that $R(a:d) =0$.
\end{proof}

Also, two disjoint intervals cannot both be yes. Notice that $b$ and $c$ are separated by a strict inequality in the following proposition. 

\begin{proposition} \label{pr:disjoint}
Let $R$ be an oracle. Let $a:b$ and $c:d$ be such that no rational $q$ is in both $a:b$ and $c:d$. Then we cannot have both $R(a:b) = 1$ and $R(c:d) = 1$. 
\end{proposition}

\begin{proof}
By relabelling, we can assume $a:b:c:d$, that is, we assume $b$ and $c$ are the closest endpoints of the two intervals. This can be done as they are disjoint as the other possibilities, such as $a:d:b:c$, $a:c:b:d$ or $d:a:b:c$ all involve having a rational in common between the intervals.  

By being Rooted, we must have at least one of $R(b:b) = 0$ or $R(c:c) = 0$ holding true, probably both. Let's assume $R(b:b) = 0$ without loss of generality.
 
Either $R(a:b) = 0$ or $R(a:b)=1$. If it is the first case, then we have established what we want to claim. So let us assume we are in the second case. 

Since $R(a:b) = 1$ and $a:b$ is contained in $a:d$, then, by Consistency, we have $R(a:d) = 1$. Separation applies for the intervals $a:b$ and $b:d$. Thanks to $R(b:b) = 0$, we have $1 = R(a:b) \neq R(b:d)$ and thus $R(b:d) = 0$. As $c:d$ is contained in $b:d$ and No intervals propagate downwards by Consistency, we must have $R(c:d)=0$  as we were to establish in this case. 
\end{proof}

\begin{corollary}\label{cor:pair-inter}
    Given an oracle and two Yes intervals for that oracle, their intersection is non-empty and is a Yes interval. 
\end{corollary}

\begin{proof}
    By the above proposition, if two intervals are Yes, then there must be a rational $q$ contained in both intervals. Let $a:b$ and $c:d$ be the two intervals. For two intervals to have a common intersection, they must, up to relabelling,  either satisfy $a:c:b:d$ or $c:a:b:d$ or $a:c:d:b$. For the latter two, the intersection is one of the given Yes intervals and thus is a Yes interval. For the first case, that of overlapping intervals, it corresponds to the setup of Proposition \ref{pr:inter}. Thus, that proposition applies and we have that $c:b$ is a Yes interval. It can be a singleton. 
\end{proof}

\begin{corollary}\label{cor:finite-inter}
    A finite collection of Yes intervals will have a non-empty intersection which is a Yes interval. 
\end{corollary}

\begin{proof}
The previous Corollary establishes it for pairs. We can thus reduce any finite collection. Indeed, given a collection $\{A_i\}_{i=1}^n$ of Yes intervals, we can define $B_1 = A_1$, $B_i = B_{i-1} \cap A_i$ for $i > 1$. If $B_{i-1}$ is a Yes interval, then $B_i$ is a Yes interval by the Corollary. $B_n$ will then be the intersection of all of the $A_i$ and is a Yes interval.  
\end{proof}

For infinite collections of intervals, the intersection can be empty. It can also fail to be an inclusive rational interval. For examples and related discussions, see Section \ref{sec:ni}. 

We can thus see that one method of demonstrating that $R(a:b)=0$ for some given interval $a:b$ is to produce an interval $c:d$ disjoint from $a:b$ such that $R(c:d)=1$. On the other hand, if all neighborly $R$-Yes intervals intersect $a:b$, but none of them are strictly contained in $a:b$, then either $a$ or $b$ is contained in all such intervals, but not both by being Rooted. Let's say $a$ is. Then $R(a:a)=1$ by the Closed property and thus $R(a:b)=1$ as well. This situation is where the usual computational difficulties arise and we often have to resort to the notion of compatibility as described in the next section. 

We will also find the following couple of propositions to be useful. 

\begin{proposition}\label{pr:subinter}
Let $R$ be a neighborly oracle. Then given an $R$-Yes interval $a\lt b$, there exists an interval $c:d$ strictly contained in $a:b$ which is also an $R$-Yes interval. 
\end{proposition}

\begin{proof}
Since $R$ is a neighborly oracle but also satisfies the closed property, there exist $R$-Yes intervals $e:f$ and $g:h$ such that $a$ is not contained in $e:f$ and $b$ is not contained in $g:h$. 

Since $e:f$ and $g:h$ are Yes-intervals, their intersection is non-empty and also a Yes interval. Let's call that interval $c\lte d$.  Notice that the interval does not contain $a$ and does not contain $b$. It is an $R$-Yes interval and thus must have a non-empty intersection with $a:b$. Let $q$ be a rational in $a\lte b$ and $c \lte d$. If $c\leq a$, then we would have $c \leq a \leq q \leq d$ which would imply $a$ was in $c \lte d$. Since that is false, we must have $a < c$. Similarly, $d < b$. Thus, $c:d$ must be strictly contained in the interval, as was to be shown. 
\end{proof}

\begin{proposition}\label{pr:multi}
Given an oracle $R$ and an $R$-Yes interval $a\lt b$, any finite rational partition of $a:b$ will yield exactly one $R$-Yes interval if none of the partition numbers form a Yes singleton. 
\end{proposition}

This is repeated application of the Separation property. 

\begin{proof}
Let $a < c_1 < c_2 < c_3 < \cdots < c_n < b$ be a given rational partition of the interval $a:b$ where all of the $c_i$ are rationals and $R(c_i:c_i) = 0$ along with $R(a:a) = 0 = R(b:b)$. We are given $R(a:b) = 1$. By Separation, $R(a:c_1) \neq R(c_1:b)$. If $R(a:c_1)=1$, then we are done as $R(c_1:b)=0$ and also all sub-intervals are 0 as well since No propagates downwards. 

Let us then assume $R(c_1:b)=1$. We can then repeat. Assume $R(c_i:b)=1$. Then $R(c_i:c_{i+1})=1$ or $R(c_{i+1}:b)=1$. If the former, we are done, having shown all the intervals previously to be 0 as well as all the intervals later to be 0.  If the latter continues to hold as we iterate, then we continue until $i+1 = n$. At that point, we have shown all prior intervals to be 0. Thus, $R(a:c_n)=0$ as $a:c_{n}$ is a union of overlapping No intervals. By separation, we must have $R(c_n:b) = 1$. 

We have shown what was desired. 
\end{proof}

\begin{corollary}
    For a neighborly oracle $R$ and an $R$-Yes interval $a\lt b$, any finite rational partition of $a:b$ will yield exactly one $R$-Yes interval. 
\end{corollary}

\begin{proof}
    There are no Yes singletons so the proposition automatically applies. 
\end{proof}

For a singleton $R$, there is not much to say other than that any interval containing the singleton will be a Yes interval. So partitions would lead to either one Yes interval if the singleton is not on the boundary of the internal partitions or lead to two Yes intervals if it is on the boundary between two of them. 

\begin{proposition}\label{pr:no-is-disjoint}
    If an interval $a:b$ is a No interval for an oracle $R$, then there exists an $R$-Yes interval disjoint from $a:b$.
\end{proposition}

\begin{proof}
    Since $R(a:b) = 0$, we have $R(a:a) = 0$ and $R(b:b) = 0$. By Existence, there is a Yes interval and by being closed, there are Yes intervals $c:d$ and $e:f$ that do not contain $a$ and $b$, respectively. Let $g:h$ be the intersection of $c:d$ and $e:f$. This is a Yes interval. It does not contain $a$ or $b$ as neither of those are common to both $c:d$ and $e:f$. Thus, $g:h$ is either strictly contained in $a:b$ or disjoint from it. If it was contained in $a:b$, then $a:b$ would be a Yes interval by Consistency. But $a:b$ is not. Thus, $g:h$ is a Yes interval disjoint from $a:b$.
\end{proof}

\begin{corollary}\label{cor:exclude-singleton}
    If $a : b$ is a Yes interval and $a:a$ is a No singleton, then there exists a number $c$ in $a:b$, not equal to $a$ or $b$, such that $c : b$ is a Yes interval.
\end{corollary}

\begin{proof}
    By Proposition \ref{pr:no-is-disjoint}, there exists a Yes interval $e : f$ disjoint from $a:a$. By Corollary \ref{cor:pair-inter}, we have the intersection of $a\lte b$ with $e\lte f$, say $c:d$, is a Yes interval. Since $a$ is not in $e\lte f$, we have that $a$ is not in $c:d$ and therefore not equal to $c$.  If $c$ is taken to be the number in $c:d$ closest to $a$, then we have $a:c:d:b$ and since $c:b$ contains $c:d$, a Yes interval, $c:b$ must be a Yes interval.  If $c$ happens to be $b$, then $b:b$ is a Yes singleton. We redefine $c$ to be any rational between $a$ and $b$, say the average of the two. Since that interval will contain $b:b$, it is a Yes interval by Consistency and we have produced the required Yes interval. 
\end{proof}

\subsection{Equality and Ordering}

Two oracles are equal if they agree on all inclusive rational intervals. To prove inequality, it is therefore sufficient to find a disagreement. The basic tool is to find two disjoint intervals that we can guarantee different results on. 

If we have two different ``real numbers'', do we get different oracles for them? Since we are defining real numbers, this gets a little tricky, but we can imagine that what we mean by knowing the two numbers are different is that they are separated by a rational number. Let's assume $a < r < p < s < b$ where $r$ and $s$ are our two distinct real numbers,  $p$ is a known rational that separates them, and $a$ and $b$ are two rationals we know sandwich the two reals. Then $a:p$ is a Yes-interval for the Oracle of $r$ while $p:b$ is a Yes-interval for the Oracle of $s$, and each of those intervals are No-intervals for the other oracle. 

For example, if we we want to distinguish the square roots of 2 and 3 from one another,\footnote{We will discuss square roots later; just assume the usual ordering properties for now.} then $a = 1$, $p = \tfrac{3}{2}$, and $b = 2$ would suffice as $1 < \sqrt{2} < \tfrac{3}{2} < \sqrt{3} < 2$. Let $R$ be the Oracle of $\sqrt{2}$ and $S$ be the Oracle of $\sqrt{3}$, then we should have $R(1:\tfrac{3}{2}) = 1$, $R(\tfrac{3}{2}:2) = 0$,  $S(\tfrac{3}{2}:2) = 1$, $S(1, \tfrac{3}{2}) = 0$. Thus, these oracles are different. In particular, from this presentation, it would be reasonable to say $R < S$ and this is in line with our definition below.

If the oracles $R$ and $S$ cannot find an interval on which they disagree, but we also cannot prove that they agree on all intervals, then we can talk about compatibility and the resolution of that compatibility.

Let $R$ and $S$ be two oracles. We adopt the following relational definitions:

\begin{enumerate}
    \item $R=S$. $R$ and $S$ are equal if and only if they agree on all intervals. 
    \item $R < S$. $R$ is less than $S$ if and only if there exists $a:b < c:d$ such that $R(a:b) =1 = S(c:d)$. As $a:b$ and $c:d$ are disjoint, necessarily $R(c:d) = 0 = S(a:b)$. 
    \item $R > S$. $R$ is greater than $S$ if and only if there exists $a:b > c:d$ such that $R(a:b) =1 = S(c:d)$. As $a:b$ and $c:d$ are disjoint, necessarily $R(c:d) = 0 = S(a:b)$. This is equivalent to the statement that $R$ is greater than $S$ if and only if $S < R$. 
    \item $a:R?S:b$. $R$ and $S$ are $a:b$ compatible if $R(a:b)=S(a:b) = 1$.
    \item $R ? S$. $R$ and $S$ are compatible if and only if $R(a:b) = S(a:b)$ for all tested intervals $a:b$. 
\end{enumerate}

The first four items are statements of mathematical fact. The fifth item reflects the state of human knowledge, perhaps limited by mathematical unknowns. It is useful to have a way of expressing this indeterminate state. 

The interval $a:b$ is the \textbf{resolution of the compatibility of $R$ and $S$} if it is the shortest known interval which is both an $R$-Yes and an $S$-Yes interval. This is necessarily the intersection of all known mutually Yes intervals for $R$ and $S$.  If $R(a:b) = 1$, $S(c:d) = 1$, and $ a \leq c \leq b \leq d$ then the resolution of the compatibility of $R$ and $S$ is contained in $c:b$. In the constructive texts, they formulate the $\varepsilon$-Trichotomy proposition which states that two real numbers $\alpha$ and $\beta$ satisfy either $\alpha < \beta$, $\beta < \alpha$, or $|\alpha - \beta| < \varepsilon$ for all $\varepsilon > 0$. Their assertion is that this is what can be proven constructively, i.e., with enough resources, one can establish this to be true for two given real numbers and an $\varepsilon$. See \cite{bridger} page 57. Even that result seems ambitious. 

We now state and prove a couple of basic statements required for the definition above to be correct. 

\begin{proposition}[Well-Defined]\label{pr:wd}
Both equality and inequalities are well-defined and exclusive to one another.  
\end{proposition}

\begin{proof}
Equality is defined comprehensively across all intervals, leaving no room for differences. 

Let us prove that if $R < S$ then we cannot have $R=S$ or have $R > S$. For both of these, we use Proposition \ref{pr:disjoint} which states that disjoint intervals cannot both be Yes-intervals. 

Let us assume that $R<S$ with the implied Yes intervals $a:b < c:d$ where $R(a:b)=1 =S(c:d)$.

To show that they are not equal, we need to demonstrate that the rules give a different result on an interval. By the disjointness of these intervals, we know that $R(c:d) = 0 = S(a:b)$. $R$ and $S$ therefore differ on both $c:d$ and $a:b$ which establishes that they are not equal as oracles.  

Now we look to prove that $R>S$ cannot also be true. If $R > S$, then we have intervals $e:f > g:h$ such that $R(e:f) = 1 = S(g:h)$. We have the existence of $p$ in common to both $a:b$ and $e:f$ as well as $q$ in common to both $c:d$ and $g:h$. Since $e:f > g:h$, we have $p > q$. Since $a:b < c:d$, we have $p < q$. This is a contradiction and thus $R>S$ cannot hold true if $R<S$ holds trues.
\end{proof}

\begin{proposition}
    Given oracles $\alpha$ and $\beta$,  $\alpha$-Yes interval $a\lte b$, and $\beta$-Yes interval $c\lte d$, we have $\alpha ? \beta$ based on the union of the two intervals.  
\end{proposition}

\begin{proof}
Let $e= \min(a,c)$ and $f = \max(b,d)$. Then, we have $e:a:b:f$ and $e:c:d:f$. Since $e:f$ contains both $a:b$ and $c:d$, it is a Yes interval for both $\alpha$ and $\beta$. We therefore have $e:\alpha?\beta:f$.
\end{proof}

\begin{proposition}[Known Trichotomy]
One of the following three relational properties must hold for any given pair of oracles $\alpha$ and $\beta$ based on $\alpha$-Yes interval $a \lte b$ and $\beta$-Yes interval $c \lte d$: $\alpha < \beta$,  $\alpha > \beta$, or $\alpha ? \beta$. 
\end{proposition}

\begin{proof}
     If $b < c$, then $a:b < c:d$ and $\alpha < \beta$. If $d < a$, then $c:d < a:b$ and $\beta < \alpha$. 
     
     We can therefore assume $b \geq c$ and $d \geq a$ in what follows and, in particular, there is a non-empty intersection $e\lte f$ of $a:b$ and $c:d$.  If $e \lte f$ is a Yes interval for both, then $e:\alpha ? \beta:f$. 
     
     If not, then at least one of them has it be a No interval. Let's say $e:f$ is a No interval for $\alpha$ but it is a Yes interval for $\beta$. Because it is a No interval for $\alpha$, both $e:e$ and $f:f$ are No singletons for $\alpha$. We have that $a:e:f:b$ is a partition and by Proposition \ref{pr:multi}, either $a:e$ or $f:b$ is a $\alpha$-Yes interval. If $a:e$ is Yes, then $\alpha < \beta$. If $f:b$ is Yes, then $\beta < \alpha$. Both of those follow from using Corollary \ref{cor:exclude-singleton}. 
     
     Finally, if $e:f$ is a No interval for both, then the two intervals must be strictly overlapping (one does not contain the other) and we have two cases, depending on the ordering of $c$ and $a$. The first is $a:c:b:d$ which would imply $a:c$ is $\alpha$-Yes and $b:d$ is $\beta$-Yes implying $\alpha < \beta$. Or we have $c:a:d:b$ with $c:a$ being $\beta$-Yes and $d:b$ being $\alpha$ Yes which implies $\beta < \alpha$. 
     
\end{proof}

From the perspective of a single comparison of given Yes-intervals, the only way equality could be established is if those intervals were equal singletons. To establish equality, particularly for neighborly oracles, one must have an argument that somehow applies to an infinite number of intervals. We have some statements to that effect in Section \ref{sec:eq}. 

\begin{proposition}[Transitive Law]\label{pr:transitive}
Let $R$, $S$, and $T$, be oracles that satisfy $R<S$ and $S < T$. Then $R < T$.
\end{proposition}

\begin{proof}
By the assumptions, we have $a:b < c:d$ where $R(a:b) = 1 = S(c:d)$ and $R(c:d) = 0 = S(a:b)$. We also have $e:f < g:h$ where $S(e:f) = 1 = T(g:h)$ and $S(g:h) = 0 = T(e:f)$. We need to show $a:b < g:h$.

Let $m:n$ be the intersection of $c:d$ with $e:f$. This exists since disjoint intervals cannot both be Yes intervals for the same oracle (Proposition \ref{pr:disjoint}). Note that $S(m:n) = 1$ since the intersection of Yes-intervals is again a Yes-interval (Proposition \ref{pr:inter}). Since $m:n$ is contained in $c:d$, it satisfies $a:b < m:n$. Similarly, $m:n$ is contained in $e:f$ which gives us $m:n < g:h$. Since the inequality of intervals is transitive, we have $a:b < g:h$ as was to be shown. 
\end{proof}

We can also add in definitions of the customary real intervals. While we could rely on the customary definitions based on the ordering relations, we will cast it in the language of oracles. In what follows, we assume $\alpha < \beta$ are oracles,  $a\lte b$ represents an almost generic $\alpha$-Yes interval, $c\lte d$ represents an almost generic $\beta$-Yes interval, and that the not-quite generic part are the requirements that $b < c$ and none of the endpoints are roots of $\alpha$ or $\beta$. Then: 

\begin{enumerate}
\item The oracle $\gamma$ is in the closed interval $[\alpha, \beta]$ if $a \lte  d$ is a $\gamma$-Yes interval for all such $\alpha$ and $\beta$-Yes intervals. 
\item The oracle $\gamma$ is in the open interval $(\alpha, \beta)$ if $b\lte c$ is a $\gamma$-Yes interval for some such $\alpha$ and $\beta$-Yes intervals. 
\item The oracle $\gamma$ is in $[\alpha, \beta)$ if $a\lte c$ is a $\gamma$-Yes interval for all such $\alpha$-Yes intervals and for some such $\beta$-Yes interval.
\item The oracle $\gamma$ is in $(\alpha, \beta]$ if $b\lte d$ is a $\gamma$-Yes interval for all such $\beta$-Yes intervals and for some such $\alpha$-Yes interval.
\item The oracle $\gamma$ is in the open interval $(\alpha, \infty)$ if $b$ is the lower bound for a $\gamma$-Yes interval for some $\alpha$-Yes interval. 
\item The oracle $\gamma$ is in the open interval $(-\infty, \alpha)$ if $a$ is the upper bound for a $\gamma$-Yes interval for some $\alpha$-Yes interval. 
\item The oracle $\gamma$ is in the half-closed interval $[\alpha, \infty)$ if for all $\alpha$-Yes intervals, $a$ is the lower bound for a $\gamma$-Yes interval. 
\item The oracle $\gamma$ is in the open interval $(-\infty, \alpha]$ if for all $\alpha$-Yes intervals, $b$ is the upper bound for a $\gamma$-Yes interval. 
\end{enumerate}

\subsubsection{Equality}\label{sec:eq}

\begin{proposition}\label{pr:reflexive}
The equality relation is reflexive, symmetric, and transitive. 
\end{proposition}

\begin{proof}
This is immediate from the properties of equality of natural numbers, particularly the numbers 0 and 1. Indeed, the statements of these are the proofs:
\begin{itemize}
    \item Reflexive: $R(a:b)=R(a:b)$ for all intervals $a:b$ and rules $R$.
    \item Symmetric: For all intervals $a:b$ and rules $R$, $S$, $R(a:b)=S(a:b)$ if and only if $S(a:b) = R(a:b)$ 
    \item Transitive: If $R(a:b)=S(a:b)$ and $S(a:b) = T(a:b)$ then $R(a:b)=T(a:b)$. This holds for all intervals $a:b$ and rules $R$, $S$, $T$.
\end{itemize}
\end{proof}

If we have $a:R?S:b$, then they also agree (Yes) on all intervals that contain $a:b$ as well as all intervals disjoint from $a:b$ (No). 

Thus, two oracles could be equal, unequal, or compatible. The latter is for rules that we cannot find an interval of disagreement, but we have not been able to establish equality. 

We next prove that two oracles whose Yes intervals always overlap are, in fact, the same oracle.

\begin{proposition}\label{pr:overlap}
Let $R$ and $S$ be two oracles such that whenever $R(a:b)=1$ and $S(c:d)=1$, we have the existence of $e:f$ such that $e:f$ is contained in $a:b$ and $c:d$.  Then $R =S$.
\end{proposition}

Note we do not assume $e:f$ is a Yes-interval for either. It would be trivial by consistency if we had this. 

\begin{proof}
To start, let us look at the case of $R$ being Yes on a singleton, say $R(a:a) = 1$. Then for every $S$-Yes interval $c:d$, we have $e:f$ contained in $a:a$ and $c:d$. But that means $a=e=f$ and $a$ is contained in $c:d$. Since $S$ is Closed, we must have $S(a:a)=1$. By being Rooted and Consistent, $S$ and $R$ must agree on all intervals and are equal. 

We can now proceed with the assumption that both are neighborly oracles which means every rational in a Yes interval will separate the intervals. 

Let us take $R(a \lte b) = 1$ and $S(c \lte d) = 1$. Define $e$ and $f$ such that $e \lte f$ is the intersection of $a:b$ and $c:d$; this is non-empty by the assumed condition of the statement. We will prove from this that $R(e:f)=1 = S(e:f)$

By $e:f$ being contained in both intervals, we have $a,c \lte e \lte f \lte b,d$.

By not being singletons, Proposition \ref{pr:multi} tells us that exactly one of $a:e$, $e:f$, or $f:b$ is an $R$-Yes interval. Similarly, for $S$, exactly one of $c:e$, $e:f$, and $f:d$ is a $S$-Yes interval. 

Let's assume $R(a \lte e) = 1$. By assumption of this being a neighborly oracle, we we can apply Proposition \ref{pr:subinter} to obtain a sub-interval of $a \lte e$, say $m:n$, such that $R(m:n) = 1$. This is strictly less than $e \lte f$ which means it will not intersect $c \lte d$ as $e \lte f$ was the mutual intersection of $a:b$ and $c:d$. But the hypothesis is that $R$-Yes intervals should intersect the $S$-Yes interval $c \lte d$. Thus, $R(a \lte e) = 1$ is contradictory and we conclude $R(a \lte e) = 0$. Similarly, $R(f:b) = 0$, $S(c:e) = 0$, and $S(f:d) = 0$. This leaves $R(e:f) = 1 = S(e:f)$.

 Now that we have $e:f$ is a Yes-interval for both, by Consistency we conclude $R(c:d) = 1 = S(a:b)$. As $c:d$ and $a:b$ were arbitrary Yes intervals for the two oracles, we have established equality across all Yes intervals and thus the two oracles are equal. 

\end{proof}

The next statement is used in establishing the distributive rule for oracles. 

\begin{corollary}
    If $R$ and $S$ are two oracles such that whenever $R(a:b) = 1$, there exists an interval $c:d$ contained in $a:b$ such that $S(c:d) = 1$. Then $R=S$.
\end{corollary}

\begin{proof}
Let $a:b$ be any $R$-Yes interval and $e:f$ be any $S$-Yes interval. To use the proposition above, we need to establish that $e:f$ and $a:b$ intersect. By assumption, there exists an $S$-Yes interval $c:d$ contained in $a:b$. Since both $c:d$ and $e:f$ are $S$-Yes intervals, they have a non-empty intersection. This intersection must be contained in $a:b$ since $c:d$ is. Therefore, $e:f$ and $a:b$ have a non-empty intersection. As these were arbitrary, the proposition above allows us to conclude $R=S$.
\end{proof}

\subsection{Bisection Approximation}

This section concerns being able to narrow in on the value of an oracle. 

\begin{proposition}\label{pr:short}
For any rule $R$ that satisfies the Existence and Separation properties, we can produce a Yes interval shorter than any given positive rational number. In particular, this holds for oracles. 
\end{proposition}

\begin{proof}
The existence property of $R$ gives us a starting interval, say $a:b$. Let $L = |b-a|$ be its length. Then we take the average of the endpoints: $c = \frac{a+b}{2}$. Since $R$ is Separating, we can use it to determine whether $R(a:c) = 1$ or $R(c:b) = 1$ or $R(c:c) = 1$. If the singleton is a Yes, then we are done as the singleton has length 0 and is therefore shorter than any positive number. For an oracle, that would mean $r$ is the Oracle of $c$. Otherwise, we have that $R(a:c) \neq R(c:b)$ implying that one is Yes and the other is No. We can now repeat with the new Yes interval; the length of that chosen interval is $\frac{L}{2}$. 

If we do this $n$ times, then the length will be $\frac{L}{2^n}$. 

Given an $\varepsilon >0 $, we have that $n > \log_2 (\frac{L}{\varepsilon})$ will be sufficient to ensure that the final length is less than $\varepsilon$. 
    \end{proof}

This is helpful in establishing the arithmetic properties. In Section \ref{sec:mediant}, we will discuss the mediant approximation which is generally a pretty pleasant computational method to employ with a nice relationship to continued fractions. 

Note that the bisection method will not usually produce the singleton if $r$ is a singleton. The mediant approximation does. 

\subsection{Two Point Separation}

The Consistency, Existence, and Closed Properties are basically background properties, often flowing immediately from the definition of the oracle. Both Consistency and Closed are often just added in. Their reason for being in the properties is to ensure the uniqueness of the oracle via maximality. Consistency is relatively trivial and mostly uninteresting, but the Closed property is, in many ways, at the heart of the promise and difficulties of real numbers, in whatever way they are defined. 

The two properties that are more negotiable are the Separation and Rooted properties. They are very much a pair that could be replaced with something else. I desired to highlight the Separation property because of the bisection and mediant approximation schemes. It seems that a good way of approximating is to pick a next number and then replace one of the endpoints with it. The way I wrote the Separation property, however, does not guarantee by itself that the procedure will work. It also has to be the case that if it gets stuck on a rational that that is the end of the process and the oracle is the oracle of that rational number. For that, we need the rooted property. 

One equivalent version, which we will talk about later, is requiring Yes-intervals to always intersect and that, given a Yes-interval, we can always find a Yes subinterval of it whose length is less than any given length. This is what the family of overlapping, notionally shrinking intervals covers. See Section \ref{sec:ni}.

Another equivalent version is to demand the ability to separate any two given points. Namely, the \textbf{Two Point Separation Property} is defined as, for a rule $R$: Given an $R$-Yes interval $a:b$ and two rationals $c,d$ in $a:b$, we can find an $R$-Yes interval $e:f$ contained in $a:b$ such that at least one of the two rationals is not in $e:f$. Furthermore, we also require the \textbf{Disjointness Property}: Any two disjoint intervals cannot both be Yes intervals. 

We need to assume disjointness as No does not propagate upwards from the background properties, unlike Yes. For example, consider the rule which is Yes if an interval $a:b$ satisfies either $a^2:2:b^2$ or $a^2:3:b^2$. That is, they ought to contain either $\sqrt{2}$ or $\sqrt{3}$. Consistency and Existence are immediately satisfied. Closed follows since there is no rational number common to all such intervals. The Two Point Separation Property is satisfied since given any rational number, we can find a Yes interval that does not contain it. Yet we want this to fail as we want it to represent a unique answer. This example is ruled out only by the Disjointness property, by considering the intervals $1.4:1.5$ and $1.7:1.8$

In our definition of oracles, it is the Interval Separation Property which rules this out, particularly applying it to the interval $1.4:1.8$ and the separation point $1.5$ which is not a Yes singleton, but both $1.4:1.5$ and $1.5:1.8$ are Yes. 

\begin{proposition}
    Let $R$ be a rule that satisfies the Consistency, Existence, and Closed Properties. Then requiring the properties of Separating (Interval Separation) and Rooted is equivalent to requiring the properties of Two Point Separation and Disjointness. 
\end{proposition}

\begin{proof}
    Let us assume we have the Separating and Rooted properties holding. And let $a:b$ and $c < d$ be given as in the Two Point Separation Property. By the Separating property, we know that either $m =\frac{c+d}{2}$ is in all the intervals or that $R(a:m) \neq R(m:b)$. If it is the latter, then we know that either $c$ or $d$ is in a No-interval and we are done. If it is the former, than the rooted property tells us that both $c$ and $d$ are not Yes singletons and their own singleton intervals are the necessary No intervals required. We also know that any interval not containing $m$ is a No interval. 
    
    Proposition \ref{pr:disjoint} establishes the Disjointness Property. 

    For the other direction, the Rooted property follows immediately from the Disjointness Property since $c \neq d$ implies $c:c$ and $d:d$ are disjoint. 
    
    To establish the Separation property, we need to consider a sequence of numbers approaching the $c$ of that property and argue that either $c$ is not in one of those intervals or it is a Yes singleton. Let $a_0=a$, $a_{i+1} = \frac{a_i + c}{2}$, $b_0 = b$, and $b_{i+1} = \frac{b_i + c}{2}$. Then using the Two Point Separation Property on $a_i$, $c$, we have the existence of a Yes interval, say $e_i:f_i$, that does not contain at least one of them. Similarly, we also have intervals $g_i:h_i$ which does not contain at least one of $b_i$ or $c$. 
    
    If for some $i$,  we have that either $e_i:f_i$ or $g_i:h_i$ does not contain $c$, then we are done as that interval, let us call it $I$,  will be in $a:b$, but since $I$ does not contain $c$, it must be contained either in $a:c$ or $c:b$. By consistency, we can expand the Yes one to be fully one of those. The other one must be No since the intersection with $I$ is empty. 
    
    The other case is that $e_i:f_i$ and $g_i:h_i$ contains $c$ for every $i$. Then we claim that every Yes interval contains $c$. For let $m:n$ be an interval in $a:b$ that did not contain $c$. We have two cases, but, without loss of generality, we can assume that $m:n$ is contained in $a:c$.\footnote{The other case is that $m:n$ is contained in $b:c$. The argument changes by replacing $a$ with $b$, $e$ with $g$, and $f$ with $h$.} Let us assume that $n$ is closer to $c$ than $m$. Then define $p =\frac{n+c}{2}$. By the construction of the $a_i$, there exists $j$ such that $a_j$ is closer to $c$ than $p$ is. We then have $e_j:f_j$ contains $c$ but not $a_j$. This means that $e_j:f_j$ cannot intersect $m:n$. Since $e_j:f_j$ is a Yes interval, that means $m:n$ is a No interval. As a rough sketch of the relations of these numbers, we have $a:m:n:p:a_j:e_j:c:f_j$ where $a_j$ is definitively not in the interval $e_j:f_j$. The same argument, after relabeling, establishes that there can be no interval in $b:c$ not containing $c$ which can be a Yes interval. 

    Thus, $c$ is contained in every Yes-interval and, by the Closed property, $c:c$ is a Yes interval. 

    Between the two cases, we have established the Interval Separation Property. 
\end{proof}

The Two Point Separation Property generalizes in a straightforward way to other kinds of spaces. The Interval Separation property, on the other hand, feels inherently one dimensional. One might wonder why we focus on the Interval Separation. The reason is that it is more direct, applying directly to the intervals created by the questioner separating an interval. Two Point Separation, on the other hand, just asserts the existence of some interval that does not include both given numbers. It is weaker in its control and more of an abstract wish-making. In short, it is not constructive while the Interval Separation is more in that direction. It suffers, of course, from making an assertion across all Yes intervals which is not entirely constructive.

\subsection{No New Oracles from the Reals}

In our setup, we are trying to fill in the gaps between the rationals. While oracles can be used to compactify a space (see \cite{taylor23metric} ),  here we were using them to complete the rationals. The question arises, do we get new oracles using real number intervals? 

The answer ought to be, and is, no. Let $\bar{R}$ be a meta-real oracle that affirms Yes or No on closed intervals of real numbers, following the same rules as the real oracle definition we have given. The properties are the same except now we are using more general intervals. We need to show that every meta-real oracle is rooted and, therefore, representing a real oracle. Of crucial importance is that the existence property implies the existence of a closed (finite) interval $[\alpha, \beta]$ of real numbers which is Yes for the given meta-real oracle. 

We construct the root of a meta-oracle as follows. Define the real oracle rule $R$ as yielding a Yes exactly for rational inclusive intervals $a\lte b$ that satisfy: 1) $a$ being a lower endpoint of a Yes-interval for a lower endpoint real of an $\bar{R}$-Yes interval, e.g., a lower endpoint of the $\alpha$-Yes interval in the real interval above,  and 2) $b$ is an upper endpoint of a Yes-interval for an upper endpoint real of an $\bar{R}$-Yes interval, e.g., an upper endpoint of the $\beta$-Yes interval in the real interval above. We will establish two things: $R$ is a real oracle and $\bar{R}$ is the meta-real oracle version of $R$. That is, $R$ and $\bar{R}$ are essentially the same. If we have a rational $p$, we will denote $\bar{p}$ as the rational-based oracle version of that rational, i.e, the oracle that says yes to any rational interval that contains $p$. 

Note that $R(a:b) = 1$, with associated $[\alpha, \beta]$-$\bar{R}$ Yes interval, implies $\bar{R}([\bar{a}, \bar{b}]) = 1$ by Consistency since $[\bar{a}, \bar{b}]$ contains $[\alpha, \beta]$ as $\bar{a} \leq \alpha \leq \beta \leq \bar{b}$.

We first establish that $R$ is an oracle. It is a rule that for every rational interval does assign a Yes or No. 
\begin{itemize}
    \item Consistency. If $c\lte a \lte b \lte d$ and $a \lte b$ is an $R$-Yes interval, then there exists a pair $[\alpha, \beta]$ such that $a$ is a lower bound of an $\alpha$-Yes interval, $b$ is an upper bound of a $\beta$-Yes interval. Thus $c$ is a lower bound of an $\alpha$-Yes interval and $d$ is an upper bound of a $\beta$-Yes interval by consistency with regards to the $\alpha$ and $\beta$ oracles.
    \item Existence. The existence of a finite interval $[\alpha, \beta]$ and their existence of Yes intervals leads to our ability to make a rational inclusive interval as well. 
    \item Separation. Let $R(a\lte b) = 1$ and let $c \neq a, b$ in $a:b$ be given. Since $R(a:b)= 1$, we have $\bar{R}([\bar{a}, \bar{b}]) = 1$ and  thus, by separation, either $\bar{R}([\bar{c}, \bar{c}]) = 1$ or $\bar{R}([\bar{a}, \bar{c}]) \neq \bar{R}([\bar{c}, \bar{b}])$. In the first option, $c$ is an upper and lower bound of the oracle $\bar{c}$ and thus $R(c:c) =1$. For the second, the Yes interval has the rationals being their own upper or lower bound and hence the Yes carries to the unbarred version. For the no version, if it were Yes in the $R$ version, then the oracle version would contain an $\bar{R}$-Yes interval implying that it would be $\bar{R}$-Yes which it specifically was not. Thus, separating of $\bar{R}$ translates to $R$ having the separating property. 
    \item Rooted. Let $p$ be such that $R(p:p)=1$. Then there exists $[\alpha, \beta]$ such that $p$ is a lower bound for an $\alpha$-Yes interval and $p$ is an upper bound for a $\beta$-Yes interval. But since $\alpha \leq \beta$, we have $\bar{p} \leq \alpha \leq \beta \leq \bar{p}$ implying that they are all equal and that $[\bar{p}, \bar{p}]$ is an $\bar{R}$-Yes interval. Since $\bar{R}$ is a rooted oracle, there can be at most one such oracle and thus there can be no other rational $q$ such that $R(q:q) = 1$.
    \item Closed. If $q$ was present in all $R$-Yes intervals, then it is present in all $\bar{R}$-Yes intervals and hence the oracle version of it is a Yes and thus $q$ itself is both a lower bound and an upper bound of a Yes interval. Hence $q:q$ is an $R$-Yes interval.
\end{itemize}

The other part to establish is that $R$ is the root of $\bar{R}$. To establish this, we need to have that $R$ is in every $\bar{R}$-Yes interval. So let $[\alpha, \beta]$ be an $\bar{R}$-Yes interval. We need to find an $R$-Yes interval that is included in that interval. We can apply the bisection algorithm to either construct an $\bar{R}$-Yes interval that is contained in there or either $\alpha$ or $\beta$ is a root. If it is the latter, then we are done since $R$ will be that root as well given that the upper and lower bounds of the singleton interval will lead to the upper and lower bounds of $R$-Yes intervals being the root's Yes-intervals. Otherwise, we have a strictly smaller interval, say $[\gamma, \kappa]$ and we can then find a lower bound of $\gamma$-Yes and an upper bound of $\kappa$-Yes which is strictly above $\alpha$ and below $\beta$, respectively. Thus, that interval is in the $[\alpha, \beta]$ interval and so $R$ is as well. Since the interval was an arbitrary Yes interval, we have established our claim.

\section{Examples}

It is always good to have examples. In particular, how do we obtain various oracles in common situations? 

We shall start with how the rational numbers appear. We then define oracles for $n$-th roots,  numbers with more general approximation schemes, and least upper bounds of sets. We also investigate a couple of examples of indeterminate rules. For each of them, we will define the rule and then establish the properties by the definition. 

\subsection{Rational Oracles}\label{sec:rat-ora}

Given a rational $q$, we define the Oracle of $q$ as the rule $R(a:b) = 1$ if and only if $q$ is contained in $a:b$. This includes the singleton $q:q$.  We may call these rational oracles, rooted oracles, or singleton oracles. The number $q$ is the root of the oracle. 

We can verify the properties of the rational Oracle of $q$ as follows: 

\begin{enumerate}
    \item Consistency. If $R(a:b)=1$, then $q$ is contained in $a:b$. If $c:d$ contains $a:b$, then $q$ is contained in $c:d$. Thus, $R(c:d)=1$.
    \item Existence. $q$ is contained in $q:q$ so $R(q:q)=1$.
    \item Separating. If $R(a:b) =1$, then $q$ is contained in $a : b$. Let $c$ be strictly in $a:b$. We have three possibilities: 
    \begin{enumerate}
    \item $c=q$. Then $R(c:c) = R(q:q) = 1$.
    \item $c \neq q$, $a : q : c : b$. Then $a:c$ contains $q$ and $c:b$ does not. So $R(a:c)=1 \neq 0 =R(c:b)$.
    \item $c \neq q$, $a : c : q : b$. Then $c:b$ contains $q$ and $a:c$ does not. So $R(c:b)=1 \neq 0 =R(a:c)$.
    \end{enumerate}    
    \item Rooted. For $c \neq q$, $c:c$ does not contain $q$ and therefore $R(c:c)=0$.
    \item Closed. Assume $c$ is contained in all $R$-Yes intervals. Then, in particular, $c$ is in $q:q$ and thus $c=q$ and $R(c:c)=1$. That feels a little too reliant on the inclusion of the singleton. To make this a little more robust, assume $c \neq q$, say, $c < q$. Then  $a=c-1 < c < d=\dfrac{c+q}{2} < q < b=q+1$. Since $q$ is in $b:d$ but not in $a:d$, we have $R(a:d)=0$ and $R(b:d)=1$. Hence there is an $R$-Yes interval that does not include $c$.
\end{enumerate}

We will see with the arithmetic operations that these oracles are the natural representatives of the rational numbers, obeying the arithmetic that we would want them to obey.  

We also claim that if we have an oracle with rule $R$ such that there is a rational $q$ with $R(q:q)=1$, then it is the Oracle of $q$, whose rule we shall call $Q$. If the oracle is different, then there is an interval $a:b$ on which they disagree. Since $R(q:q) =1$, all the  $Q$-Yes intervals are also $R$-Yes intervals by Consistency applied to $R$. Therefore, we need to prove that for a given  $Q$-No interval $a:b$, we must also have it be an $R$-No interval. Because it is $Q$-No, it does not contain $q$. It is therefore disjoint from $q:q$. But by the disjoint property, Proposition \ref{pr:disjoint}, we have $R(a:b)=0$. Thus, the two oracles agree on all intervals and we have uniqueness. 

The property of being closed also prevents having an oracle which agrees with all the $Q$-Yes intervals except $q:q$. Closed forces $R(q:q)=1$ if all $R$-Yes intervals contain $q$.

Now that we have defined the rationals as oracles and defined inequality of oracles, we can prove that Yes intervals do contain their oracles.

\begin{proposition}\label{pr:yes-trap}
    If $a\lte b$ is a Yes interval for an oracle $r$, then $a \leq r \leq b$ as oracles.
\end{proposition}

\begin{proof}
    The statement to prove is equivalent to asserting that one of the three following statements holds true: $a:a$ is an $r$-Yes singleton,  $b:b$ is an $r$-Yes singleton, or $a:a < c:d < b:b$ for some $r$-Yes interval $c:d$.  If either of the first two is true, we are done. So we need to show that if $a:a$ and $b:b$ are both $r$-No singletons, then we have such an interval $c:d$. We use Corollary \ref{cor:exclude-singleton} twice, first to find a $c$ such that $a < c < b$ with $c:b$ being a Yes interval, and then again to find a $d$ such that $c \leq d < b$ with $c:d$ being a Yes interval. By construction, we have that $a < c \leq d < b$ and so $a:a < c:d < b:b$. This establishes that $a < r < b$ if $r \neq a, b$.  
\end{proof}

\subsection{Roots}\label{sec:roots}

For the positive $n$-th root of a positive rational number $q$, the Oracle rule would be $R(a\lte b) = 1$ if and only if $b> 0$ and $q$ is contained in $\max(a,0)^n:b^n$. If $a\geq 0$, then this is the statement $a^n:q:b^n$. From this definition, if we have $a:a$, then $R(a:a) = 1$ if and only if $a>0$ and $a^n = q$.

We will use the monotonicity of $x^n$ for positive $x$ which follows from the basic inequality fact of $ 0 < a < b$ implying $0 < a^n < b^n$.\footnote{This, in turn follows from the fact that $0<a<b$ and $0<c<d$ implies $0<ac<ad$, $0<ad<bd$ and, by transitivity, $0<ac<bd$. Then we setup an induction using the result that $0 < a<b$ and $0 < a^{n-1} < b^{n-1}$ implies $0 < a^n < b^n$. The core inequality fact underlying this is if $0<c$, then $a<b$ implies $ac < bc$. That is, multiplying by a positive number preserves the direction. This is equivalent to asserting that $0 < c (b-a)$ which follows from the product of positive numbers being positive.}

We can verify the properties of the Oracle of $\sqrt[n]{q}$ as follows: 

\begin{enumerate}
    \item Consistency. Because of the monotonicity of $x^n$ for positive $x$, consistency holds. Namely, assume $a\lte b$ is contained in $c \lte d$ and $R(a:b)=1$. If $c>0$, then $c^n \leq a^n$ by monotonicity,  $a^n \leq q \leq b^n$ by definition of $a:b$ being a Yes interval, and $b^n \leq d^n$ again by monotonicity. By transitivity, we have $c^n \leq q \leq d^n$ which is what we need for $R(c:d) = 1$.  If $c<0$ then we need to show $0 \leq q \leq d^n$. Since $0 \leq q$, and $q \leq  b^n$, and $b^n \leq d^n$, this holds. 
    \item Existence. Let $M = \max(q, 1)$. Then we claim $0 < q \leq M^n$. If $q \geq 1$, then $q^{n-1} \geq 1^{n-1} = 1$ and $M^n = q^n \geq q \geq 1 > 0$. If $ q < 1$, then $M=1$ and $0 < q < M^n = 1$. Either way, $R(0:M) = 1$. 
    \item Separation. Let $a: c: b$ be given such that $R(a \lte b)=1$. We need to show that either $R(c:c) = 1$ or $R(a:c) \neq R(c:b)$. We proceed by cases:
    \begin{enumerate}
        \item $c \leq 0$. Then $a<0$ and $q$ is contained in $0:b^n$. Thus, $R(c:b) = 1$ and $R(a:c) = 0$. 
        \item $c>0$, $c^n <q$. Then $c^n:q:b^n$ and $R(c:b)=1$. Since we have $\max(a, 0)^n :c^n:q$, we have $R(a:c)=0$. 
        \item $c>0$, $c^n = q$, then since $c^n : q : c^n$, we have $R(c:c) =1$.
        \item $c>0$, $c^n > q$. Then $a^n:q:c^n$ and $R(a:c) = 1$. We also have that $q:c^n:b^n$ implying $R(c:b)=0$.
    \end{enumerate}
     \item Rooted. This relies on the equation $x^n = q$ having at most one positive solution. This follows from monotonicity. 
    \item Closed. Consider a rational number $p$. We want to show that either $p>0$ with $p^n = q$ in which case $R(p:p)=1$ or, if not, then there exists an interval $a:b$ such that $R(a:b)=1$ but $p$ is not in $a:b$. Let $M = \max(q, 1)$; the 1 is needed since if $q<1$, then $q^n < q < 1$ and we want to make sure we have $M^n > q$.
    
    If $p < 0$, then the interval $0:M^n$ contains $q$ and does not include $p$ so $0:M$ is an $R$-Yes interval excluding $p$. Let us therefore assume $p \geq 0$ and that $p^n \neq q$. In the case that $p^n > q$, there is an $s$ such that $s<p$ and $s^n > q$ (well-known fact, but see Appendix \ref{app:A} Lemma \ref{app:greater}). Therefore, $R(0:s) = 1$ and $p$ is not in $0:s$. For the case of $p^n < q$, we have the existence of an $s$ such that $s > p$ and $s^n < q$ (Appendix \ref{app:A} Lemma \ref{app:lesser}). We thus have $R(s:M)=1$ and $p < s$ is not in the interval $s:M$.
    
\end{enumerate}

Is this oracle the $n$-th root of $q$?

While we have not done the arithmetic of oracles yet, the short version is that taking an oracle to the $n$-th power means the new oracle consists of intervals that are the result of applying the $n$-th power to the Yes-intervals of the original oracle. Since $q$ is in $a^n:b^n$ for every $0<a:b$ $\sqrt[n]{q}$-Yes interval, we have that $(\sqrt[n]{q})^n$ does equal the oracle $q$.\footnote{For $n$-th root Yes \-intervals of the form $a:0:b$, the $n$-th powering of that interval will still include $q$ even if $-a>b$. See Section \ref{containment}, item \ref{natpow}.} 

Finally, we can establish the ordering of square roots, namely that if $0 \leq p<q$, then $\sqrt[n]{p} < \sqrt[n]{q}$. We do this by narrowing the intervals sufficiently, using Proposition \ref{pr:short} so that the $n$-th power of the intervals are still disjoint. We then argue inequality based on the gap and translate it back down to oracle intervals via monotonicity. 

\begin{proposition}
    If $0 \leq p <q$, then $\sqrt[n]{p} < \sqrt[n]{q}$.
\end{proposition}

\begin{proof}
    Let $L = q-p > 0$. Also let $M = \max(1, q^n)$.  Since we have established that $n$-th roots are oracles, we can find intervals $a\lt b$ and $c \lt d$ such that $a^n:p:b^n$, $b-a < \frac{L}{3(n-1)M^{n-1}}$, $c^n:q:d^n$, $d-c < \frac{L}{3(n-1)M^{n-1}}$,  and $d, b < M$.  Then $b^n-a^n = (b-a) \sum_{i=0}^{n-1} b^i a^{n-1-i} < (b-a)(n-1) b^{n-1} < \frac{L}{3(n-1)M^{n-1}} (n-1) M^{n-1} = \frac{L}{3}$ using the fact that $b<M$ implies $b^{n-1} < M^{n-1}$. Similarly, $d^n-c^n < \frac{L}{3}$. This implies that $b^n < a^n + \frac{L}{3}  < p + \frac{L}{3}  = q - \frac{2 L }{3} < q -  \frac{L}{3} < d^n - \frac{L}{3} < c^n$. Since $b^n < c^n$, we have $b < c$ thanks to the monotonicity of $x^n$ for positive $x$. We therefore have $a:b < c:d$ and therefore their respective oracles satisfy this as well, namely $\sqrt[n]{p} < \sqrt[n]{q}$.
\end{proof}

\subsection{Intermediate Value Theorem}\label{sec:ivt}

A well-known process that is a canonical example for defining an oracle based on the separation property, is that of finding solutions to equations using the process of the Intermediate Value Theorem. 

In this section, we will assume we have a function $f$ defined on the rationals to the rationals. The paper \cite{taylor23funora} has a different approach to functions and also generalizes the Intermediate Value Theorem in that context. 

Let us say that we are trying to solve $f(\alpha) = y$, for rational $y$, on the interval $a:b$ and we have $f(a):y:f(b)$. Then we define a rule $R$ for intervals in $a:b$ such that $R(c:d) = 1$ exactly when $f(c):y:f(d)$ holds true. If we have $c:d$ partially outside $a:b$, then the oracle pronounces Yes or No based on the pronouncement on the intersection of $c:d$ with $a:b$. It gives a No if $c:d$ does not intersect $a:b$ at all. This choice will yield consistency. If $f$ is strictly monotonic for all rationals, then the definition simplifies to $R(c:d)=1$ exactly when $f(c):y:f(d)$ without having to worry about intersecting an interval $a:b$. We would still need to establish the existence of one such interval. 

For a random function $f$, this is not going to define an oracle. It will define a process that can certainly define an oracle. In particular, we can take the midpoint of the interval, make a choice of interval, take that midpoint, pick a new subinterval, etc. This is a sequence of narrowing, overlapping intervals and hence will define an oracle as explained in Section \ref{sec:ni}. But without more control on the function, different choices will potentially lead to different outcomes. 

To make it independent of choices, we can look at the case where $f$ is monotonic on $a:b$ and rationally continuous.\footnote{A function $f$ is rationally continuous if given $N >0$ and a rational $q$, we can find $M$ such that if $|q-r|<\frac{1}{M}$ for a given rational $r$, then $|f(q)-f(r)| < \frac{1}{N}$.} This will ensure that any subinterval in $a:b$ will give consistent results. In particular, it will satisfy

\begin{enumerate}
    \item Consistency. If we have the rational sandwiching $a:c:g:h:d:b$ and we have that $g:h$ is a Yes interval, then by monotonicity, we have $f(c):f(g):y:f(h):f(d)$ and so $c:d$ is a Yes interval. For intervals $c:d$ that are partially outside, $a:b$, the definition of following the pronouncements of the intersection is what yields consistency.
    \item Existence. $a:b$ is a Yes interval.
    \item Separating. Let $c:d$ be a Yes interval in $a:b$ and let $q$ be a rational satisfying $c:q:d$. We want to establish that $q$ separates the interval. Since $c:d$ is a Yes interval, we have $f(c):y:f(d)$. We compute $f(q)$ which by monotonicity satisfies $f(c):f(q):f(d)$ . If $f(q) = y$, then $f(q):y:f(q)$ and $q:q$ is a Yes singleton. Let us assume that is not the case. If $f(c):f(q):y:f(d)$, then $q:d$ is a Yes interval and, thanks to monotonicity, $y$ cannot be in $f(c):f(q)$ so $c:q$ is a No interval. If $f(c):y:f(q):f(d)$, then $c:q$ is a Yes interval and $q:d$ is a No interval by monotonicity. Thus, Separation holds. 
    \item Rooted. We need strictly monotonic here. Assume $q:q$ and $r:r$ are both Yes singletons. Then $f(r):y:f(r)$ and $f(q):y:f(q)$ holds true. This implies $f(q)=f(r)=y$. Strict monotonicity then implies that $q=r$. If we did not have the strictness, then the example of $f(x)=y$ for all $x$ is an example which satisfies all the oracle properties except for being rooted. 
    \item Closed. Let's assume that $q$ is such that $q$ is in every Yes-interval. We need to show that $q:q$ is a Yes interval which can be accomplished by showing that $L = |f(q) - y| = 0$ as this would imply $f(q):y:f(q)$ satisfying our definition of the Yes intervals. 
    
    To establish this, we use rational continuity. If $L \neq 0$, then we can find $N$ such that $\frac{1}{N} < \frac{L}{3}$. By continuity, we have an $M$ such that $|q-r|<\frac{1}{M}$ implies $|f(q)-f(r)| < \frac{1}{N} < \frac{L}{3}$. Since $q$ is contained in all Yes intervals and by the bisection Proposition \ref{pr:short}, we have a Yes interval $c:q:d$ contained in $q-\frac{1}{M}:q+\frac{1}{M}$, so that $f(c):y:f(d)$ holds true implying that $|f(c) - y| \leq |f(c)-f(d)|$. We then have $L = |f(q)-y| = |f(q) - f(c) + f(c) - y| \leq |f(q)-f(c)| + |f(c)-y| \leq |f(q)-f(c)| + |f(c) - f(d)| \leq |f(q) - f(c) | + |f(c) - f(q)| + |f(q) - f(d)| < \frac{3}{N} < L$. Since this is a contradiction, we must have $L=0$ implying $f(q) = y$ and thus $f(q):y:f(q)$.
\end{enumerate}

So under those conditions we have a well-defined oracle. Under more general conditions, we can describe a method which will produce an oracle, possibly using the notion of a fonsi from the next section, though it might not be unique.  

If we had a process to extend $f$ continuously to the oracles, then the Closed argument above could be adapted to demonstrate that $f(\alpha)=y$. One approach can be found in the function oracle  paper \cite{taylor23funora}.

We would then be on solid ground to define $\sqrt[n]{q}$ as the solution to $x^n = q$ with starting interval $1:q$ if we wanted to do so. Since $a^n : q: b^n$ is what the function process above yields, it is trivial to see that this agrees with the previous definition. This process can be used to define solutions to a variety of rational equations. 

More ambitiously, we can define $\pi$ as the zero of $\sin(x)$ on the interval $[3,4]$. We can define $e$ as the solution to  $\ln(x)=1$ on the interval $[2,3]$. These require some words about rational approximations to these transcendental functions, but our rational approximations to them can be made sufficiently fine to ensure we can get the information we need to make the choices, at least for a normal practical level of concern.  To get detailed approximations, we can use the bisection method or can we use the mediant method as detailed in Section \ref{sec:mediant} and even produce the continued fraction representation for these numbers.
 
Most of the real number constructions would be compatible with the sequence approach in working with general continuous functions, but they have less of a motivation to have the monotonicity conditions for uniqueness of the process. The one exception is the Dedekind cut approach which fits most easily with the monotonicity. We can define the cut as $\{x| (x<a) \vee (a \leq x \leq b \wedge f(x) < y) \}$. But this would not give any particular guidance in narrowing down to the solution nor would it be particularly helpful in the general situation. Oracles handle both points of view equally well. 

\subsection{Family of Overlapping, Notionally Shrinking Intervals} \label{sec:ni}

A common way of defining a real number is through some specification of intervals whose length is shrinking to 0. This is often in the guise of some core estimate along with some error bounds. We will establish that this is an oracle.

We define a \textbf{fonsi} to be a family of overlapping, notionally shrinking intervals meaning it is a set of rational intervals which are all pairwise intersecting and such that if we are given a positive rational length $q$, there exists at least one interval in the family whose length is less than $q$.\footnote{We use the term ``notionally shrinking'' to indicate that we want to think of it as shrinking intervals, but there need not be any sequential aspect to this that would qualify as shrinking. Rather, we just have the ability to find an interval whose length is at least as small as any non-zero length we care to specify.} Singletons are allowed in the family and they have length 0.\footnote{There can be at most one singleton as the singletons need to intersect, but we need not presuppose the uniqueness.} A set consisting of a single singleton does qualify for being a fonsi.

Given a fonsi, we define an oracle to be the rule that an interval is a Yes-interval if the interval contains an intersection of a finite number of elements of the fonsi or contains the intersection of all of the elements of the fonsi if that exists. This last one is how we include singletons for singleton oracles. 

One method of defining real numbers among constructivists is to equate a fonsi with a real number though they do not use that term. See \cite{bridger} and \cite{bridges} for details. The fonsi framework is appealing to the constructivist framework as it more aligns with their methods of proof and construction. It also aligns with measurements as we briefly discuss below. 

An easy category is that of nested intervals whose lengths are going to 0. We can also have a sequence of numbers with error bounds such that successive numbers are contained within the previously indicated error bounded intervals.  It is very common to have error bounds in applications and this is how we can easily establish them as oracles. 

One failed example is the set of intervals $\frac{1}{n}:1$. It is an overlapping set which creeps up to zero, but it fails to be shrinking. Another failed example is something like $\frac{1}{n}:\frac{1}{m}$, which does have intervals shrinking to zero, but it fails the overlapping property. A fonsi forces an overlap across the real number that it is representing. We could, for instance, look at $\frac{-1}{m}:\frac{1}{n}$ for all positive integers $m$ and $n$. These overlap and we can find as small an interval as we like. The fonsi $-\frac{p}{q} : \frac{r}{s}$ for all positive integers $p, q, r, s$ is nearly the oracle of 0. If we allow $p$ and $r$ to be 0, then the fonsi becomes the set of Yes intervals for the Oracle of 0. 

Before we establish that we can define an oracle from a fonsi, we need to establish that the pairwise intersection scales to intersections of arbitrary finite collections of intervals in a fonsi. 

\begin{proposition}
 Let $\mathcal{I}$ be a fonsi and $\{I_i\}_{i=1}^n$ be a finite collection of intervals in $\mathcal{I}$. Then the intersection $\bigcap_{i=1}^n I_i$ is non-empty and an interval.
\end{proposition}

\begin{proof} 
    Consider the intersection $\bigcap_{i=1}^n I_i$ where each $I_i$ is an interval in $\mathcal{I}$ and we have ordered them such that if $I_i = a_i : b_i$ then the lower bounds $a_i$ are ordered from least to greatest. That is,  $a_i \leq a_{i+1}$ for all $1 \leq i \leq n$. By the definition of a fonsi, $I_n$ intersects $I_i$ for each of the $i$ and for that to happen, we must have that $a_n \leq b_i$ for every $i$. That is, since the lower bound of $I_n$ is greater than or equal to the lower bound of $I_i$ and the intervals intersect, the lower bound of $I_n$ must be less than or equal to the upper bound of $I_i$. Therefore, $a_n$ is contained in $I_i$ for every $i$ and the intersection is non-empty. In fact, $a_n : \min(\{b_i\}_{i=1}^n)$ is the interval. Note that this can be a singleton. 
\end{proof}

\begin{corollary}
Let $A$ and $B$ be two finite intersections of intervals in the fonsi $\mathcal{I}$. Then the intersection of $A$ and $B$ is non-empty and an interval. 
\end{corollary}

\begin{proof}
    If $A$ is the result of the intersection of $I_i$ for $1 \leq i \leq n$ and $B$ is the result of the intersection of $I_i$ for $n+1 \leq i \leq m$ where $I_i$ are intervals in $\mathcal{I}$, then the intersection $\bigcap_{i=1}^m I_i$ is the same as the intersection of $A$ and $B$. By the proposition, the intersection of any finite collection of intervals of $\mathcal{I}$ is non-empty and, therefore, $A \cap B$ is non-empty and an interval. 
\end{proof}

Define a \textbf{community of intervals} in a fonsi to be  a collection of intervals such that their total intersection is non-empty. That is, a collection of intervals $\mathcal{A}$ contained in a fonsi are in community if $\bigcap_{\{I \in \mathcal{A}\}} I$ is non-empty. 

We need a little helper statement that describes a bit of the structure of a community.

\begin{lemma}
    Let $A$ be the intersection of a community of intervals $\mathcal{A}$ in a fonsi $\mathcal{I}$. Then given $a, b \in A$, we have that $a : b \subseteq A$.
\end{lemma}

\begin{proof}
    Let $c\lte d \in \mathcal{A}$ be given. We have that $c:d$ is an interval and that $a$ and $b$ are in $c:d$. Without loss of generality, we may assume $a < b$.  We then have $c \lte a \lte b \lte d$. Any rational number in $a:b$ is therefore also in $c:d$. Since $c:d$ was an arbitrary element in $\mathcal{A}$, we have that $a:b$ is contained in all elements of $\mathcal{A}$.
\end{proof}

\begin{proposition}\label{pr:fonsi-inf-inter}
Let $\mathcal{A}$ be a community of intervals of the fonsi $\mathcal{I}$. Then the intersection $A = \bigcap_{\{ I \in \mathcal{A}\}} I $ intersects with any element in the fonsi. 
\end{proposition}

Note that the intersection need not be a closed interval as the example of $\mathcal{A} = \{ -a_i : a_i \}$ demonstrates where $a_i$ is some sequence of upper bounds to the square root of $2$. This could then be made into a fonsi of 0 by by adding $0:0$:  $\mathcal{A} \cup \{0:0\}$. The intersection across $\mathcal{A}$ is $[-\sqrt{2}, \sqrt{2}]$, but since we are dealing only with rational intervals, this will not have the endpoints. This is a little pathological in the sense that such intersections are not the interest of the fonsi, but it can certainly be possible. 

Unfortunately, the intersection of a community of intervals need not intersect other intersections of a community of intervals. Consider the Yes intervals of the $\sqrt{2}$ Oracle. This is a fonsi. Consider the collection of Yes intervals of the form $1:b$ where $b$ is an upper endpoint of a Yes interval. Also consider the collection of intervals of the form $a:2$ where $a$ is a lower endpoint of a Yes interval. The intersection for the first collection is all rationals in the interval $[1, \sqrt{2})$ and the second collection is all rationals in the interval $(\sqrt{2}, 2]$. These clearly have no intersection. But with any particular element in the fonsi, that element has to surround the $\sqrt{2}$ and is a rational interval, hence it will intersect both of these intersections. The main issue with a community of intervals is that, while they overlap, they need not shrink. 

\begin{proof}
Let $c:d$ be an interval in $\mathcal{I}$. We need to establish that $c:d \cap A$ is non-empty. 

If $c=d$, then all intervals in $\mathcal{I}$ must intersect $c:c$ which is just $c$ and so $c$ is in all elements of $\mathcal{A}$ and is therefore in $A$.

Therefore, let's take $c < d$ without loss of generality. If $c$ or $d$ is in $A$, then we are done. If not, then there exists an interval $u\lte U$ in $\mathcal{A}$ which does not contain $c$ and an interval $v \lte V$ in $\mathcal{A}$ which does not contain $d$. Since $c:d$ is in the fonsi, $u:U$ and $v:V$ both intersect it. This implies $c : V : d$ and $c:u:d$.\footnote{ If $V > d$, then $v>d$ since $v:V$ was chosen so that $d$ is not in the interval. But since $c<d$, this implies $v:V$ and $c:d$ do not intersect, but these are both elements of the fonsi. So $V <d$. If $V < c$, then $v:V$ would again not intersect $c:d$. So $c:V:d$. Similarly, $c:u:d$. } Take $e:f$ to be the intersection of $u:U$ and $v:V$ which exists as they are elements of the fonsi. Then we know $c < u \leq e$ and $f \leq V < d$. That is, the intersection is contained in $c:d$. But $A$ is contained in any intersection of elements of $\mathcal{A}$ and so $A$ is contained in $e:f$ which is contained in $c:d$. This implies that there is a non-empty intersection with $c:d$ as long as the assumption of  $A$ being non-empty holds. 

\end{proof}

\begin{proposition}\label{pr:fon-unique-inter}
    For a fonsi $\mathcal I$, the intersection of all elements $\bigcap_{I \in \mathcal{I}} I$ is either empty or consists of a single rational number.
\end{proposition}

\begin{proof}
    We need to rule out the possibility of at least two rational numbers in the intersection. So let $p$ and $q$ be rationals in the intersection. Then define $L = |p - q|$. By the notionally shrinking property, let $a\lte b$ be an interval in the fonsi such that $b-a < \frac{L}{2}$. If $L \neq 0$, then we can do this as this is a fonsi. But then $p$ and $q$ cannot both be in the interval $a:b$ as their separation is greater than the length of the interval. Thus, $L=0$ and they are the same rational number. 
\end{proof}

This suggests that we can categorize fonsis in a similar fashion as we did to oracles. We call a fonsi a \textbf{rooted fonsi} if the intersection of all of its elements consists of a single rational number and we call that number its root. If it is not rooted, we call it a \textbf{neighborly fonsi}. 

\begin{proposition}\label{pr:fon-oracle-exists}
Given a fonsi $\mathcal{I}$, there exists a unique oracle $r$ such that all of the elements of the fonsi are Yes intervals of $r$.
\end{proposition}

\begin{proof}
If the fonsi is a rooted oracle with root $p$, then the oracle is the Oracle of $p$, namely, it is the oracle such that the Yes intervals are precisely those that contain $p$. Can we have an oracle whose Yes intervals include the elements of this rooted fonsi, but has a Yes interval which does not contain $p$? No as we will now show. 

Assume that we have an interval $a:b$ in the oracle which does not contain $p$. Let $a$ be the closest endpoint to $p$. Then let $L = |p-a|$. Take an interval $c:d$ in the fonsi whose length is less than $\frac{L}{2}$. Since $p$ must be in the interval, we have that $c:d$ is contained in $p-\frac{L}{2}: p + \frac{L}{2}$. This implies that $c:d$ does not intersect $a:b$. But we know from Proposition \ref{pr:inter} that all Yes intervals must intersect. Thus, there exists an element of the fonsi which is a No interval for any oracle which has a Yes interval that does not contain $p$. Hence, we have uniqueness for the oracle associated with a rooted fonsi. 

The other case is that of a neighborly fonsi. 

We define the Oracle of the fonsi $\mathcal{I}$ as $R(a :b) = 1$ if and only if $a : b$ contains the intersection of a community of intervals in $\mathcal{I}$. This will include all elements of the fonsi. 

It will be helpful to rule out any singleton being a Yes interval. If $R(c:c)=1$, this would mean that $c:c$ is the intersection of a community of intervals in the fonsi since it is Yes exactly when the interval contains the such an intersection and there are no subintervals of a singleton. But by Proposition \ref{pr:fonsi-inf-inter}, $c:c$ must intersect with every element of the fonsi. Since only $c$ is in the singleton, that means that $c$ is in every element of the fonsi. But then $c$ is the root of the fonsi and we are explicitly dealing with a neighborly fonsi. Thus, there is no Yes singleton.

\begin{enumerate}
    \item Consistency. Let $c:d$ contain $a:b$ and $R(a:b)=1$. Then, by definition, $a:b$ contains the intersection of elements of $\mathcal{I}$. Thus, $c:d$ contains that same intersection and therefore $R(c:d)=1$ by definition. 
    \item Existence. By definition of a fonsi, there exists an interval $a:b$ which has length, say,  less than 1. Thus, $R(a:b)=1$ by definition of the oracle.
    \item Separating. Assume $R(a:b)=1$ and let $c$ be strictly contained in $a:b$. We already established that $c:c$ is not a Yes interval. So we need to show that $R(a:c) \neq R(c:b)$. Because this is a neighborly fonsi, there exists an interval $B$ in $\mathcal{I}$ which does not contain $c$. 
    
    Let $A$ represent an intersection of elements of $\mathcal{I}$ which is contained in $a:b$; $A$ exists by definition of $a:b$ being a Yes-interval.   Consider $A \cap B$. By Proposition \ref{pr:fonsi-inf-inter}, this intersection is non-empty. $A \cap B$ will be contained in $a:b$ since $A$ is.  This intersection does not contain $c$ since $c$ is not in $B$. Thus, $A \cap B$ must either be contained in $a:c$ or $c:b$. Without loss of generality, we can take that to it to be in $a:c$. Then $R(a:c) = 1$ by definition. We need to show that $R(c:b) = 0$. Assume it was 1. Then there exists an intersection $C$ of some intervals in $\mathcal{I}$ which is contained in $c:b$. But by the proposition again, we have that $B \cap C$ is non-empty. Since $B$ intersects $a:c$ and does not contain $c$ while $C$ is contained in $c:b$, this is impossible. This contradiction leads us to conclude that $c:b$ does not contain an intersection of elements in the fonsi and, therefore, $R(c:b) = 0$.
    \item Rooted. There are no Yes singletons. 
    \item Closed. If $c$ is in every Yes interval, then it would be contained in all of the elements of the fonsi. But this would imply the fonsi is rooted and we are in the neighborly case. 
\end{enumerate}

This establishes existence of an oracle for neighborly fonsis and they are neighborly oracles. We now need to establish that any oracle that has all of the elements of the fonsi being Yes intervals must be this oracle. By Corollary \ref{cor:finite-inter}, we know that any Yes interval will intersect any intersection of a finite collection of Yes intervals. Thus, any given Yes interval in this oracle must intersect the elements of the fonsi and also their finite intersections. But this means every Yes interval in this other oracle must intersect the neighborly oracle defined above since every Yes interval in that one contains a finite intersection of fonsi elements. Thus, every Yes interval of both oracles must intersect. By Proposition \ref{pr:overlap}, they must be the same oracle. 
\end{proof}

Define a \textbf{maximal fonsi} as a fonsi such that every rational interval that contains the intersection of a community of intervals in a fonsi are contained in the fonsi. This implies the intersections themselves are in the maximal fonsi if they are a rational interval. If a rational is contained in  all elements of the fonsi, then it is in the maximal fonsi as it is in the common intersection of all the intervals. 

\begin{proposition} Let $r$ be an oracle and define $\mathcal{M}$ to be the set of all Yes-intervals of $r$. Then $\mathcal{M}$ is a maximal fonsi. 
\end{proposition}

\begin{proof}
     Proposition \ref{pr:short} gives us the notionally shrinking requirement. Corollary \ref{cor:pair-inter} tells us that two Yes-intervals must have non-empty intersection which is the overlapping requirement; that corollary further tells us that the intersection is actually in the fonsi. Thus $\mathcal{M}$ is a fonsi. 

     As for the maximality, let $A$ be a non-empty intersection of Yes intervals. Let $a:b$ be an interval that contains $A$. We need to show $a:b$ is a Yes interval. If $a:b$ is No, then by Proposition \ref{pr:no-is-disjoint} there is an interval $c:d$ such that it is Yes with $a:b$ being disjoint from $c:d$. But by Proposition \ref{pr:fonsi-inf-inter}, $c:d$ must intersect $A$. Since $a:b$ contains $A$, we have $c:d$ intersects $a:b$ contradicting the disjointness. 
\end{proof}

We have established that every fonsi $\mathcal{I}$ gives rise to an oracle which is equivalent to the maximal fonsi $\mathcal{M}$ containing the fonsi $\mathcal{I}$. Every oracle can be viewed as a maximal fonsi.

We can import the concepts of equality and inequality from our work with oracles. In particular, if there exists two intervals $I$ in $\mathcal{I}$ and $J$ in $\mathcal{J}$ such that $I < J$, then the oracles formed from these have the same relation as these intervals are Yes intervals. The oracles are equal if no separation exists. 

\begin{proposition}\label{pr:fonsi-inter}
Let $\mathcal{I}$ and $\mathcal{J}$ be two fonsis such that for every rational $q>0$ we have the existence of intervals $I \in \mathcal{I}$ and $J \in \mathcal{J}$ such that $|I| <q$, $|J| < q$, and $I$ and $J$ intersect. Then the oracle $r$ of $\mathcal{I}$ and the oracle $s$ of $\mathcal{J}$ are the same. 
\end{proposition}

\begin{proof}
To show two oracles are the same, we need to show that their Yes intervals intersect. Let $A=a\lte b$ be a Yes interval of $r$ and $B=c \lte d$ be a Yes interval of $s$. If the intervals overlap, we are done. If not, then there is a separation between the two intervals. To be explicit, let's say $a \leq b < c \leq d$. Let $t = c-b$ be the separation and $q = \frac{t}{3}$. Let $I= e\lt f$ and $J= g \lt h$ be the respective fonsi intervals that intersect and whose length is less than $q$;  they exist by the hypothesis. By the assumptions, we have rational numbers $m \in I \cap J$, $n \in I \cap A$, and $p \in J \cap B$. Thus,  $n \leq b < c \leq p$ implying $c-b \leq p - n  = p-m + m-n \leq q + q = t - q < t = c-b$ where the replacement by the $q$'s comes from the length of an interval being the greatest of all the differences of its members. The inequality is clearly a contradiction and therefore the separation $t$ must be zero and the two intervals must intersect. 
 
As any Yes intervals of $r$ and $s$ must intersect, Proposition \ref{pr:overlap} asserts the oracles are the same.  
\end{proof}

We have some useful corollaries from this. A \textbf{sub-fonsi} $\mathcal{J}$ of a fonsi $\mathcal{I}$ is a fonsi whose elements are also in $\mathcal{I}$, that is, $\mathcal{J} \subseteq \mathcal{I}$ and $\mathcal{J}$ is a fonsi in its own right.  

\begin{corollary}\label{cor:sub-fonsi}
    If a fonsi contains a sub-fonsi which consists of the Yes-intervals of an oracle, then the fonsi is associated with that oracle. 
\end{corollary}

\begin{proof}
    The oracle associated with the sub-fonsi has to be the Yes-interval oracle by Proposition \ref{pr:fon-oracle-exists} and by Proposition \ref{pr:overlap}. The fonsi and its sub-fonsi certainly meet the criteria for applying Proposition \ref{pr:fonsi-inter}. Thus, we can conclude that they are associated with the same oracle. 
\end{proof}

\begin{corollary}
    If the set union of two fonsis is a fonsi, then they are associated with the same oracle.
\end{corollary}

\begin{proof}
    Since the union is a fonsi, all the elements of the union must overlap. Since each fonsi is a fonsi, given a $q > 0$, there exists elements in each fonsi whose lengths are less than $q$. Since they overlap, we meet the criteria for applying Proposition \ref{pr:fonsi-inter}. Thus, they are associated with the same oracle. 
\end{proof}

For constructing an oracle, a fonsi can be a very appealing pathway.  Here are a few common examples. In what follows, $i, m, n, M, N$ will all be positive natural numbers unless otherwise specified.

\begin{itemize}
    \item Nested Intervals. If we have a sequence of rational intervals that are nested and the lengths go to zero, then that sequence is a fonsi.
    \item Nesting function. If we have a function $I(p)$ that yields a rational interval of length less than the non-negative rational length $p$ such that $I(p)$ is contained in $I(r)$ for $r>p$, then we call this a nesting function and the function's range is a fonsi. If it is possible to define $I(0)$ in such a way that this maintains the nesting property, then it should be defined. This happens if and only if there is a rational number $q$ contained in each of the intervals; that rational number would be unique. 
    
    \item Sequences with Shrinking Error Bounds. If we have a sequence of rational numbers $q_i$ paired with an upper bound $\varepsilon_i$ such that $\varepsilon_i$ goes to zero and all future numbers $q_m$ in the sequence are contained in $q_i:q_m:q_i+\varepsilon_i$, then the set of intervals $q_i:q_i+\varepsilon_i$ form a fonsi. Note that the $\varepsilon_i$ do not need to be consistently decreasing. The oracle that is formed from this is denoted $\lim_{i \to \infty} q_i $. 
    
    \item Positive Sums with a Shrinking Error Bound. Assume that we have a sequence $a_i\geq 0$ and let $S_n = \sum_{i=0}^n a_i$. Further, assume we know that there exists $p_n \geq 0$ such that $S_n: S_m : S_n + p_n$ for all $m > n$ and that for any given $M$, we can find $N$ such that $p_n < \frac{1}{M}$ for all $n \geq N$, then we have $S_n: S_n + p_n$ form a fonsi. The resulting oracle is denoted by $\sum_{i=0}^\infty a_i$. Note that we did not need to show that $S_m+p_m < S_n + p_n$. We only needed to know that $S_m$ was in the previous intervals. This is very convenient.
    
    \item Absolutely Converging Series. Assume that we have a sequence $a_i$ and let $S_n = \sum_{i=0}^n a_i$. Further, assume that we have absolute convergence which is the claim that $\sum_{i=0}^\infty |a_i|$ exists as a positive sum  with a shrinking error bound. Set $p_n$ to be an error bound for $A_n = \sum_{i=0}^n |a_i|$, i.e., $A_n : A_m : p_n + A_n$ for all $m \geq n$ and, for a given $M$, we can find $N$ such that $p_n \leq \frac{1}{M}$ for $n \geq N$.  This is doable by the assumption of absolute convergence.
    
    We claim that $S_n-p_n: S_m : S_n + p_n$ for all $m > n$.  The claim relies on the triangle inequality in the form of $- \sum |a_i| \leq \sum a_i \leq \sum |a_i| $ for any of the possible finite sums we can form of the $a_i$.

    For $n < m$, define $A_{n,m} = \sum_{i=n+1}^m |a_i|$ and $S_{n,m} = \sum_{i=n+1}^m a_i$. By the triangle inequality, we have $-A_{n,m} : S_{n,m} : A_{n,m}$. But as $A_m = A_{n,m} + A_n $, we have that $A_n:A_m:p_n + A_n$ implies $A_n : A_n+ A_{n,m} : p_n + A_n$ telling us that $0: A_{m,n} : p_n$. Thus, $-p_n : S_{n,m} : p_n$. Adding $S_n$, we have $S_n - p_n : S_n + S_{m,n} : S_n + p_n$  or $S_n - p_n : S_m : S_n + p_n$. Thus, the collection of intervals $S_n-p_n:S_n+p_n$ forms a fonsi. The resulting oracle is denoted by $\sum_{i=0}^\infty a_i$.
\end{itemize}

\subsubsection{Measurements and Finite Fonsis}

Another reason to favor fonsis is that they seem to be an infinite version of our idealized notion of taking measurements. A good measurement of a physical quantity is one in which the idealized actual value is always within the interval of the value plus or minus the error. A heuristic for that is that all measurements should overlap. A goal of measurement is to be able to reduce the error to any level that we please though that is never fully realized, of course. 

It is in this sense that a fonsi is the infinite version of measurements. We can then go the other way and introduce a notion of $\varepsilon$-partial fonsis, namely, a set of intervals such that there is at least one interval which is as small or smaller than the given $\varepsilon$ and that they all intersect. In fact, we can simplify this by taking the common intersection of all the measurements to produce a measurement interval which is less than the single specified $\varepsilon$. 

This is essentially a finite fonsi version of the resolution of the compatibility of two oracles. Here, one oracle is the imagined actual true value that we are trying to measure and the other oracle is the oracle of all the measurement intervals. If we could carry out the infinite processes required, these would be the same. 

We can also view the oracle rule itself as being recast as a questioning of ``Is this measurement interval a valid measurement of the actual value that we are seeking?''

\subsection{Examples}

We now pause and give some explicit computations for computing out the three most widely known irrational numbers: $\pi$, $e$, and $\sqrt{2}$. The hope is to demonstrate how this framework interfaces with getting explicit bounds. The notion of a fonsi is very helpful here. 

\subsubsection{Circles, \texorpdfstring{$\pi$}{pi}, and Sums}

There are, of course, many ways of computing $\pi$, but here we will start with the classic approach of circumscribed and inscribed polygons and then separately proceed into the Bailey-Borwein-Plouffe formula which is entirely rational in the approximations. 

The classic approach from Archimedes is to enclose the unit circle with a circumscribed regular polygon and sandwich that with an inscribed regular polygon. This gives us an upper and a lower bound for both the semi-perimeter and area, both of which equal $\pi$. 

A nice discussion of this can be found at Math Scholar\footnote{\url{https://mathscholar.org/2019/02/simple-proofs-archimedes-calculation-of-pi/}} by David H. Bailey. They start with regular hexagons and determine that the circumscribed semi-perimeter is $A_1 = 2 \sqrt{3}$ and the inscribed semi-perimeter is $B_1 = 3$. They then double the number of sides and compute again. They repeatedly do this to compute the recurrence relations $A_{k+1} = \dfrac{2A_k B_k}{A_k + B_k}$ and $B_{k+1} = \sqrt{A_{k+1}B_k}$. Notice that there is a square root involved in this computation. We have defined the oracles of square roots so that we can proceed, but this now requires oracle arithmetic, which we have not yet done and shall overlook for the moment.  

It is straightforward to see $A_{k+1} < A_k$, $B_{k+1} > B_k$, and that $A_k > B_k$.\footnote{Start with the observation that $A_1 > B_1 > 1$. Assume $A_k > B_k > 1$. Then $2 A_k  > A_k + B_k$. This implies that $M_k = \frac{2A_k}{A_k + B_k}>1$ as is its square root. This leads to both $B_{k+1} =  \sqrt{A_{k+1} B_k} = \sqrt{M_k} B_k > B_k > 1$  and $A_{k+1} = M_k B_k > \sqrt{M_k} B_k = B_{k+1}$.  For $A_{k+1} < A_k$, this follows from $2A_k B_k < A_k^2 + A_kB_k$ which follows from $A_k > B_k$.} The post further shows that $0 < A_k - B_k \leq \dfrac{128}{9*4^k} = p_k $. This is a sequence of nested intervals whose lengths are going to 0. Thus, the set of $A_k:B_k$ forms a fonsi and therefore yields an oracle. 

The post also establishes that it is indeed $\pi$ that is being approximated with these intervals.

This approach is basically fine except for the square root formula. It adds in extra complications in terms of interval arithmetic which would be nice to avoid if possible. 

Many of the approaches to computing $\pi$ do use square roots. But there are some that are able to avoid it. One such example is from the paper by Bailey, Borwein, and Plouffe \cite{BBP}, in which they give the formula 

\[ 
\pi = \sum_{i=0}^\infty \frac{1}{16^i} \bigg( \frac{4}{8i+1} - \frac{2}{8i+4} - \frac{1}{8i+5} - \frac{1}{8i+6} \bigg)
\]

To apply our fonsi results, we need to show that that the sum is absolutely convergent. We can get an overestimate by replacing the parenthetical portion with 1 for estimating the remainder since the parenthetical will always be strictly between 1 and 0. Then the sum becomes a geometric sum leading to $p_n = \frac{1}{15*16^n}$ for the partial sum up to $n$. By absolute convergence as established earlier, this is a fonsi and thus leads to an oracle.  

This formula for $\pi$ is entirely rational and the $p_n$ gives us interval estimates.

\subsubsection{Products, \texorpdfstring{$e$}{e}, and Sums}\label{sec:e}

We define the Oracle of $e$ as follows. 

Let $S_N = \sum_{i=0}^N \frac{1}{i!}$. Notice that since we are adding positive terms as $N$ increases, we have $S_N < S_{N+1}$. 

We also can compute $S_M - S_N$ for $M - N = k > 0$ and see that we have $\sum_{i=N+1}^M \frac{1}{i!} <  \sum_{j=0}^k  \frac{1}{(N+1)!(N+1)^j} $ where we have replaced factors of $N+1 + j$ with $N+1$, making the sum larger due to making the denominators smaller. We factor out the factorial and then compute the remaining sum as a geometric sum, leading to $S_M - S_N < \frac{1}{N!} \frac{1}{N+1} [\frac{N+1}{N} (1 - \frac{1}{(N+1)^{k+1}} )] < \frac{1}{N! N}$.

We define $p_n = \frac{1}{n!n}$. We established that $S_n < S_m < S_n + p_n$ for $m> n$. We can also clearly find $n$ such that $p_n$ is less than a given positive length. So we have a fonsi and an oracle. 

That was the standard sum approach to $e$. We can also look at another standard approach, that being $(1+\frac{1}{n})^n$. The typical path is to look at the limit as $n\to \infty$ and either use some differential / logarithmic calculus arguments or argue that it is an increasing sequence, bounded above, and therefore must have a limit value. Some approaches also compare this to the sum above, e.g., Rudin\cite{rudin}, page 64. 

That approach is fine with oracles, but not illuminating. Our viewpoint is that we want converging intervals. The expression $a_n = (1+\frac{1}{n})^n$ is a lower bound for $e$. We need an upper bound. A convenient upper bound is to add an extra factor: $b_n = (1+\frac{1}{n})^{n+1}$ will be greater than $e$. 

To establish this as an oracle, we again use the notion of a fonsi. We are done if we can show these are nested intervals whose lengths are shrinking to 0. It is immediate that $a_n < b_n$. The method we cite shows complete nesting: $a_n < a_{n+1} < b_{n+1} < b_n$. By extension, we then have $a_n < a_m < b_m < b_n$ for all $m > n$. The method is from Mendelsohn \cite{mend} which uses the Arithmetic-Geometric Mean Inequality on the two collections of numbers: $\{1, (1+ \frac{1}{n})_{,n} \}$ and $\{1, (\frac{n+1}{n})_{,(n+1)}\}$.\footnote{We use the notation $(a)_{,k}$ to indicate a quantity $a$ being included in the collection $k$ times.} We include the details of the argument in Appendix \ref{app:e}.  Another presentation, and the origin of finding the reference, can be found at the Mathematics Stack Exchange.\footnote{ \url{https://math.stackexchange.com/questions/389793/}}

We still need to show that the lengths shrink to 0. To do this, we look at the difference: $b_n - a_n = (1+\tfrac{1}{n})*a_n - a_n = \tfrac{a_n}{n}$. Because $a_n < b_n < b_1=2^2 = 4$, we have a simple bound of $\tfrac{4}{n}$ which can be made as small as we like. Thus, we have a fonsi and an oracle. 

It is useful to point out that our viewpoint leads us to find useful bounds so that we can know how accurate it is and what kind of convergence we can expect as well. For the compound interest formulation, we have an error on the order of $\frac{1}{n}$ while the Taylor-based approximation is on the order of $\frac{1}{n! n}$. It brings to the forefront explicit information that practical applications can use. As we shall explore, arithmetic with oracles is interval arithmetic and being able to have such explicit bounds allows us not only to compute what the resulting interval lengths will be, but also to figure out what value of $n$ to use to get a certain level of final precision. 

We now have two oracles claiming to be the same and we would like to establish their equality. We will be done if we can establish overlaps between the intervals. It is clearer if we lay it out in steps. We will use the labels from above. 
\begin{enumerate}
    \item $a_n < S_n$ for all $n$. This follows from the binomial theorem. See Appendix \ref{app:e}, Lemma \ref{lem:ansn}. If $n < m$, then $a_n < S_n < S_m$.
    \item Given $n$, there exists $m$ such that $S_n < a_m$. This is trickier. See Appendix \ref{app:e}, Lemma \ref{lem:snam}. 
    \item $S_n < b_m$ for all $m$ and $n$. This follows from $S_n < a_p$ for some $p$ and that $a_p < b_m$ for all $m$.
    \item $S_n < S_m < S_n + p_n$ for $m > n$. That was established above. 
    \item Given $a_n < b_n$ and $S_m < S_m + p_m$, they overlap. This follows from two cases. If $a_n \leq S_m$, then since $S_m < b_n$, we have $a_n : S_m : b_n$. The other case is that $S_m < a_n$. In this case, $m < n$ and we have $S_m < a_n < S_n < S_m + p_m$. 
    \item Since the intervals of both fonsis always overlap, we can use Proposition \ref{pr:fonsi-inter} to conclude that they are the same oracles. 
\end{enumerate}

For our final exploration of $e$, we would like to show that it is not a singleton, i.e, not a rational number. Given any rational of the form $\frac{r}{q}$, we claim that $\frac{r}{q}$ is not in $S_q : S_q + p_q$. If we show this, then if we have established that $e$ is a neighborly oracle. Our approach will be to proceed by contradiction, that is, we will assume that the rational is in that interval:  $S_q < \tfrac{r}{q} <  S_q + p_q = S_q + \frac{1}{q!q}$. We subtract off $S_q$ and obtain $0 < \tfrac{r}{q} - S_q < \tfrac{1}{q! q}$. If we multiply by $q! q$, we have $0 < q! r - q!S_q < 1$.  The expression $q! S_q$ is an integer since $q!$ will cancel all the denominators in that sum. But then we have the difference of two integers being strictly between 0 and 1. Since this cannot happen, we cannot have $\tfrac{r}{q}$ in the interval $S_q : S_q + \frac{1}{q! q}$. This establishes that the oracle is not a singleton with denominator $q$. Since $q$ was arbitrary, it cannot be a singleton for any rational. 

It is instructive to compare this to Rudin's proof of $e$ being irrational. It is the same, except Rudin only concludes that $e$ is not rational. What we have naturally drawn out of that same work is that there can be no rational with denominator $q$ in the approximation intervals contained in $S_q$'s interval. That is, we have an explicit marker after which we know the denominators in the interval must be larger than $q$. 

\subsubsection{Explicitly Computing Roots}

Section \ref{sec:roots} shows the existence of $n$-th roots, but it does not give a great way of finding an $n$-th root. 

We will give a common algorithm for computing these roots. It is equivalent to Newton's method, but that need not concern us here. What we produce here is a fonsi which therefore produces an oracle. By construction, they are Yes intervals of the $n$-th root of $q$ and so the fonsi yields the same oracle. 

We want to find the $n$-th root of $q$. This starts by realizing that given any positive $x$, we have $x:\dfrac{q}{x^{n-1}}$ is a $\sqrt[n]{q}$-Yes interval.

Our argument for that starts with the observation that $x^{n-1}*\dfrac{q}{x^{n-1}} = q$. Let $r$ represent the root, which exists as an oracle. For what follows, we will assume the usual rules hold for oracle arithmetic; we will cover them later. We have either $x^n < q$, $x^n = q$, or $x^n > q$ as the three possible cases. If we have equality, then $x$ is the solution and the interval above is a singleton as the root partner becomes itself when $x$ is the root, i.e, $x= \frac{q}{x^{n-1}}$ exactly when $x^n  =q$. 

Let us assume that $x^n < q$. We have that $x:\frac{q}{x^{n-1}}$ is a Yes interval as long as $(\frac{q}{x^{n-1}})^n \geq q$. So let us show that we have a contradiction if is is less than $q$. If it is lesser, then by Lemma \ref{app:lesser}, we have there exists an $s$ such that $\frac{q}{x^{n-1}} < s$ and $s^n < q$. By the same lemma, there is $t$ such that $x <t$ and $t^n < q$. Let $m = \max(s, t)$. Note $m^n < q$. From $x < m$, we have $x^{n-1} <m^{n-1}$. We also have $\frac{q}{x^{n-1}} < m$. Therefore, $q = x^{n-1} \frac{q}{x^{n-1}} < m^n < q$. This is a contradiction and we therefore must have $(\frac{q}{x^{n-1}})^n \geq q$ as we wanted to establish. 

Using Lemma \ref{app:greater}, we can make a similar argument that if $x^n > q$, then we must have $(\frac{q}{x^{n-1}})^n \leq q$. 

Having exhausted our cases, we have that $x:\dfrac{q}{x^{n-1}}$ is a $\sqrt[n]{q}$-Yes interval.

Since $a<b$ implies $\frac{q}{b^{n-1}} < \frac{q}{a^{n-1}}$, we can choose any rational in $x:\frac{q}{x^{n-1}}$ for our next ``$x$'' and interval computation. This will produce an interval that is contained in the previous one. If we use bisection to choose the next $x$, then we can guarantee the length of the following intervals is at least halved at each step. 

Newton's method suggests a different selection, namely, that we compute the weighted average of the guess and its partner, weighting the guess $n-1$ times in that average. Specifically, $\frac{1}{n} \big( (n-1) x + \frac{q}{x^{n-1}} \big)$. This will have a quadratic convergence once it gets close enough to the root.

For the square root, $n=2$, the formula becomes a simple averaging of the guess and its partner, the same as bisecting the interval.  

To begin the iteration, it is useful to first get close by considering a power of 10. For example, to compute the square root of $52400$, we can view that roughly as $5*10^4$ leading to a guess of $200$. The first interval would then be $200: \frac{52400}{200} = 262$ with the next guess being the average $231$ whose complement $\frac{52400}{231}$ is about $22 6$. It could iterate quite quickly from there. Note that one can be rather loose with the bounds as long as one is careful to round away from the partner.  

There are error estimates for this method which can help suggest how many iterations are necessary as well as establish this as a fonsi. We should also note that the interval itself generated by this method is an error estimate. Since the convergence is quadratic, we can roughly get a sense of the number of iterations required once the interval length drops below one-tenth at which point we start doubling the decimal precision at every iteration.

\subsection{Least Upper Bound}

We are given a non-empty set of rationals $E$ with an upper bound $M$, meaning that if $x \in E$ then $x$ has the property $x < M$. Let $U$ be the collection of rational upper bounds of $E$, namely, for rational $y$, $y \in U$ if and only if $y \geq x$ for every $x$ in $E$.  Thus, $M \in U$.

We define the Oracle of $\mathrm{sup} E$ to be the rule such that $R(a\lte b) = 1$ if and only if $a \leq y$ for all $y$ in $U$ and $b \geq x$ for all $x$ in $E$. That is, $a$ is a lower bound for $U$ and $b$ is an upper bound for $E$. This is also known as the least upper bound of $E$ (lub).

\begin{enumerate}
    \item Consistency. Assume $R(a\lte b)=1$ and $c\lte d$ contains $a\lte b$. Then $c \leq a \leq y$ for all $y$ in $U$ and $d \geq b \geq x$ for all $x$ in $E$. Thus, $R(c\lte d) = 1$.
    \item Existence. By assumption, there is an $x \in E$ and an upper bound $M$ in $U$. Thus, $R(x:M) = 1$. 
    \item Separating. Let $R(a\lt b)=1$ and $a < c< b$. We need to show that either $R(c:c)=1$ or $R(a:c) \neq R(c:b)$. 
    
    If $c \geq x$ for all $x$ in $E$, then $R(a:c)=1$ by definition and the fact that $a$ being a lower bound for $U$ still holds. We have two cases.  If $c \leq y$ for all $y$ in $U$, then $R(c:c)=1$ and we are done. If $c > y$ for some $y$ in $U$, then $R(c:b) = 0$ since $b > c$ and therefore neither of them can be less than $y$ for all $y$ in $U$. 
    
    The other case is that $c < x$ for some $x$ in $E$. Since $a < c$, we know that neither $a$ nor $c$ is greater than all $x$ in $E$. Thus, $R(a:c) = 0$. Since $c < x$ for some $x$ in $E$, and any upper bound $y$ has the property that $x < y$, we have $c < y$ for all $y$ in $U$. We therefore have $R(c:b)=1$.
    
    \item Rooted. For $R(c:c)=1$, we would need to have that $c \leq y$ for all $y$ in $U$ and $c \geq x$ for all $x$ in $E$. Assume we had another rational such that $R(d:d) = 1$. Then we also have $d \geq x$ and $d \leq y$ for all $y$ in $E$. Since both $c$ and $d$ are upper bounds of $E$, we have that each are in $U$. So $c \leq d$ and $d \leq c$. This implies that $c = d$. 
    \item Closed. Assume $c$ is in all Yes intervals. Then given any upper bound $y$ in $U$ and element $x$ in $E$, we have $c$ is in $x:y$ since $R(x:y)=1$ by definition. We therefore have that $c \leq y$ and $c \geq x$. Since $x$ and $y$ were arbitrary elements of $E$ and $U$, respectively, we have $c$ satisfies the conditions such that $R(c:c)=1$.
\end{enumerate}

We also need to establish that $\sup E$ satisfies $a \leq \sup E \leq b$ for all $a \in E$ and all $b \in U$. For any given $a \in E$ and $b \in U$, we have that $a:b$ is a $\sup E$ interval by the oracle definition of $\sup E$. Thus, we have that $a \leq \sup E \leq b$ by Proposition \ref{pr:yes-trap}. 

If the non-empty set $E$ is bounded below by $K$, then we can define the set $L$ of lower bounds of $E$, namely, $z \in L$ if $z < x$ for all $x$ in $E$. Then the greatest lower bound (glb), denoted by $\mathrm{inf} E$, is the least upper bound of $L$ which exists since $L$ is bounded above by elements of $E$ which is assumed to be non-empty. 

We have defined the lub and glb in terms of sets of rationals. We can extend this to sets of oracles. Assume we are given a nonempty set $F$ of oracles bounded above, i.e., each oracle in $F$ is less than some given oracle $M$. Let $E$ be the set of lower endpoints of all the Yes intervals of all the oracles in $F$. $E$ is bounded above by an upper endpoint of an $M$-Yes interval. So $\sup E$ exists. Immediately, we have that $\sup E \geq r$ for all $r \in F$ since it is greater than all of the lower endpoints of the Yes intervals of the oracles in $F$.\footnote{If $\sup E < r$ for some $r \in F$, then there would exist $\sup E$-Yes interval $a \lte b$ and $r$-Yes interval $c\lte d$ such that $a:b < c:d$. But this would imply that $b < c$, but all of the upper endpoints of $\sup E$ need to be upper bounds of the rationals in $E$ which $c$ is a part of. This is a contradiction. Thus, there is no such $r$.}

The harder case is to show that $\sup E \leq \alpha$ for all upper bounds $\alpha$ of $F$. Let $\alpha < \sup E$ implying there is an $\alpha$-Yes interval $a\lte b$ and a $\sup E$-Yes interval $c\lte d$ such that $a:b < c:d$. If $\alpha$ was an upper bound of $F$, then $b$ would be an upper bound of all the rationals in $E$. But that would imply that it is an upper bound of $E$. Since $b$ is less than $c$, we would have $c$ is an upper bound. But the oracle $\sup E$ has the property that the lower endpoints of its neighborly Yes-intervals are not upper bounds of $E$. Thus, we have a contradiction and $\sup E$ is indeed the least upper bound of $F$.

We have established

\begin{theorem}\label{th:lub}
Any set of oracles bounded from above has a least upper bound. 
\end{theorem}

An alternate approach is to say that $a:b$ is a Yes interval if $a$ is a lower endpoint of a Yes interval of an oracle in $F$ and $b$ is an upper endpoint of a Yes interval of an oracle which is an upper bound of F. This may be closer in spirit to the oracle notion than the above treatment. We do so in Appendix \ref{app:sup}. 

\subsubsection{The Distance Function}

We can now define the distance between two oracles. We could also do this using oracle arithmetic, but the distance is both useful and instructive to do at this stage. 

Given two oracles $r$ and $s$, we define the distance set, $D_{r,s}$ to be the set of rational values generated by looking at each pairing of $r$-Yes intervals with $s$-Yes intervals, say  $a:b$ and $c:d$, respectively, and computing the interval distance by $d(a:b,c:d) = \mathrm{max}(|a-c|, |a-d|,|b-c|,|b-d|)$. The \textbf{distance between $r$ and $s$, denoted by $d(r,s)$,} is then defined to be the greatest lower bound of $D_{r,s}$, which is bounded below by 0. 

Note that this means that $d(r,s) \leq d(a:b, c:d)$ for any $r$ and $s$ Yes intervals. We also have that the interval distance function has the property that $d(e:f, c:d) \leq d(a:b, c:d)$ for the situation where $a:b$ contains $e:f$ since $e, f$ will be closer than the farther of $a,b$ to $c,d$. For example, if $a < e < f < b < c< d$, then $d(e:f, c:d) = d-e < d-a = d(a:b, c:d)$. 

If $r<s$, then for sufficiently narrow $r$ ($a\lte b)$ and $s$ ($c\lte d$) Yes intervals, we have $a\leq b<c \leq d$ so that $d(a:b,c:d) = d-a$ and $D_{r,s}$ will be bounded below by $c-b$. That is, $(d-a):(c-b)$ is a Yes interval for $d(r,s)$.

The distance function is always non-negative. If $r=s$, $d(r,s)=0$ as we can choose $a<b$ as the same interval for both and compute $d(a:b,a:b)= b-a$. Since these are the same oracle, $a\lte b$ can be chosen such that $b-a$ can be as small as we like via the bisection approximation (Proposition \ref{pr:short}) and hence the greatest lower bound is 0. 

The distance function also satisfies $d(r,t) = d(r,s) + d(s,t)$ whenever $r < s < t$. Indeed, let $a\lte b\lte c\lte d\lte e\lte f$ with $a\lte b$ being an $r$-Yes interval, $c\lte d$ an $s$-Yes interval, and $e\lte f$ a $t$-Yes interval. Then we have: 
\begin{enumerate}
\item $c-b \leq d(r,s) \leq d-a$
\item $e-d \leq d(s,t) \leq f-c$
\item $e-b \leq d(r,t) \leq f-a$
\item Adding the first two, we have $c-b + e-d \leq d(r,s) + d(s,t) \leq d-a + f-c$.
\item Rearranging, we have $e-b + (c-d) \leq d(r,s) + d(s,t) \leq f-a + (d-c)$.
\end{enumerate}
We have not yet covered arithmetic, but narrowing on interval addition is how we define adding oracles. What we have shown here is that the interval $(e-b + (c-d)): (f-a + (d-c)$ is a $[d(r,s) + d(s,t)]$-Yes interval. Since $c-d$ is negative and $d-c$ is positive, the $d(r,t)$-Yes interval $(e-b):(f-a)$ is contained in it. Since we can narrow the intervals arbitrarily,\footnote{Let $f = e + \delta_1$ and $b = a + \delta_2$. Then $f-a = e-b:e-b + \delta_1 + \delta_2$. We can take $\delta_1$ and $\delta_2$ as small as we like by bisection algorithm.} by Proposition \ref{pr:fonsi-inter}, we know their oracles are the same. 

From this, we also get the triangle inequality. For any oracles $r, s, t$, we have $d(r,t) \leq d(r,s) + d(s,t)$.  To argue for this, if $s = r$ or $s=t$, then we have equality. If $r=t$, then we have $0 \leq 2 d(r,s)$ which is true. We can assume, without loss of generality, that $r < t$. We then have three cases: $r< s< t$, $s < r < t$, and $r< t  < s$. If $r < s< t$, then we have equality as above. If $s < r < t$, then $d(s,t) = d(s,r) + d(r,t)$ as above and thus $d(r,t) = d(s,t) - d(s,r) < d(s,t) + d(s,r)$. The same can be said for $r < t < s$. Indeed,  we have $d(s,r) = d(r,t) + d(t,s)$ which leads to $d(r,t) \leq d(s,t) + d(s,r)$.

\subsection{Cauchy Sequences}\label{sec:cauchy}

Closely related to nested intervals are the Cauchy sequences. We will first define Cauchy sequences and then describe the associated fonsi which yields an oracle. The basic idea is an interval is a Yes interval if it contains the tail of the sequence. 

A sequence $a_n$ of oracles is \textbf{Cauchy} if given any natural number $M > 0$, there exists an $N$ and a rational interval $I_N$ whose length is less than or equal to $\frac{1}{M}$ such that for $n \geq N$, the interval $I_N$ is an $a_n$-Yes interval. That is, $I_N$ contains the oracles later in the sequence. 

\begin{theorem}\label{th:cauchy}
Any Cauchy sequence of oracles converges to an oracle. 
\end{theorem}

\begin{proof}
    Consider the collection of $I_N$. They all intersect one another as they all contain the latter elements of the sequence. Specifically, if we have $I_N$ and $I_P$, take $n \geq \max(N, P)$. Then $I_N$ and $I_P$ must intersect as they are both $a_n$-Yes intervals. We also can find an $I_N$ which is less than any given length as that is the condition for choosing $N$ based on a given $M$. Therefore, the collection of $I_N$ is a fonsi and it generates an oracle, one in which the tail of the sequence is as close to the oracle as we care to prescribe.     
\end{proof}

The proof is short as the bulk of the work is being done by the fonsi which, as we can see here, is notionally close to what the Cauchy sequences are essentially representing. 

The Cauchy condition presented here is equivalent to the condition of the more usual formulation.

\begin{proposition}
    A sequence of oracles is Cauchy if and only if the sequence satisfies for any $n, m \geq N$, $d(a_n,  a_m) < \varepsilon$ where $d$ is the oracle distance function and $\varepsilon > 0$ is an oracle. 
\end{proposition} 

\begin{proof}
    
If a sequence is Cauchy by our definition, then we can easily prove it satisfies the usual formulation.  Let such a sequence be given with $I_N$ having length less than $\frac{1}{M} < \varepsilon$ and for all $n \geq N$, we have $I_N$ is an $a_n$-Yes interval. Then consider $d(a_n, a_m)$ for $n, m \geq N$. Since $I_N$ is a Yes interval for both, we have $d(a_n, a_m) \leq d(I_N, I_N) < \frac{1}{M} < \varepsilon$ as the distance of an interval from itself is just it's length. 

The other direction requires a little more careful choosing of the lengths. We are given a natural number $M>0$. By assumption, we can find $N$ such that that $d(a_n, a_m) < \frac{1}{9M}$ for all $n, m \geq N$. Take an $a_N$-Yes interval whose length is less than $\frac{1}{9M}$; let's say that interval is $a \lt b$. Then define $I_N = a-\frac{1}{3M} \lte b + \frac{1}{3M}$. The length of $I_N$ is less than $\frac{1}{M}$. We are therefore done if we can show that $I_N$ is an $a_n$-Yes interval for all $n \geq N$.

The distance statement implies that we can find an $a_n$-Yes interval $e:f$ and an $a_N$-Yes interval $c:d$ such that $d(e:f, c:d) < \frac{1}{9M}$. Since subinterval distances shrink and we can bisect Yes intervals to get smaller Yes intervals, we can also assume that $e:f$ and $c:d$ have lengths smaller than $\frac{1}{9M}$. 

In order for $e:f$ to be contained in $I_N$, we need $d(a:b, e:f) \leq \frac{1}{3M}$. The distance function satisfies a version of the triangle inequality, namely $d(a:b, e:f) \leq d(a:b, c:d) + d(c:d, c:d) + d(e:f, c:d)$. This holds for all intervals\footnote{This follows from the definitions and the usual triangle inequality. Let $|A-E| = d(a:b, e:f)$, $|B-D| = d(a:b, c:d)$, $|c-d| = d(c:d, c:d)$, and $|C-F| = d(c:d, e:f)$ where the capital letters are the different points in their respective intervals that realize the distances. Then $|A-E| = |A - D + D - C + C - E| \leq |A-D| + |D-C| + |C-E| \leq |B-D| + |c-d| + |C-F|$ where the substitutions are justified by each term being replaced with the maximum difference of the points in the intervals.  We therefore have $d(a:b, e:f) \leq d(a:b, c:d) + d(c:d, c:d) + d(e:f, c:d)$.}. Substituting in our choices, we get $ d(a:b, e:f)  \leq \frac{1}{9M} + \frac{1}{9M} + \frac{1}{9M} =  \frac{3}{9M} = \frac{1}{3M}$ as was to be shown. 

The interval $I_N$ therefore contains the $a_n$-Yes interval $e:f$ and, by Consistency, is an $a_n$-Yes interval itself. 

\end{proof}

\subsection{The Collatz Number}

There is an open mathematical problem, called the Collatz conjecture, which conjectures that a certain process stops for every natural number.\footnote{\url{https://en.wikipedia.org/wiki/Collatz_conjecture}} We will use that process to define a couple of oracles whose ultimate nature can only be settled by someone solving the Collatz conjecture. 

We define one Oracle of Collatz $R$ to be such that $R(-\tfrac{1}{n}:\tfrac{1}{n})$ is 1 if the $k$-Collatz sequence terminates at 1 for all $k \leq n$ and 0 otherwise. For all other intervals $a:b$, $R(a:b) = 1$ if $a:b$ contains such a Collatz Yes interval. If there is an $n$ which does not satisfy the Collatz conjecture, then, taking $N$ to be the first such $n$, we define $R(\tfrac{1}{N}:\tfrac{1}{N}) = 1$. If there is no such $n$, then $R(0:0) = 1$.

This is a rule which is currently known to be $-2^{-68}:2^{-68}$ compatible with the Oracle of 0. But unless it gets proven or falsified, we cannot establish equality or inequality with respect to 0. 

In the ultimate form of a known result for the conjecture, this will be an Oracle of 0 or an Oracle of $\tfrac{1}{N}$ where $N$ is the first number where the conjecture fails to hold. Until that happens, we cannot practically use the Separating or Rooted properties in all possible cases. 

An alternative approach to using this conjecture is to define an oracle that implies the $n$-th digit in its decimal approximations is $1$ if the Collatz conjecture fails on it and $0$ otherwise. We do this in the following way. Let $C(n)$ be an indicator function which is 1 if $n$ does not satisfy the Collatz conjecture and 0 otherwise. Define the following function as follows: $f(1) = 0$,  $f(n) = \sum_{i=1}^n C(i)*10^{-i} $. Define $p_n = \frac{1}{9*10^n}$ which is the closed form of the geometric sum $\sum_{i=n+1}^{\infty} 10^{-i}$ which obviously approaches $0$ as $n \to \infty$. The set of $f(n): f(n) + p_n $ define a fonsi and therefore an oracle. This is assuming, of course, that one can compute out what $C(i)$ is. This is also compatible with the 0 oracle at the present time, but even if we establish that it is not the 0 oracle, we may still not be able to compute $C(i)$ for all $i$. 

This example is an issue for all definitions of real numbers, but it feels as if the oracle approach accommodates this a little better as the incompleteness of our information is naturally incorporated into the framework using the language of compatibility.

\subsection{Coin Tosses}

We want to model an oracle that is probabilistic in nature and explore how that might work. 

The first version is to create a fonsi by a bisection method. We start with an initial interval, say $I_1 = 0:1$. Then, given $I_n = a:b$, we let $c = \tfrac{a+b}{2}$ and then use a random function that yields 1 with probability $p$ and 0 with probability $1-p$. If it is 1, then $I_{n+1} = a:c$. Otherwise, $I_{n+1} = c:b$. These intervals are clearly shrinking and overlapping, hence we can define a family $F_n$ of oracles compatible up to $I_n$ for any given $n$. We will never have a single, final oracle as it can only be defined fully after this infinite subdividing process completes. 

A process closer to defining an oracle directly is to flip a coin when asked a question about an interval when the answer is not already known. We will use $R$ for the rule. Let $I$ be a given starting interval and we set $C=I$ where $C$ stands for current. Let $a:b$ be the interval asked about. If $a:b$ contains $C$ (or is $C$), then $R(a:b) = 1$. If $a:b$ and $C$ are disjoint, then $R(a:b) = 0$. If neither, then $a:b$ and $C$ intersect. Let $c:d$ then be that intersection. If the interval $c:d$ divides $C$ in two, then we use a random process to determine if $c:d$ is a Yes interval or a No interval; the complement interval in $C$ is then the opposite. Then $a:b$ is a Yes or No interval if $c:d$ is Yes or No, respectively. We also redefine $C$ to be $c:d$ or its complement, depending on which one is Yes. It is possible for $c:d$ to divide $C$ into three intervals in the case that $a:b=c:d$ is strictly contained in $C$. We check $c:d$ first; if it is Yes, then the other two are No. If $c:d$ is No, then we do the random process to decide which of the two remaining intervals making up the complement are Yes. Whichever one is Yes, becomes the new $C$ for future Rule consultations.  

This oracle will satisfy consistency, existence, and separating all by definition. Rooted will apply as well. As for closed, $R(c:c)=1$ can only happen if we ask about it specifically and it comes up with a 1 in which case we have finished our process. Until we ask about a singleton and get a Yes, we will have many rationals that are contained in all the Yes intervals asked about. Because the oracle never gets finished if not rooted, we never get to apply the closed property. One could argue that this fails to be an oracle because we cannot establish that it will satisfy the closed property. 

This is an example of a rule which is guaranteed to give an answer for any question we ask even if the answer cannot be known in advance. In fact, the very nature of what is possible is based on what we ask. For example, if we never ask about  $c:c$, then it will not happen. But if we do ask, then it could come up as Yes. 

In a certain sense, our rule here is not ``complete''. An infinite being looking at it would see the finiteness. But from our limited perspective, it works just as much as any ``complete'' rule. In contrast, the Collatz example is one which we might not get an answer to. It is ``incomplete'' on a practical level. 

Comparing both of these random processes, the first seems to give an air of incompleteness which the second lacks. In particular, the first one's design looks like an infinite process that should go on until the completion, something which can never be accomplished. In contrast, the second one requires an interaction. It feels more like an answer given just in time. While it is also incomplete, it does not feel incomplete. 
 
Part of it is that we get to ask the second process directly the question of interest, namely, ``is this interval that we care about a Yes interval?'' In the first process, we need to keep computing the bisected intervals until we can definitively answer the question about any given interval. 

In comparing it the Collatz number example, a striking difference is that the Collatz procedure is one in which different people doing the mathematics correctly will not find different answers. The random processes yields different answers for different questioners, quite radically so for the second random process. 

Both the random processes and the Collatz example explains our focus on the rule rather than the set of Yes-intervals. The rule approach allows for more possibilities than the set version, as these examples demonstrate. There is no way to present the completed set of Yes intervals for these examples. One can argue about whether one can present the completed set of Yes intervals for the square root of 2 or not, but one can clearly not do so for these processes. Whether one wants to include these examples as valid oracles is a question which we leave to the reader to decide. 

\section{Interval Arithmetic}

This section is not novel, but rather a review of applying arithmetic to intervals. We will need this in defining the arithmetic of oracles. This material can be seen, for example, in the videos by NJ Wildberger.\footnote{\url{https://youtu.be/xReU2BJGEw4}} Intervals are of interest as they help propagate error bounds in scientific computations. The oracle approach reflects the usefulness of this thinking. 

The guiding idea of interval arithmetic is to have the arithmetic operation operate on each pair of elements from the two intervals. We then deduce the minimal interval that contains all those results. We could try to define it that way and then prove what we are about to assume. We will, instead, define these operations below and then establish that given any two elements in the operated-on intervals, the result will fall in that resultant interval. Due to the inclusion of the endpoints and our formulas use the endpoints, there is no smaller interval that could be obtained that covers all the pairwise combinations. 

This is not the same as the eventual oracle arithmetic. Here, for example, subtracting $a<b$ from itself results in the interval $a-b:b-a$ while subtracting an oracle from itself will result in $0$, coming from the ability to shrink that $a-b:b-a$ result arbitrarily and noting that $0$ is always contained in it. 

\subsection{Definition of Interval Arithmetic}

Let $a \leq b$ and $c \leq d$, all of them being rational numbers. Then we define:
\begin{enumerate}
    \item Addition. $a\lte b \oplus c\lte d = (a+c)\lte (b+d)$
    \item Negation. $\ominus\ a\lte b = -b\lte -a$
    \item Subtraction. $a\lte b \ominus c\lte d = a\lte b \oplus (-d\lte -c) = a-d\lte b-c$
    \item Multiplication. $a\lte b \otimes c\lte d = \min(ac, ad, bc, bd)\lte  \max(ac,ad,bc,bd)$. For $0<a<b$ and $0<c<d$, this is equivalent to $a\lte b \otimes c\lte d = ac\lte bd$. 
    \item Reciprocity. $1 \oslash (a\lte b) = \frac{1}{b}\lte \frac{1}{a}$ as long as $a\lte b$ does not contain 0. If 0 is contained in $a \lt b$, then the reciprocal is undefined as it actually generates the split interval of $-\infty\lte \frac{1}{a}$ and $\frac{1}{b}\lte \infty$.\footnote{We could define the interval $0 \lt b$ to have a reciprocal as $\frac{1}{b}\lte \infty$. This could be useful if we wanted to introduce one-sided oracles, such as having a $0^+$ oracle which would consist of $0$-Yes intervals of the form $0\lt b$. This would allow us to define the reciprocal of the Yes intervals as $\frac{1}{b}\lte \infty$ and would represent the oracle of $+ \infty$. We will not pursue that here.  }
    \item Division. $(a\lte b) \oslash (c\lte d) = a\lte b \otimes \frac{1}{d}\lte \frac{1}{c}$ where $c\lte d$ does not contain 0. Applying the multiplication rule, we find that we can view it as:   $\min(\frac{a}{c}, \frac{a}{d}, \frac{b}{c}, \frac{b}{d})\lte  \max(\frac{a}{c},\frac{a}{d},\frac{b}{c},\frac{b}{d})$. For $a, b, c, d > 0$, we have $\frac{a}{d} \lte  \frac{b}{c}$
    \item Natural Powers. Let $n$ be a natural number. $(a\lte b)^n = a^n\lte b^n$ if $a$ and $b$ have the same sign, 0 inclusive. If $-a > b$, then $a^n : a^{n-1} b$. Otherwise, we have $b^n \lte a b^{n-1}$. Note this is repeated interval multiplication and not the application of $n$ directly to a number as a power in the interval.\footnote{ As an example, $(-3 \lte 2)^2$ becomes $-6\lte 9$. It is not $0\lte 9$ which is what it would be if we squared the elements. In other words, we are not just self-pairing the elements, but rather we are considering all the products. } 
    \item Negative powers ($a\lte b$ not containing 0). $(a\lte b)^{-n}$ for natural $n$ is defined as $(\frac{1}{b}\lte \frac{1}{a})^n$. They must have the same sign and this is therefore the same as $b^{-n} \lte  a^{-n}$.
\end{enumerate}

\subsection{Containment}\label{containment}

We now want to show that each of the above operations applied to pairings of numbers in the intervals  does lead to a number in the defined interval above. 

Let $a \leq p \leq b$ and $c\leq  q \leq  d$. Then from normal inequality arithmetic, we have: 

\begin{enumerate}
    \item Addition.   $a +c \leq  p + q \leq  b +d$ thus yielding $(a+c):(b+d)$.
    \item Negation.  $-b \leq -p \leq -a$ thus yielding $-b:-a$.
    \item Subtraction.  $a - d \leq p-q \leq b -c$  thus yielding $a-d:b-c$.
    \item Multiplication. Signs can lead to a number of cases to check. We will avoid this by  applying our colon notation for three numbers, namely,  $x:y:z$ means that $y$ is between $x$ and $z$. This is helpful since we do not need to track the inequality directions. 
    
    For $a:p:b$, we can multiply by a number and containment is maintained. So $ca:cp:cb$, $qa: qp: qb$ and $da:dp:db$ all hold true. We also have $c:q:d$ which leads to $ca:qa:da$, $cp:qp:dp$, and $cb:qb:db$. This means that $qp$ is contained within the bounds of $cp$, $qa$, $qb$, and $dp$. Those bounds are contained within $ca$, $cb$, $da$, and $db$. Therefore, $qp$ is contained within $\mathrm{min}(ca, cb, da, db):\mathrm{max}(ca, cd, da, db)$. 
    
    A table form of the inclusion would be 
    
    \begin{tabular}{ccccc}
        $ca$ &$:$& $cp$ &$:$& $cb$ \\
        $\cdots$ & & $\cdots$ & & $\cdots$ \\
        $qa$ &$:$& $qp$ &$:$& $qb$\\
        $\cdots$ & & $\cdots$ & & $\cdots$ \\
         $da$ &$:$& $dp$&$:$& $db$
    \end{tabular}
    
    Let us demonstrate with some examples. 
    
    A fully positive example is $2:p:3$ and $5:q:7$ leading to 
    
     \begin{tabular}{ccccc}
        $10$ &$:$& $5p$ &$:$& $15$ \\
        $\cdots$ & & $\cdots$ & & $\cdots$ \\
        $2q$ &$:$& $qp$ &$:$& $3q$\\
        $\cdots$ & & $\cdots$ & & $\cdots$ \\
         $14$ &$:$& $7p$&$:$& $21$
    \end{tabular}
    
    with the result of $10:21$ being the multiplication of $2:3$ with $5:7$.
    
    For $-2 : p : 7$ and $3: q : 5$, we have
    
     \begin{tabular}{ccccc}
        $-6$ &$:$& $3p$ &$:$& $21$ \\
        $\cdots$ & & $\cdots$ & & $\cdots$ \\
        $-2q$ &$:$& $qp$ &$:$& $7q$\\
        $\cdots$ & & $\cdots$ & & $\cdots$ \\
         $-10$ &$:$& $5p$&$:$& $35$
    \end{tabular}
    
    and thus $-10:35$ is the result of multiplying $-2:7$ with $3:5$. Notice that if we replaced $3$ with, say, $2$ or $4$, the result is unchanged. This is because of the mixed sign interval which leads to the minimum value being created by growing the negative value's absolute value. 

     For $-2 : p : 7$ and $-3: q : 5$, we have
    
     \begin{tabular}{ccccc}
        $6$ &$:$& $-3p$ &$:$& $-21$ \\
        $\cdots$ & & $\cdots$ & & $\cdots$ \\
        $-2q$ &$:$& $qp$ &$:$& $7q$\\
        $\cdots$ & & $\cdots$ & & $\cdots$ \\
         $-10$ &$:$& $5p$&$:$& $35$
    \end{tabular}
    
    and thus $-21:35$ is the result of multiplying $-2:7$ with $-3:-5$.

    For our final example, let's do all negatives: $-2 : p : -7$ and $-3: q : -5$ leading to
    
     \begin{tabular}{ccccc}
        $6$ &$:$& $-3p$ &$:$& $21$ \\
        $\cdots$ & & $\cdots$ & & $\cdots$ \\
        $-2q$ &$:$& $qp$ &$:$& $-7q$\\
        $\cdots$ & & $\cdots$ & & $\cdots$ \\
         $10$ &$:$& $-5p$&$:$& $35$
    \end{tabular}
    
    and thus $6:35$ is the result of multiplying $-2:-7$ with $-3:-5$.
    
    \item Reciprocity. Let $0 < a \leq p \leq b$, then $1/a \geq 1/p  \geq 1/b > 0$. Similarly, $a \leq p \leq b< 0$ has $1/a \geq 1/p \geq 1/b$. What fails is if, say,  $a \leq p < 0 < b$, we would have $1/b > 0 > 1/a \geq 1/p $ and similarly if $p$ was positive, it would flip over $1/b$ but not $1/a$.
    \item Division. This follows from Multiplication and Reciprocity. 
    \item\label{natpow} Natural Powers. We could go to basic principles and play around with inequalities of the elements of the products, breaking into cases and dealing with various sign flippings. Instead, we can consider what happens under iterative multiplication. For squaring, we have that we need to find the maximum and minimum of $a^2, ab, b^2$. If they are the both positive, then $a^2 \leq ab \leq b^2$. If they are both negative, then $a ^2 \geq ab \geq b^2$. If they are different signs, then $ab < 0 < a^2, b^2$ and we need to compare the size of $-a$ and $b$.  In general, if they are of the same sign and $n$ is odd or they are both positive, then $a^n \leq a^{n-1} b \leq \cdots \leq ab^{n-1} \leq b^n$. For $n$ even and both are negative, then $a^n \geq a^{n-1} b \geq \cdots \geq ab^{n-1} \geq b^n$. If they are of different signs, then let $c = \max( |a|, |b|)$ and $d=\min(|a|, |b|)$. Then  $c^n \geq d c^{n-1} \geq d^{i}c^{n-i}$  for $i > 1$.  Since the signs are different, we have the products are negative when the power of $a$ is odd and positive for even powers of $a$. Thus, $c^n$ and $d c^{n-1}$ are the absolute value of the two endpoints. We have three cases: $c=b$ in which case $b^n$ will be the largest product and we have $b^n > 0 > ab^{n-1}$; $c=|a|$ and $n$ is even in which case $a^n > 0 > a^{n-1} b$; $c=|a|$ and $n$ is odd in which case $a^{n-1} b > 0 > a^n$.  

    As an example of the powers, consider $(-2:3)^4$. All of the products of the endpoints are: $16, -24, 36, -54, 81$. We therefore have the interval being $-54:81$ which is the product of $ab^3$ and $b^4$. 
     
    \item Negative powers are defined by combining reciprocity with natural powers. 
\end{enumerate}

We also want to point out that the above pointwise containment then extends so that if $a:b$ and $c:d$ are contained in $e:f$ and $g:h$, respectively, then $a:b \oplus c:d $ is contained in $e:f \oplus g:h$ and similarly for the various other operations, with the understanding of the appropriate restrictions on not containing 0 for the division and reciprocals.

\subsection{Verifying the Rules}\label{sec:rules}

The associative and commutative rules of arithmetic apply to intervals. The distributive rule somewhat applies. Each of these is a distinct computation, but very straightforward. We will do the distributive property separately. 

\begin{enumerate}
    \item Interval addition is closed, namely, the sum of two intervals is another interval by definition. 
    \item Addition is Commutative. $a:b \oplus c:d= a+c : b+d = c+a:d+b = c:d \oplus a:b$. We used the commutativity of rational addition in the middle step. 
    \item Addition is Associative. $(a:b \oplus c:d) \oplus e:f = (a+c:b+d) \oplus e:f = ((a+c)+e):((b+d)+f) = (a+c+e):(b+d+f)$ where the last step is the associative property of rationals.  On the other hand, $a:b \oplus (c:d+e:f) = a:b \oplus (c+e:d+f) = (a+(c+e)):(b+(d+f)) = (a+c+e):(b+d+f)$ again by associativity of addition of rationals. Since they are equal to the same quantity, we have the associative rule of addition applying to intervals and we can comfortably write $a:b \oplus c:d \oplus e:f$ without requiring parentheses. 
    \item The singleton $0:0$ is the additive identity as $a:b \oplus 0:0 = a+0:b+0 = a:b$. 
    \item There is no additive inverse for neighborly intervals since subtraction leads to the length of the new interval being the sum of the lengths of the two given intervals. The singleton intervals do have additive inverses, namely $a:a \oplus -a:-a = 0:0$. 
    \item Interval multiplication is closed, namely, the product of two intervals is another interval by definition. 
    \item Multiplication is Commutative. This follows from the product interval being defined in terms of the products of the boundaries and those individual products commute. 
    \item Multiplication is Associative. If we are multiplying $a:b$, $c:d$, and $e:f$, then the product interval of the three is formed from the max and min of the set $\{ace, acf, ade, adf, bce, bcf, bde, bdf\}$. Since those products are unchanged by reordering, we have multiplication is associative. In terms of the square approach to multiplication, we can imagine extending that to a cube of 27 entries and seeing that the ordering is nothing but an irrelevant reorientation of the cube. 
    \item $1:1$ is the multiplicative identity as $a:b \otimes 1:1$ has the form $\max(1*a, 1*b):\min(1*a, 1*b) = a:b$. 
    \item There is no multiplicative inverse as multiplication has a non-zero length for neighborly intervals. The non-zero singleton intervals do have multiplicative inverses, namely $c:c  \otimes  \frac{1}{c} : \frac{1}{c} = 1:1$.
\end{enumerate}

For the distributive property, we do not have equality of the intervals. But we do have that one contains the other, which will be sufficient for our purposes to have the distributive rule apply to oracles. 

\begin{proposition}
We have the subdistributive property: $I = a:b\otimes(c:d \oplus e:f)$ is contained in $J = (a:b \otimes c:d) \oplus (a:b \otimes e:f)$. 
\end{proposition}

\begin{proof}
We can compute the interval limits. The interval $I$ has boundaries from the max and min of the set 
\[
\{a(c+e), a(d+f), b(c+e), b(d+f)\} = \{ac+ae, ad+af, bc+be, bd+bf\}
\] 
The interval $J$ is the interval 
\[
(\min(ac, ad, bc, bd) + \min(ae, af, be, bf) ) \lte (\max(ac, ad, bc, bd) + \max(ae, af, be, bf) )
\]
Since the boundaries of $I$ are contained in the possibilities of $J$, we do have $J$ containing $I$.
\end{proof}

Let us do an example demonstrating this. Consider $I = 2:3 \otimes ( 4:7 \oplus -6:-3)$ versus $J = (2:3 \otimes 4:7) \oplus (2:3 \otimes -6:-3)$. Computing out $I$, we have $2:3 \otimes -2:4 = -6:12$. For $J$, we have $8:21 \oplus -18:-6 = -10: 15$. This is one example of $I$ being strictly contained in $J$. 

As we have seen, the arithmetic of intervals has some common properties with normal number arithmetic, but it is not the same.

\section{Oracle Arithmetic}

We can now define oracle arithmetic. The basic idea is that if an interval contains the result of combining Yes intervals of the oracles being combined, then it is a Yes interval itself. 

We will prove a general statement about creating an oracle out of other oracles based on maps that have the property that shrinking intervals of the inputs lead to shrinking the interval of the output. We will then apply it to the various forms of arithmetic operators. The maps under consideration take in rationals and yield rationals. 

This shrinking is what allows us to go from the arithmetic of intervals, which does not have a mechanism to use this property, to an arithmetic of oracles, which can use it since we can shrink the intervals of interest for oracles. 

A good example to keep in mind is that of multiplication. Let's assume, to avoid cases in this example, that  $0 < a < b < c < d \leq M$ and we are computing $a:b \otimes c:d$. Then we have that the length of that product interval is $bd-ac = bd-ad+ad - ac= (d(b-a) + a(d-c))< M ( ( b-a) + (d-c) )$. We can replace this with $2ML$ where $L = \max(b-a, d-c)$. If $e\lt f$ and $g\lt h$ are contained in the intervals $a:b$ and $c:d$ respectively, then the length of the product of the two intervals is bounded above by $2M\max(f-e, h-g)$. We can therefore see that the length of the product interval goes to 0 as the lengths of the subintervals go to 0. Notice that the bound for this does involve the original intervals, but once that initial choice is made, then we can scale the product length down to 0. This is the key property we are abstracting out in what follows. 

\subsection{Narrowing of Intervals}\label{sec:nrwint}

A \textbf{rational interval operator} $f$ is a mapping that operates on a finite number of rational intervals to produce a rational interval. An \textbf{oracle operator} $F$ is a mapping that operates on a finite number of oracles to produce an oracle.

We will use $|I|$ to denote the length of the intervals which for the rational interval $a\lte b$ is $b-a$. We shall consider $n$-tuples of intervals by which we mean an ordered collection of $n$ rational intervals. The $n$-tuple $\vec{J}$ is contained in the $n$-tuple $\vec{I}$ if the $i$-th interval in $\vec{J}$ is contained in the $i$-th interval of $\vec{I}$ for each $1 \leq i \leq n$. For our purposes, we will define the length of the $n$-tuple $\vec{I}$ as $|\vec{I}| = \max_{i=1}^n (|I_i|)$.

For an $n$-tuple $\vec{\alpha}$ of oracles, we say that the $n$-tuple $\vec{I}$ is an $\vec{\alpha}$-Yes tuple if for each of the $I_i$ intervals, we have $I_i$ is an $\alpha_i$-Yes interval. We say it is a neighborly $\vec{\alpha}$-Yes tuple if none of the endpoints are the roots of their respective $\alpha_i$ oracle. In particular, a neighborly tuple has no singletons. It has the property that there exists another $n$-tuple $\vec{J}$ that is an $\vec{\alpha}$-Yes tuple such that each $J_i$ is strictly contained in $I_i$. 

A $\vec{\alpha}$ of oracles is neighborly if all of its elements are neighborly oracles. It is rooted if they are all rooted oracles. It is partially rooted if someone of them are rooted oracles and others are neighborly. 

We use the terminology that $\vec{I}$ is in the domain of $f$ if $f(\vec{I})$ is defined and, specifically, it is not the empty set. We say that $\vec{\alpha}$ is $f$-compatible if there exists an $\vec{\alpha}$-Yes $n$-tuple of rational intervals $\vec{I}$ in the domain of $f$. 

An interval operator has the \textbf{narrowing property} if the following holds for all $n$ tuples $\vec{I}$ in the domain of $f$:
\begin{itemize}
    \item Containment. If $\vec{J}$ is contained in $\vec{I}$, then $f(\vec{J})$ is contained in $f(\vec{I})$. In particular, $f$ is defined on all tuples of rational intervals contained in $\vec{I}$. By the containment, we have $|f(\vec{J})| \leq |f(\vec{I})|$
    \item Notionally Shrinking. Given an $f$-compatible $\vec{\alpha}$, we can find, for any given $N>0$, an $\vec{\alpha}$-Yes neighborly $n$-tuple $\vec{J}$  such that $|f(\vec{J})| \leq \frac{1}{N}$. 
\end{itemize}

The Notionally Shrinking property requires the existence of neighborly $n$-tuples to help ensure appropriate continuity in various applications of this notion. This is a restriction on the rooted and partially rooted oracle tuples as the neighborly oracles automatically only have neighborly tuples. 

If $f$ satisfies the narrowing property and $\vec{I}$ is not in its domain, then all $n$-tuples that contain $\vec{I}$ are not in the domain. 

One common scenario where this holds is when the length of the image scales down with the length of the input. We will say that an interval operator $f$ has the \textbf{Lipschitz property} if it satisfies the Containment condition of the narrowing property and the additional fact that given an $\vec{I}$ such that $f(\vec{I})$ is defined, then there exists a constant $M_{\vec{I}}$ that satisfies for every $\vec{J}$ contained in $\vec{I}$, we have $|f(\vec{J})| \leq M_{\vec{I}} |\vec{J}|$. We additionally require that any $f$-compatible $n$-tuple that contains a singleton is contained in an $f$  compatible neighborly $n$-tuple. 

We do allow singletons in the $n$-tuples and, if they are all singletons, then the Lipschitz condition implies the image of a singleton tuple is a singleton as the length of a singleton is 0 and the length of the image is bounded by a multiple of the length. 

\begin{proposition}
    Let the rational interval operator $f$ be Lipschitz. Then it also satisfies the narrowing property. 
\end{proposition}

The arithmetic operators are Lipschitz and this is how we will establish that they have the narrowing property and, therefore, form a fonsi and an oracle. 

\begin{proof}
    By assumption, it satisfies containment. 

    To establish the notionally shrinking condition, let a rational $N >0$ and an $f$-compatible  $n$-tuple $\vec{\alpha}$ be given. Being $f$-compatible and Lipschitz, there exists a neighborly $\vec{\alpha}$-Yes $n$-tuple of intervals $\vec{I}$ for which $f(\vec{I})$ is defined. By the Lipschitz condition, we have a constant $M_{\vec{I}}$, which we will just call $M$. Because each $\alpha_i$ is an oracle, we can find, by the bisection algorithm of Proposition \ref{pr:short}, $\alpha_i$-Yes intervals $J_i$ whose length is less than $\frac{1}{MN}$ and $J_i \subseteq I_i$. Let $\vec{J}$ be the $n$-the tuple whose $i$-th component is $J_i$. This is an $\vec{\alpha}$-Yes interval by choice of $J_i$. By Lipschitz, $|f(\vec{J})| \leq M |\vec{J}| \leq M \frac{1}{MN} = \frac{1}{N}$ as was to be shown. 
\end{proof}

The intersection of two $n$-tuple intervals is the $n$-tuple formed by intersecting their respective components. They are disjoint if at least one of the intersecting components is disjoint. 

We now prove an essential property of narrowing operators. 

\begin{proposition} \label{pr:op-nrw}
Let $f$ be an operator on $n$-tuples with the narrowing property. If $\vec{I}$ and $\vec{J}$ intersect, then $f(\vec{I})=K$ and $f(\vec{J}) = L$ intersect.
\end{proposition}

This is equivalent to if $f(\vec{I})=K$ and $f(\vec{J}) = L$ are disjoint, then $\vec{I}$ and $\vec{J}$ are disjoint.

\begin{proof}
Assume $\vec{I}$ and $\vec{J}$ intersect. Let $\vec{A}$ be the intersection. Since $\vec{A}$ is contained in both, we have $f(\vec A)= B$ is contained in both $K$ and $L$ by the narrowing property. Therefore the intersection of $K$ and $L$ contains $B$ and is not empty.  
\end{proof}

\begin{theorem}\label{th:intop}
If $f$ is an interval operator with the narrowing property, then there is an associated oracle operator $F$. For any $\vec{\alpha}$ such that there is an $\vec{\alpha}$-Yes interval tuple $\vec{I}$ that is in the domain of 
$f$, then $\vec{\alpha}$ is considered to be in the domain of $F$ and the family of intervals of the form $f(\vec{I})$ for all $\vec{\alpha}$-Yes intervals $\vec{I}$ in the domain of $f$ is a fonsi. We write $F(\vec{\alpha})$ to denote the unique oracle associated with that fonsi. 
\end{theorem}

We could define an oracle directly, but the proof would mimic that which we did in establishing the oracle of a fonsi. Arithmetic operators inherently produce fonsis and only secondarily produce oracles. This seems to be an essential part of the difficulty of doing arithmetic with real numbers. 

For the relevance of domain, the operator to keep in mind is the reciprocal operator. It is undefined on the Oracle of 0 as any 0-Yes interval will contain 0 and thus the reciprocal interval operator is undefined on it. 

Also note that because of the narrowing property, once there is an $\vec{\alpha}$-Yes tuple for which $f$ is defined on, then all contained ones are also defined. Again, the reciprocal is a good one to keep in mind as say the $2$-Yes interval $-2:5$ is not defined for the reciprocal, but $1:3$ is defined (and it is $1:\frac{1}{3}$) and all subintervals of $1:3$ are defined as well, leading to the Oracle of $\frac{1}{2}$. One should also note that, say, the interval $-1:1$ is also a $\frac{1}{2}$-Yes interval, but not produce through reciprocation. That is, while the interval operator is exclusionary as one gets more expansive $n$-tuples, the oracle version is inclusionary of the larger intervals, being a maximalized fonsi. 

\begin{proof}
To prove that the family of intervals is a fonsi, we need to demonstrate that any two such intervals intersect and that given a rational $\varepsilon > 0$, we can find an interval in the family smaller than that. 

By Proposition \ref{pr:op-nrw}, since any pair of $\vec{\alpha}$-Yes tuples intersect, we must have that their images intersect. Thus, we do have that the family is an overlapping family. 

For the other part, the narrowing property's notionally shrinking condition is exactly what we need. Note that the $\varepsilon > 0$ is equivalent to requiring being smaller than $\frac{1}{N}$ for a given $N$.

For a given $\vec{\alpha}$ in the domain of $F$, we have a fonsi and, therefore, an associated oracle.
\end{proof}

\subsection{Arithmetic Operators}

We can now establish the arithmetic of oracles by defining the arithmetic operators via interval arithmetic and establishing that they have the narrowing property.

The interval arithmetic operators all have the property that subintervals of inputs lead to subintervals of the output as we were motivated to make sure all rationals in an interval were mapped into the output interval by the given operation. 

In what follows, $\alpha$ and $\beta$ are two oracles and we will use $A=a\lte b$ for a generic $\alpha$-Yes Interval and $B=c\lte d$ for  a $\beta$-Yes interval. For binary operators, we define $L = |A,B| = \max(b-a, d-c)$ while for a unary operator, we define $L = |A|$.

\begin{enumerate}
    \item $\alpha+\beta$ is based on the interval addition operator, namely $A \oplus B = a:b \oplus c:d = (a+c):(b+d)$ which has length $(b+d) - (a+c) = (b-a) + (d-c) \leq 2L$, establishing the narrowing property. 
    \item $\alpha * \beta$ is based on interval multiplication. If the interval endpoints are all the same sign, $a:b \otimes c:d$ is $|bd-ac| = |bd -bc + bc -ac| =  |b(d-c) + c(b-a)|$ though it is also equal to $|bd - ad + ad - ac| = |d(b-a) + a(d-c)|$. For mixed signs, we have the maximum of $|a(d-c)|$, $|b(d-c)|$, $|c(b-a)|$, $|d(b-a)|$, $|d(b-a)+a(d-c)| = |b(d-c) + c(b-a)|$. 
    
    For a simple bounding estimate on the multiplicative length, we can take the maximum $M$ of $|a|, |b|, |c|, |d|$ and multiply that by the maximum length $L$ of $b-a$ and $d-c$ and then double that. So $2*M*L$. This satisfies the narrowing property since this $M$ can bound all sub-interval length computations.
    \item $-\alpha$ is based on interval negation. Negation does not change the interval length so that the bound is $1*L$, establishing the narrowing property. 
    \item $\alpha - \beta$ is based on interval subtraction. Subtraction has length $b-c - (a-d) = (d-c) + (b-a)$, the same as addition. Thus, $2L$ is the bound to establish the narrowing property. 
    \item $\frac{1}{\alpha}$ is based on the reciprocity of intervals and $\alpha$ cannot be $0$. The length of a reciprocated interval is $\frac{1}{b} - \frac{1}{a} = \tfrac{b-a}{ab}$. Note $a$ and $b$ must be the same sign to avoid having 0 in there, which was part of the definition. Let $m = \min(|a|, |b|)$. Then $\tfrac{b-a}{ab} \leq \tfrac{L}{m^2}$. Letting $M = \tfrac{1}{m^2}$, this is a bound that holds for all subintervals since a given $a < p < q < b$ will have the property that $|p| > m$ and $|q| > m$. Hence, $\frac{q-p}{qp} \leq \frac{L}{m^2}$. This very much depends on $0$ not being in that interval. 
    \item $\tfrac{\alpha}{\beta}$ for $\beta \neq 0$ is defined based on multiplication and reciprocity. Using the multiplication bound on the reciprocated $\beta$, we have the bound as $M = \max(|a|, |b|, |\frac{1}{c}|, |\frac{1}{d}|)$ and $K= \max(b-a, \frac{1}{c} - \frac{1}{d}= \tfrac{d-c}{cd} )$ so that the length of the operated interval is bound by $2MK \leq 2*M^3*L$.
    \item Raising to a power is repeated multiplication, but we can be explicit here. For $\alpha^n$ for a given natural number $n>1$, and $\alpha$-Yes interval $a:b$, define $f = \max(|a|, |b|)$ and $e=\min(|a|,|b|)$.  We have two cases. If $a$ and $b$ are the same sign, then $a^n:b^n$ is the power interval and the length is $b^n - a^n = (b-a)\sum_{i=0}^{n-1} a^i b^{n-1-i} < (b-a)n f^{n-1}$. So a bound of $M= nf^{n-1}$ suffices for the narrowing property of same sign intervals. For differing signs, we have the power interval is $sf^n:-sef^{n-1}$ where $s$ is 1 or $-1$ depending on which one is the maximum and the parity of $n$. The difference is $f^{n-1} (f+e)$. Note that $f+e$ is $b-a$ as they have opposite signs. Therefore, our bound for the opposite signs is $f^{n-1} L$. For a single bounding constant, we can use $M = n f^{n-1}$.
    \item For negative powers, this is positive powers combined with reciprocity. Note that there is no 0 allowed which allows us to focus on the same sign cases. We start with $0 < a < b$ as the negative version changes only by a sign which vanishes in the computation of lengths.  Taking $a:b$ to be our enclosing $\alpha$ interval, then $\frac{1}{a^n} : \frac{1}{b^n}$ is our power interval and the difference is $\frac{b^n-a^n}{a^n b^n} \leq (b-a) \frac{n b^{n-1}}{a^n b^n} = \frac{n}{a^n b} L$. So taking $\frac{n}{a^n b}$ as our bounding constant $M$ gives us the narrowing property.  For raising an interval of negative numbers, we use absolute values noting that  $|a|>|b|$. The bound then becomes $\frac{nL}{|b|^n |a|}$.
\end{enumerate}

\subsubsection{\texorpdfstring{$\sqrt{2} \sqrt{3} = \sqrt{6}$}{sqrt2sqrt3eqsqrt6}}

We can prove arithmetic properties of square roots, for example. 

Let's prove that $\sqrt[n]{u} \sqrt[n]{v} = \sqrt[n]{uv}$  for rational $u$ and $v$. 

To do this, we need to multiply intervals representing the $n$-th roots and show that they become intervals of the $n$-th root of the product. This is almost self-evident. 

Indeed, let $a\lte b$ be a Yes-interval of $\sqrt[n]{u}$. We will assume the boundaries are all positive which is fine as we are narrowing on the oracle in question; Consistency does the rest of the work. That is, we have $a^n:u:b^n$.  Similar, let $c \lte d$ satisfy $c^n : v : d^n$. When we multiply these intervals, since they are all positive, we have $a^n c^n : uv : b^n d^n$ and we also have $(ac:bd)^n$ is $a^n c^n : b^n d^n$. Thus, $ac:bd$, the product of the $n$-th root Yes-intervals, is a Yes-interval of the $n$-th root of $uv$. The family of such product intervals form a fonsi as we can use arbitrarily narrow $a:b$ and $c:d$ intervals and multiplication is a Lipschitz operator. Since those intervals are Yes-intervals of the $n$-th root of the product, Corollary \ref{cor:sub-fonsi} establishes that the fonsi is associated with $\sqrt[n]{uv}$, as we were to show. 

For some mathematical history related to this, see \cite{fowler}.

\subsubsection{\texorpdfstring{$\sqrt{2}*(e + \pi)$}{2epi}}

Let us demonstrate this arithmetic by computing $x = \sqrt{2}(e + \pi)$. 

What we are asserting that this computation yields is an oracle that says yes or no as to whether the number is contained in a given interval. Let's show how we would compute that out given a few intervals:  $8:9$, $8.1:8.2$, and $8.286:8.288$.

We start off with some rather large interval computation: $(1 \lte 2) \otimes ( (2\lte  3) \oplus (3 \lte 4))$. We do the indicated operations to get the $x$-Yes interval $1*(2+3) \lte 2*(3+4) = 5 \lte 14$. This fully includes our intervals and thus provides no answer about them. But it would tell us, for example, that $20 \lte  21$ is a $x$-No interval. 

We increase our intervals by a decimal place to compute $1.4 \lte  1.5 \otimes ( 2.7\lte 2.8 \oplus 3.1 \lte 3.2)$. This leads us to the $x$-Yes interval $8.12 \lte 9$ which tells us $8:9$ is an $x$-Yes interval, but does not give us enough clarity for the next one. 

For that one, we will continue onto the next decimal place: $1.41 \lte 1.42 \otimes (2.71\lte 2.72 \oplus 3.14 \lte   3.15)$ leading to $8.2485 \lte  8.354$. This rules out $8.1:8.2$, so that is an $x$-No interval. But we still need to see about $8.286:8.288$.

For that, we can use error formulas. For the square root of 2, we use Newton's method which will roughly have the error quadratically decrease.  For $e$, we have that the sum error is $s(n) = \frac{1}{n!n}$. And for $\pi$, we have $t(n) = \frac{1}{15 * 16^n}$. We also have for addition, that the error is bounded by twice the largest error. For multiplication, we have that but also times the maximum of the endpoints of the intervals. We want the interval length to be less than $\frac{1}{1,000}$. 

Doing some computations, we have $s(7) < \frac{1}{35,000}$, and $t(3) < \frac{1}{60,000}$. For the square root of 2, we start with $\frac{141}{100}$, and compute out $\frac{56400}{39881} \lte \frac{39881}{28200}$ whose difference is about $\frac{1}{130,000}$.

For multiplication, the bound $M$ can be taken as 6 since that is an upper bound on adding $e$ and $\pi$ and that addition is what dominates the endpoint estimates. So our error estimate is $2*6*\max(\frac{1}{130,000}, 2*\max(\frac{1}{35,000}\frac{1}{60,000})) \leq \frac{12}{17,000} \leq \frac{1}{1000}$. We now compute the intervals.  
\begin{itemize}
    \item $\pi$-Yes interval $3.141592:3.141609$
    \item $e$-Yes interval $2.71827:2.718294$
    \item $e + \pi$-Yes interval $5.859862:5.859903$
    \item $\sqrt{2}$-Yes interval $1.41420:1.41422$
    \item $\sqrt{2}(e + \pi)$-Yes interval  $8.2870:8.28719$
\end{itemize}

Since the interval is contained in $8.286:8.288$, we conclude from consistency that $8.286:8.288$ is a $x$-Yes interval. 

Admittedly, the exercise of asking if something is a Yes interval is less useful when answering that question involves computing other intervals. We typically would just compute the intervals directly in that case, looking for a given error tolerance. But this is an exercise in showing what the oracle definition does involve.

\subsection{Fielding the Oracles}

We have already established that arithmetic combinations of oracles make sense and leads to oracles.

Now we can establish that this arithmetic of oracles is a field. We will mostly rely on the established rules of interval arithmetic. To do that, we need to make the link between interval rules and oracle rules.

\begin{proposition}
Let $f$ and $g$ be interval functions with the narrowing property such that $f(\vec{I}) \supseteq g(\vec{I})$ for all interval $n$-tuples $\vec{I}$ in the common domain, which we assume to be the same. Then the corresponding oracle functions, $F$ and $G$, are equal.  
\end{proposition}

\begin{proof}
Let $\vec{\alpha}$ be given in the domain of $F$ and $G$; they have the same domain as they are defined on the same intervals by assumption. We need to show that $\beta = F(\vec{\alpha})$ is equal to $\gamma =  G(\vec{\alpha})$. Let $A$ and $B$ be $\beta$-Yes and $\gamma$-Yes intervals, respectively. Then there exists intervals $I \subseteq A$ and  $J \subseteq B$  such that $f(\vec{I}) = I$ and $g(\vec{J}) = J$ with both  $\vec{I}$ and $\vec{J}$ being $\vec{\alpha}$-Yes $n$-tuple intervals. As we have noted before, $\vec{\alpha}$ $n$-tuples intersect with each other leading to the existence of an $\vec{\alpha}$-Yes $n$-tuple $\vec{K}$ contained in both $\vec{I}$ and $\vec{J}$ with the property that $f(\vec{K})$ and $g(\vec{K})$ are defined and contained in $I$ and $J$, respectively, by the narrowing property. By assumption, $I \supseteq f(\vec{K}) \supseteq g(\vec{K})$. Thus, the intersection of $I$ and $J$ contains the non-empty $g(\vec{K})$. Thus $A$ and $B$ have a non-empty intersection. Since $A$ and $B$ were arbitrary Yes intervals of the two oracles, Proposition \ref{pr:overlap} tells us that $\beta = \gamma$. Since $\vec{\alpha}$ was arbitrary, $F = G$, as was to be shown.
\end{proof}

We apply this proposition for five of the field properties and the conclusions from Section \ref{sec:rules} as follows: 

\begin{itemize}
    \item Addition is Commutative:  Use the interval equation $I_1 \oplus I_2 = I_2 \oplus I_1$. As one written out example, we can define $f$ as $f(I_1, I_2) = I_1 \oplus I_2$ and $g(I_1, I_2) = I_2 \oplus I_1$. Then the equality tells us that $f(\vec{I}) \supseteq g(\vec{I})$ for all $I$. Both $g$ and $f$ have the Narrowing Property as established before. Thus the proposition above tells us that they are equal. The other cases below flow the same way. 
    \item Addition is Associative: Use the interval equation $I_1 \oplus (I_2 \oplus I_3) = (I_1 \oplus I_2) \oplus I_3$.
    \item Multiplication is Commutative: Use the interval equation $I_1 \otimes I_2 = I_2 \otimes I_1$.
    \item Multiplication is Associative: Use the interval equation $I_1 \otimes (I_2 \otimes I_3) = (I_1 \otimes I_2) \otimes I_3$.
    \item Distributive Property: Use the interval containment $I_1 \otimes (I_2 \oplus I_3) \subseteq (I_1 \otimes I_2) \oplus (I_1 \otimes I_3)$.
\end{itemize}

We need to now establish the identity and inverse properties. 

\begin{itemize}
    \item The additive identity is the Oracle of the singleton $0:0$. This follows from the interval equation $I \oplus 0:0 = I$.
    \item The multiplicative identity is the Oracle of the singleton $1:1$. This follows from the interval equation $I \otimes 1:1 = I$.
    \item The additive inverse for an oracle $\alpha$, denoted $-\alpha$, is defined by the rule that $a:b$ is a $-\alpha$-Yes interval exactly when $-a:-b$ is an $\alpha$-Yes interval.  It follows that $-(-\alpha) = \alpha$ since $-(-a:-b) = a:b$. We need to establish that this is an oracle and that it does indeed satisfy $\alpha + -\alpha = 0$. 
    
    Let $R$ be the rule for $-\alpha$ and $S$ be the rule for $\alpha$. Therefore, $R(a:b)=S(-a:-b)$ and $S(a:b) = R(-a:-b)$ for all intervals $a:b$.
    
    \begin{enumerate}
        \item Consistency. Let $c:d \supseteq a:b$ and $R(a:b)=1$ be given. Therefore $S(-a:-b) = 1$ and $-c:-d \supseteq -a:-b$ so $S(-c:-d)=1$ by consistency of $S$ implying $R(c:d) = 1$
        \item Existence. By the existence property for the oracle $\alpha$, there exists $a:b$ such that $S(a:b)=1$. Thus, $R(-a:-b)=1$ and we have existence. 
        \item Separating. Given $R(a:b)=1$ and $a:c:b$, then we have $S(-a:-b)=1$ and $-a:-c:-b$. If $S(-c:-c)=1$, then $R(c:c)=1$ and we are done. If not, then by the separating property for $\alpha$, $S(-a:-c) \neq S(-c:-b)$ which then implies $R(a:c)\neq R(c:b)$ as was to be shown. 
        \item Rooted. Assume $c$ and $d$ satisfy $R(c:c)=R(d:d)=1$. Then $S(-c:-c)=S(-d:-d)=1$. By the rooted property for $\alpha$, $-c = -d$ and thus $c=d$ as was to be shown. 
        \item Closed. Assume $c$ is contained in all $R$-Yes intervals. Then $-c$ is contained in all $S$-Yes intervals. As $\alpha$ satisfies being closed, $S(-c:-c)=1$ which implies $R(c:c)=1$.
    \end{enumerate}
    
    To see this is the inverse, we start by considering the general $\alpha$-Yes interval $a:b$ and $-\alpha$-Yes interval $c:d$. We need to add them together. The interval $a:b$ generates the $-\alpha$-Yes interval $-a:-b$. We can then take the intersection of $c:d$ and $-a:-b$ which exists and is a $-\alpha$-Yes interval; let's call it $e:f$. We also have $-e:-f$ is an $\alpha$-Yes interval with $-e:-f$ being an interval in $a:b$. When we add them, we get $f-e:e-f$ and $0$ is clearly contained in that interval. Since $e:f$ is contained in $c:d$ and $-e:-f$ is contained in $a:b$, then $0$ is also in the sum of $a:b$ with $-c:-d$, i.e.,  $0 \in (e:f) \oplus (-e:-f) \subseteq c:d \oplus a:b$. Since $a:b$ and $c:d$ were arbitrary Yes intervals of their respective oracles, we have that 0 is contained in every summed interval. As the sum is an oracle, this must be the Oracle of $0$.
    
    \item The multiplicative inverse for an oracle $\alpha$ is similar, but has the complication that we can only consider reciprocating intervals that do not include 0. 
    
    The multiplicative inverse for an oracle $\alpha$, denoted $\frac{1}{\alpha} = \alpha^{-1}$, is defined, for $\alpha \neq 0$, by the rule that $a:b$ is a $\frac{1}{\alpha}$-Yes interval if either 
    
    \begin{enumerate}
        \item  $a*b > 0$  and $\frac{1}{a}:\frac{1}{b}$ is an $\alpha$-Yes interval; we call this a same-sign $\alpha$-Yes interval.
        \item If $a*b<0$, then it is only a $\frac{1}{\alpha}$-Yes interval if it contains a same-sign $\frac{1}{\alpha}$-Yes interval. 
    \end{enumerate}
    
    We do have $\frac{1}{\frac{1}{\alpha}} = \alpha$ as $1\oslash (\frac{1}{a}: \frac{1}{b}) = a:b$ We need to establish that this is an oracle and that it does indeed satisfy $\alpha * \frac{1}{\alpha} = 1$.
    
     Let $R$ be the rule for $\frac{1}{\alpha}$ and $S$ be the rule for $\alpha$. Therefore, $R(a:b)=S(\frac{1}{a}:\frac{1}{b})$ and $S(a:b) = R(\frac{1}{a}:\frac{1}{b})$ for all same-sign intervals $a:b$.
    
    In what follows, we will assume the intervals are the same-sign except for the few places where we need to remark on the alternative strategy. One key fact is that if $a<p<b$ and all three are the same sign, then $\frac{1}{a} > \frac{1}{p} > \frac{1}{b}$.\footnote{Since we are using rationals, a quick proof starts with the rational form $\frac{p}{q} < \frac{r}{s}$, $q, s > 0$, multiplying by $qs$ to get $ps < qr$, and then dividing by $pr$ to get the reciprocated and flipped inequality $\frac{s}{r} < \frac{q}{p}$, assuming that $pr$ is positive, i.e., the rationals are of the same sign. If they are of opposite sign, then for the points in between, we have $a < p < 0< q< b$ leading to $\frac{1}{p} < \frac{1}{a} < 0 < \frac{1}{b} < \frac{1}{q}$ which means an interval gets mapped to two infinite intervals instead of one finite one.} That is, $a\lte p\lte b$ if and only if $\frac{1}{b} \lte \frac{1}{p}\lte \frac{1}{a}$.
    
    \begin{enumerate}
        \item Consistency. Assume $c: a: b: d $ with $R(a:b)=1$. Then $S(\frac{1}{a}:\frac{1}{b}) = 1$ and $\frac{1}{c}:\frac{1}{a}:\frac{1}{b}\frac{1}{d}$ so $S(\frac{1}{c}:\frac{1}{d})=1$ by consistency of $S$. Thus, $R(c:d) = 1$ by definition. If $a:b$ was an opposite signed interval which is an $\alpha$-Yes interval, then there is a same-sign interval $e:f$ contained in $a:b$ which is $\alpha$-Yes. Thus any interval $c:d$ that contains $a:b$ will also contain $e:f$ and and thus will also be $\alpha$-Yes by definition. 
        \item Existence. Since $\alpha$ is an oracle and not zero, there exists an interval $a:b$ such that $S(a:b)=1$ and 0 is not between $a$ and $b$; they are therefore the same sign. Thus, $R(\frac{1}{a}:\frac{1}{b})=1$ and we have existence. 
        \item Separating. Given $R(a:b)=1$ and $a:c:b$, then we have $S(\frac{1}{a}:\frac{1}{b})=1$ and $\frac{1}{a}:\frac{1}{c}:\frac{1}{b}$. If $S(\frac{1}{c}:\frac{1}{c})=1$, then $R(c:c)=1$ and we are done. If not, then, by the separating property for $\alpha$,  $S(\frac{1}{a}:\frac{1}{c}) \neq S(\frac{1}{c}:\frac{1}{b})$ which then implies $R(a:c)\neq R(c:b)$ as was to be shown. 
        
        The above was in the case of $a:b$ being of the same sign. If $a:b$ was not a same sign interval, then it contains a same sign Yes-interval, say, $a:e:f:b$. If $c$ is outside $e:f$, then it divides $a:b$ into $a:c$ and $c:b$ and whichever one contains $e:f$ is a Yes interval while the other is a No interval. The other case is that $e:c:f$. In the same-sign case that we established above, we have that $c$ separates $e:f$ implying that we can then extend that up to $a:e:c$ and $c:f:b$.
        \item Rooted. Assume $c$ and $d$ satisfy $R(c:c)=R(d:d)=1$. Then $S(\frac{1}{c}:\frac{1}{c})=S(\frac{1}{d}:\frac{1}{d})=1$. By the rooted property for $\alpha$, $\frac{1}{c} = \frac{1}{d}$ and thus $c=d$ as was to be shown. 
        \item Closed. Assume $c$ is contained in all $R$-Yes intervals. Then $\frac{1}{c}$ is contained in all $S$-Yes intervals. As $\alpha$ satisfies being closed, $S(\frac{1}{c}:\frac{1}{c})=1$ which implies $R(c:c)=1$.
    \end{enumerate}
    
    To see this is the multiplicative inverse, we start by considering the general $\alpha$-Yes interval $a:b$ and $\frac{1}{\alpha}$-Yes interval $c:d$ and assume both are same-signed since we can restrict our attention to subintervals and $\alpha$ is not the 0 oracle. We need to multiply them together. The interval $a:b$ generates the $\frac{1}{\alpha}$-Yes interval $\frac{1}{a}:\frac{1}{b}$. We can then take the intersection of $c:d$ and $\frac{1}{a}:\frac{1}{b}$ which exists and is a $\frac{1}{\alpha}$-Yes interval; let's call it $e\lte f$. We also have $\frac{1}{f} \lte \frac{1}{e}$ is an $\alpha$-Yes interval which is contained in $a:b$. When we multiply them, we get $\frac{e}{f}\lte \frac{f}{e}$ and $1$ is  contained in that interval.\footnote{If $\frac{e}{f} <1$, then multiplying by $\frac{f}{e}$ leads to $1 = \frac{fe}{ef} < \frac{f}{e}$. Similarly for $\frac{e}{f} > 1$. Note $e$ and $f$ are the same sign so that $\frac{f}{e}$ and $\frac{e}{f}$ are greater than 0. } Since $e:f$ is contained in $c:d$ and $\frac{1}{e}:\frac{1}{f}$ is contained in $a:b$, then $1$ is also in the product of $a:b$ with $\frac{1}{c}:\frac{1}{d}$. Since $a:b$ and $c:d$ were arbitrary same-signed Yes intervals of their respective oracles, we have that $1$ is contained in every multiplied interval; the cross-signed intervals will contain the same-signed ones which implies their products will also contain $1$. As the product is an oracle, this must be the Oracle of $1$. 
    
\end{itemize}

We have proven that

\begin{theorem}
The set of all oracles, with the above defined arithmetic, is a field. 
\end{theorem}

\subsection{Ordering the Field of Oracles}

We now turn to how the ordering fits with the field. We have already defined the ordering, with the ordering largely coming down to $R < S$ if there are $R$ and $S$ Yes intervals $a\lte b$ and $c\lte d$, respectively, such that $b<c$. 

Is this a total ordering? As we discussed, there are oracles, such as the Collatz Oracle, that we cannot distinguish from 0 and yet also not establish that it is 0 at the current time. Further, there are oracles, such as the coin toss ones we described, that, in principle, we could only distinguish from all other oracles by engaging in an infinite process. This is true of all definitions of real numbers and it should be kept in mind at least as a practical consequence. We will, however, ignore that and generally settle on that if we had sufficient abilities, then this would be a total ordering. We did establish that at most one of $<$, $>$, $=$ can hold true. 

We have that the inequality of oracles is transitive from Proposition \ref{pr:transitive} and we have already stated that equality of oracles is reflexive, symmetric, and transitive from Proposition \ref{pr:reflexive}. 

We now need to establish that the ordering plays correctly with the field operations. 

\begin{proposition}\label{pr:addinq}
 $\beta + \alpha< \gamma + \alpha$ for all oracles $\alpha$, $\beta$, and $\gamma$ such that $\beta < \gamma$.
\end{proposition}

\begin{proof}
Since $\beta < \gamma$, there exists $a\lte b< c\lte d$ such that $a:b$ is a $\beta$-Yes interval and $c:d$ is a $\gamma$-Yes interval. By the bisection method, we can choose an $\alpha$-Yes interval $e\lte f$ such that $ (f-e) <  (c-b)$. Then we assert that $a:b \oplus e:f  < c:d \oplus e:f$. Indeed, $a \lte b \oplus e \lte f = [(a+e) \lte (b+f)]$ and $c \lte d \oplus e \lte f = [(c+e)\lte (d+f)]$. So we need to show that $b+f < c+e$ but this is just $f-e < c-b$ rewritten. We have therefore produced an interval of $\beta+\alpha$ that is strictly less than $\gamma+\alpha$. 
\end{proof}

\begin{proposition}
$0 < \alpha*\beta$ for all oracles $0 < \alpha$ and $0< \beta$ 
\end{proposition}

\begin{proof}
Let $0 < a \leq b$ and $0<c \leq d$ such that $a\lte b$ is an $\alpha$-Yes interval and $c\lte d$ is a $\beta$-Yes interval. Then the product interval $a \lte b \otimes c \lte d = ac \lte bd$ is an $\alpha*\beta$-Yes interval and we have $0 < ac \leq bd$. Thus, $0:0 < ac \lte bd$ which implies $0 < \alpha*\beta$. 
\end{proof}

We have proven that

\begin{theorem}
The field of oracles is an ordered field.
\end{theorem}

\subsection{The Reals}

The rationals are an embedded field inside of the field of oracles. The embedding map is $q \mapsto q:q$. It is obviously a bijection from the rationals to the singleton oracles. It is also clear that $p<q$ if and only if $p:p < q:q$. For addition, we have $(p:p) \oplus (q:q) = (p+q):(p+q)$ which establishes that the bijection respects addition. For multiplication, we also have $(p:p) \otimes (q:q) = (pq):(pq)$ which establishes that the bijection respects multiplication.  Finally, $p>0$ if and only if $p:p > 0:0$ establishing that the positive rational are positive oracles. 

\begin{theorem}
The ordered field of oracles contains the rationals as a subfield. 
\end{theorem}

We previously established the completeness properties:

\begin{itemize}
    \item  Theorem \ref{th:lub}: All sets of oracles with an upper bound have a least upper bound. 
    \item Theorem \ref{th:cauchy}: All Cauchy sequences of oracles converge to an oracle. 
\end{itemize}

\begin{theorem}
The ordered field of oracles satisfies the axioms of the real numbers and can be considered the real number field. 
\end{theorem}

\subsection{Exploring with Oracles}

Here we do a few explorations using oracles directly. Since we have established that the set of oracles with the defined addition, multiplication, and ordering satisfies the axiomatic definition of the reals, we could simply import all of the content of real analysis at this point. 

But that would seem to undercut the point of caring about this approach. We contend that by embracing oracles as fundamental, we can often get a more constructive approach than when we rely on the axiomatic approach. 

\subsubsection{Density of Rationals in the Oracles}

Let us first establish the density of the rationals in the oracles. This is essentially immediate from the definition of being less than. Indeed, let two oracles be given such that $\alpha < \beta$. Then there exist intervals $a\lte b$ and $c\lte d$ which are $\alpha$ and $\beta$-Yes intervals, respectively, and that satisfies $b < c$; this is the meaning of the inequality of oracles. Since $b < \frac{b+c}{2} < c$, the rational singleton $\gamma = \frac{b+c}{2}:\frac{b+c}{2}$ satisfies $\alpha < \gamma < \beta$. Thus, the rational oracles are dense in the real field of oracles. 

The irrationals are also dense in the rationals. Indeed, let $\frac{p}{q} <\frac{r}{s}$ be given. In the upcoming Section \ref{sec:mediant}, we define a process using mediants to choose subintervals to explore an oracle. Here, we use that process to define an irrational oracle. At each step, alternate which of the two mediant subintervals one chooses. Because we are using mediants and do not have a fixed endpoint on the intervals, this process will produce an irrational which is strictly between the two given rationals. 

\subsubsection{Archimedean Property}

Let us now establish the Archimedean property, namely, that for any oracle $\alpha > 0$ and for any oracle $\beta$, we have the existence of an integer $n$ such that $n * \alpha > \beta $. If $\alpha > \beta$, we are done. If $\alpha = \beta$, then $2 \alpha > \beta$. Otherwise, $\alpha < \beta$ and we can take $a\lte b$ to be an $\alpha$-Yes interval and $c \lte d$ to be a $\beta$-Yes interval such that $b < c$. Let $a = \frac{p}{q}$ and $d = \frac{r}{q}$, where we can take the same denominator by scaling to common denominators if necessary. Since we have $a<d$, we must have $p < r$. By the division algorithm, we know that there exists integers $r = n*p + t$ where $0 \leq  t<p$ and thus $(n+1)p > r$. So $(n+1):(n+1) \otimes a\lte b > c\lte d$ establishing that $(n+1)*\alpha > \beta$.

\subsubsection{Uncountability}

A famous part of the lore of real numbers is that they are uncountable unlike the rationals. We will first establish a result on finite collections of reals, namely, that we can always find an oracle distinct from any given finite collection of them. With that established, one can then use the machinery of infinity and contradiction to establish the uncountability of the reals. 

\begin{proposition}\label{pr:notlist}
Given a finite indexed collection of oracles, say $\{a_i\}_{i=0}^n$, we can produce an indexed collection of nested intervals, $\{I_i\}_{i=0}^n$ such that each $a_j$ is not an element of of $I_i$ whenever $j \leq i$.
\end{proposition}

\begin{proof}
Start with $0:1$. If $0:1$ is an $a_0$-No interval, then $a_0$ is not in the interval and we define $I_0 = 0:1$. If $0:1$ is an $a_0$-Yes interval, then consider the intervals $L= 0:\tfrac{1}{2}$ and $R = \tfrac{1}{2}:1$. By separation, either $L$ is $a_0$-No or $R$ is $a_0$-No or $a_0 = \tfrac{1}{2}$. In the first two cases, we take $I_0=L$ or $I_0=R$, respectively. In the third case, we take $I_0 = 0:\tfrac{1}{4}$ which does not contain $\tfrac{1}{2}$ and thus is a No interval for $a_0$.

Now we repeat this, starting with $I_i$ to produce $I_{i+1}$. Let $L_i$ be the left bisection of $I_i$ and let $R_i$ be the right bisection of $I_i$. Let $b_i$ be the bisection point. The one difference is that if $I_i$ is an $a_{i+1}$ No interval, then we take $I_{i+1} = L_i$ as we do want to make sure the intervals narrow. If $I_i$ is an $a_{i+1}$ Yes interval, then we use separation and either $L_i$ or $R_i$ is a No interval, in which case we set $I_{i+1} = L_i $ or $I_{i+1} = R_{i+1}$, respectively. If both are yes, then $a_{i+1} = b_i$ and we then bisect $L_i$ and take the left interval to be $I_{i+1}$ which will be an $a_{i+1}$ No interval. 

By construction, every prior number is excluded from the interval. We do this up through $a_n$.
\end{proof} 

With that proposition established, let us move on to the infinite case. Given an infinite sequence of oracles, we construct a fonsi as above. Indeed, let $\{a_i\}_{i=0}^{\infty}$ and then define $\{I_i\}_{i=0}^{\infty}$ to be the $i$-th interval of $\{I_i\}_{i=0}^{n}$ as provided by Proposition \ref{pr:notlist}. As long as $i \leq n$, this will be defined and consistent across $n$. They all intersect and they get arbitrarily small. They thus define a fonsi which leads to an oracle as explained in Section \ref{sec:ni}. Since, by construction, each of the intervals is a No interval for its matching $a_i$, we have that the oracle cannot be any real number on the list. 

The key distinguishment from the rationals is that we are constructing a valid oracle. If we attempted something similar with rational numbers, we would end up with something which is not a rational number. One could apply this construction to the countable list of rationals provided by Cantor's argument. By necessity, this will turn out to not be a rational number. 

\section{Approximations}\label{sec:mediant}

This section is an exploration of what this framing of real numbers suggests for the computation of a real number. We will explore the use of mediants, continued fractions, Newton's method, a geometric computation of $\sqrt{2}$, and finish by computing mediant approximations of $\pi$ with the oracle  rule associated with solving $\sin(x)=0$ on the interval $[3,4]$ per Section \ref{sec:ivt}.

An alternative to the bisection method is the mediant method. Given a rational interval $\frac{a}{b} : \frac{c}{d}$, we divide the interval into two pieces by using the mediant $\frac{a+c}{b+d}$ as the ``middle''. This does depend on the particular fraction representative of the rational endpoints. Indeed, by scaling the two rational numbers, we can use the mediant operation to produce any rational number in the interval. If we repeatedly scale the fractions to have common denominators, then the mediant operation is the same as bisecting. 

The iterative procedure is to start with an interval\footnote{In what might be considered a canonical process, we would start with deciding on the interval $[\frac{-1}{0}, \frac{0}{1}]$ vs $[\frac{0}{1}, \frac{1}{0}]$ to determine the sign and then repeated application of the process gets the integral portion. In practice, we can just try to get to the intergral interval containing the real number as fast as we can and then start the process.}, compute a mediant, ask the oracle which interval contains it, and then repeat with that new interval. This produces the best rational approximations to the real number in the sense that there is no closer rational number to the target number whose denominator is no bigger than the mediant's denominator produced by this method.

\subsection{Square Roots}

Let's apply this to computing square roots. We will start with the square root of 2. The way we compute the oracle's answer is by computing the squares of the endpoints and seeing if the lower end is less than 2 and the upper one is greater than 2.\footnote{Given $\frac{p}{q}$, we would test it as $p^2 \lge 2q^2$. For example, $\frac{3}{2} > \sqrt{2}$ because $9 > 4*2 = 8$ which can also be a comparison to 0 by computing $p^2 - q^2 \lge 0$ which in the example is $9-8=1 > 0$.}

We start our procedure with the formal interval of $\tfrac{0}{1}$ and $\tfrac{1}{0}$. For each item, we record whether we choose the left subinterval [L] or the right subinterval [R].

\begin{itemize}
    \item[R] $\tfrac{0+1}{1+0} = \tfrac{1}{1}$, squared: $1 < 2$,  $[\tfrac{1}{1},\tfrac{1}{0}] $
    \item[L] $\tfrac{1+1}{1+0} = \tfrac{2}{1}$, squared: $4 > 2$, $[\tfrac{1}{1},\tfrac{2}{1}]$
    \item[L] $\tfrac{1+2}{1+1} = \tfrac{3}{2}$, squared: $\tfrac{9}{4} > 2$, $[\tfrac{1}{1},\tfrac{3}{2}]$
    \item[R] $\tfrac{1+3}{2+1} = \tfrac{4}{3}$, squared: $\tfrac{16}{9} < 2$, $[\tfrac{4}{3},\tfrac{3}{2}]$
    \item[R] $\tfrac{4+3}{3+2} = \tfrac{7}{5}$, squared: $\tfrac{49}{25} < 2$, 
    $[\tfrac{7}{5},\tfrac{3}{2}]$
    \item[L] $\tfrac{7+3}{5+2} = \tfrac{10}{7}$, squared: $\tfrac{100}{49} > 2$, 
    $[\tfrac{7}{5},\tfrac{10}{7}]$
    \item[L] $\tfrac{7+10}{5+7} = \tfrac{17}{12}$, squared: $\tfrac{289}{144} > 2$, 
    $[\tfrac{7}{5},\tfrac{17}{12}]$
    \item[R] $\tfrac{7+17}{5+12} = \tfrac{24}{17}$, squared: $\tfrac{576}{289} < 2$, 
    $[\tfrac{24}{17},\tfrac{17}{12}]$
    \item[R] $\tfrac{24+17}{17+12} = \tfrac{41}{29}$, squared: $\tfrac{1681}{841} < 2$, 
    $[\tfrac{41}{29},\tfrac{17}{12}]$
    \item[L] $\ldots$
\end{itemize}

There is a pattern of R,2L,2R,2L,2R,$\ldots$ which continues with alternating the directional selection every two times. We will explain this in just a moment, but for now, this means that we can skip checking the squaring as well as skip the intermediate steps. Indeed, for the next two steps, we are effectively doing $\tfrac{2*41 + 17}{2*29 + 12} = \tfrac{99}{70}$ as we will be replacing the right endpoint twice, implying the left one gets reused. That is, we have $[\tfrac{41}{29}, \tfrac{99}{70}]$ for our containment of the square root of 2. And then we can do the same thing to figure out our next interval after two steps is $[\tfrac{41+2*99}{29+2*70}=\tfrac{239}{169}, \tfrac{99}{70}]$. We will record the pattern as $[1; \bar{2}]$.

For the square root of 11, we would find the pattern to be dictated by $[3;\overline{3,6}]$. So we would start with $[\tfrac{3}{1}, \tfrac{1}{0}]$ after choosing R 3 times from our starting 0 and $\infty$ representatives. So we then could do $[\tfrac{3}{1}, \tfrac{3*3 + 1}{3*1 + 0} = \tfrac{10}{3}]$. Then we use the 6 to compute a new endpoint:  $[\tfrac{3+6*10}{1+6*3} = \tfrac{63}{19}, \tfrac{10}{3}]$. Next we do 3 for the right:  $[\tfrac{63}{19}, \tfrac{3*63+10}{3*19+3} = \tfrac{199}{60}]$. The decimal value of $\tfrac{199}{60}$ is about $3.31667$ while the square root of 11 is about $3.31662$. If we know the pattern, this becomes quite easy to compute. 

We should also add that one nice feature of the mediant approximation is that if the number being approximated is a rational number, then this procedure will produce it. The details may be found in Appendix \ref{app:med}. This is not generally true of other methods. For example, the bisection method has the range of denominators being produced to be a multiple of the starting denominators and powers of 2. 

As an example, let's consider the square root of $\frac{4}{9}$. The solution is obviously $\frac{2}{3}$. With our mediant method, the intervals becomes $\frac{0}{1}:\frac{1}{1}$, then $\frac{1}{2}:\frac{1}{1}$, $\frac{2}{3}:\frac{2}{3}$ and we are done. In contrast, the bisection method gives $0:1$, $\frac{1}{2}:1$, $\frac{1}{2}:\frac{3}{4}$, $\frac{5}{8}:\frac{3}{4}$, and so on down the powers of 2 approximations to $\frac{2}{3}$. 

\subsection{Continued Fractions and Best Approximations}

The pattern is that of continued fractions. For square roots, the continued fraction is periodic which is quite nice to handle. For other computations, such as cube roots, the continued fraction is not periodic, but the alternation in computing the intervals still applies. For example, the cube root of 11 has continued fraction [2; 4, 2, 6, 1, 1, 2, 1, 2, 9, 88, ...]. We compute the intervals as:  
\begin{enumerate}
\item Default start: $[\frac{0}{1}, \frac{1}{0}]$, 
\item 2R: $[\frac{0+1*2}{1+0*2}, \frac{1}{0}]$, 
\item 4L:$[\frac{2}{1}, \frac{4*2+1}{4*1+0} =\frac{9}{4}]$, 
\item 2R:$[\frac{2+9*2}{1+4*2} = \frac{20}{9}, \frac{9}{4}]$
\item 6L:$[\frac{20}{9}, \frac{6*20+9}{6*9 + 4} = \frac{129}{58}]$
\item 1R: $[\frac{149}{67}, \frac{129}{58}]$
\item 1L: $[\frac{149}{67}, \frac{278}{125}]$
\item 2R: $[\frac{705}{317}, \frac{278}{125}]$
\item 1L: $[\frac{705}{317}, \frac{983}{442}]$
\item 2R: $[\frac{2671}{1201}, \frac{983}{442}]$
\item 9L: $[\frac{2671}{1201}, \frac{2671*9+983}{1201*9+442} = \frac{25022}{11251}]$
\item 88R: $[\frac{2671+88*25022}{1201+88*11251} = \frac{2204607}{991289}, \frac{25022}{11251}]$
\end{enumerate}

The last mediant computed already has exhausted a typical calculator precision for the cube root of 11, being about $3.7\times 10^{-13}$. It is tempting with this method to think that we could add 1 to the numerator to the more precise endpoint and get a better interval. That is not the case. For example, the quantity $\frac{25022}{11251} - 11^{1/3} \approx 8.9\times 10^{-11}$ while $\frac{2204608}{991289} - 11^{1/3} \approx 1\times10^{-6}$. The length of the interval in the last step is about $1\times10^{-11}$.\footnote{This process preserves $ad -bc$ and this is the numerator of the difference for any of the descendant intervals. Here, that difference numerator is 1. The denominator for the difference is then grown roughly by the product of $b$ and $d$. Here the denominators are about a million and a hundred thousand leading to a difference on the order of a hundred billionth.} 

If we have the continued fraction representation of a number, then we can produce as short of an interval as we please. If we have the oracle, but not the continued fraction, we can compute the continued fraction by doing the mediant process and keeping track of when we alternate from one side to the other. 

The switching from replacing one side with the other corresponds to a switch of how good of an approximation it is. The next three paragraphs are based on the excellent ``Continued Fractions without Tears'' \cite{richards}, exact statements being in quotes.  

We have the following theorem: ``Take any irrational number $\alpha$, with $0 < \alpha < 1$. The slow continued fraction algorithm (the Farey process, zeroed in on $\alpha$) gives a sequence of best left and right approximations to $\alpha$. Every best left/right approximation arises in this way.'' The slow process is taking the mediant at each step, using the oracle to decide which interval applies. The best approximation up to a given denominator $n$ is defined as the rational number $\frac{p}{q}$, $q  \leq n$, that satisfies $|\frac{p}{q} - \alpha | < |\frac{r}{s} -\alpha|$ for all rational numbers $\frac{r}{s}$ such that $s \leq n$. While this is stated for $0 < \alpha < 1$, the entire setup can be shifted to $n < \alpha < n+1 $ for any integer $n$. 

The fast process is what we did in computing the cube root of 11. It corresponds to minimizing the ultra-distance which is $q|(\frac{p}{q})-\alpha|$. It scales the normal distance by the denominator and is an attempt to counterbalance the basic advantage of larger denominated rationals for being able to get close to $\alpha$.  ``We call $\frac{p}{q}$ an \textbf{ultra-close approximation} to $\alpha$ if, among all fractions $\frac{x}{y}$ with denominators $y < q$, $\frac{p}{q}$ has the least ultra-distance to $\alpha$.''

The paper then establishes the following theorem: ``Take any irrational number $\alpha$, $0 < \alpha < 1$. The fast continued fraction algorithm gives precisely the set of all ultra-close approximations to $\alpha$.''  The slow algorithm is therefore creeping along as a best approximation while not improving the ultra-closeness. Once we hit an ultra-close approximation, we then start creeping along the other side until we get another ultra-close approximation. 

The construction of numbers using these mediants is referred to as the Farey process. The tree of descendants from this process is called the Stern-Brocot Tree. Matt Baker wrote a nice reference for viewing real numbers as paths in the Stern-Brocot tree.\footnote{{\href{https://mattbaker.blog/2019/01/28/the-stern-brocot-tree-hurwitzs-theorem-and-the-markoff-uniqueness-conjecture/}{https://mattbaker.blog/2019/01/28/the-stern-brocot-tree-hurwitzs-theorem-and-the-markoff} \\ \hspace*{10px}  \href{https://mattbaker.blog/2019/01/28/the-stern-brocot-tree-hurwitzs-theorem-and-the-markoff-uniqueness-conjecture/}{-uniqueness-conjecture/} }} The tree, restricted inclusively to the interval $0:1$, is a convenient way of constructing a list, without redundancies, of rationals between $0$ and $1$. 

We can create an analogous process to that of Proposition \ref{pr:notlist} where instead of bisection, we use the mediant. If we start with a Farey pair\footnote{$\frac{a}{b} < \frac{c}{d}$ are Farey pairs if $bc-ad=1$ which implies $\frac{c}{d} - \frac{a}{b} = \frac{1}{ab}$. The interval of a Farey pair is a Farey interval. The two intervals created by the mediant process are also Farey intervals. We can also reverse the mediant process to get the Farey partner. For any integer $n$, the interval $n:n+1$ is a Farey interval and is an excellent place to start the process.}, then we can apply the above theorems. In particular, if we are given an interval $\frac{a}{b} : \frac{c}{d}$ for the oracle to evaluate and do the mediant process starting with a Farey interval, then we can compute Farey intervals until we reach one whose endpoint denominators are greater than that of $b$ and $d$ at which point we have generate an answer to the given interval. This, of course, only works if we can also determine the result on any given singleton. We have no guarantee of when other methods, such as the bisection method, will have decided on a given interval. Even if we have one mediant for which we cannot decide between two intervals meeting at that mediant, we can still do a double mediant process, to shrink the other endpoints, by presumably finding the rejection of the intervals that do not contain that mediant. An example of this is doing the mediant process on the squaring of $\sqrt{2}$. The answer is $2$, but a naive approach to this will not be able to decide that $2$ is a Yes singleton. So instead we end up doing a mediant process with intervals whose endpoint is $2$, rejecting the interval that does not include $2$, i.e.,  if we were looking at $\frac{3}{2}:\frac{2}{1}$ then we would be checking $\frac{3}{2}:\frac{5}{3}$ and $\frac{5}{3}:\frac{2}{1}$. The first is a No while the second is a Yes. The process would then narrow down to $2$. 

Note that the pattern of selecting the intervals also corresponds to selecting which descendant of the mediant we will be selecting. Given an interval, we compute the mediant without choice. But then after that computation, the oracle selects whether to use the left subinterval or right subinterval, corresponding to taking the left or right path in the Stern-Brocot tree.

\subsubsection{Mediant Interval Representatives, an Aside}

If we wanted to, we could view a given rational number $m$ as the result of this process whose parents in the process are $p < q$. Then, if we have an interval $a:b$ such that $p : a : m : b : q$, we could say that $m$ represents $a:b$. For a given $a:b$, there will be a unique representative. For uniqueness, this follows because if there was $n$ such that $r : a: n: b : s$ with $n$ being the mediant of the Farey interval $r:s$, then let's assume that $p:m:q$ is earlier in the tree then $r:n:s$. Because of it being further down the tree, we would need to have $r:n:s$ either in $p:m$ or $m:q$ which would imply that $a:b$ could not be contained in it. If neither were descendants of the other, then, one of them would not be able to contain either $a$ or $b$. 

As for existence, both $a$ and $b$ appear somewhere in the Stern-Brocot tree. They will have a common ancestor $m$ in the tree as all numbers do. That would be the mediant. The mediant's immediate parents then would contain the interval $a:b$ with $m$ in $a:b$.

For mixed sign intervals, this process would lead to $0$ being the mediant coming from the pseudo-rationals $\frac{-1}{0}$ and $\frac{1}{0}$. For an interval spanning an integer, say $1.5:4.3$, the mediant would be the integer just above the lower one,  here being $2$ with the interval $\frac{1}{1}: \frac{1}{0}$.

From this perspective, we could have a canonical Cauchy sequence for any real number being the mediants from this process which is, of course, just the partial continued fraction representatives. 

This is an aside because while this seems interesting and related, it is not clear what the actual use of this observation is. 

\subsection{Newton's Method}

While Newton's Method has nothing to do with mediants, it is interesting to compare and contrast the method above with the standard root finding of Newton's method. Essentially, mediants are much easier to compute with a goal of getting simple rational approximations while Newton's method requires a bit more computational complexity but it yields a rapid convergence. This section is assuming that the standard view of functions works correctly in our context. See \cite{taylor23funora} for an alternative approach to functions.

As a quick review, Newton's method takes in a differentiable function $f$ and attempts to solve $f(x)= 0$ when given an initial guess of $x_0$. The method is based on the first order approximation $f(x) \approx f(a) + f'(a) (x-a) $ where we view $f(x) =0$ and solve for $x$, leading to $x = a - \tfrac{f(a)}{f'(a)}$. We can therefore define an iterative method of $x_{n+1} = x_n - \tfrac{f(x_n)}{f'(x_n)}$.

Taylor's theorem implies that if $f''(x)$ exists on an interval $(a,b)$ with $\alpha, x_n \in (a,b)$, then we have $f(\alpha) = f(x_n) +f'(x_n)(\alpha - x_n) + \frac{1}{2} f''(u_n) (\alpha - x_n)^2$ where $u_n$ is between $x_n$ and $\alpha$. If we take $f(\alpha) = 0$ and rearrange terms, we have $$\alpha - x_{n+1} = \alpha - \bigg(x_n - \frac{f(x_n)}{f'(x_n)}\bigg) = \frac{-f''(u_n)}{2 f'(x_n) } (\alpha - x_n)^2$$

Let's take an interval $I$ which represents where we expect to be close to the root $\alpha$. Let $M$ be the largest that $|f''|$ will be on $I$ and let $N$ be the smallest that $|f'|$ is on $I$. Then we have 
\begin{flalign*}
|\alpha - x_{n+1}| & \leq \frac{M}{2N} |\alpha - x_n|^2 & \\
  & \leq \frac{M}{2N} (\frac{M}{2N}(|\alpha-x_{n-1}|)^2)^2 & \\
  & = (\frac{M}{2N})^3 |\alpha-x_{n-1}|^4 & \\
  & \leq (\frac{M}{2N})^7 |\alpha-x_{n-2}|^8  \\ 
  & \leq ... \\
  &  \leq \bigg(\frac{M}{2N}\bigg)^{2n-1} |\alpha - x_1|^{2n} \\
  & \leq \bigg(\frac{M}{2N}\bigg)^{2n-1} (b-a) ^{2n}
\end{flalign*}
The error gets quadratically smaller and smaller assuming $\frac{M}{2N} (b-a) < 1$. To avoid assuming knowledge of the root, one can use Kantorovich's Theorem. 

Since we did the cube root of 11 above in detail with the mediants, let's try it with Newton's Method. We will start with an initial guess of $x=2$ and we know that the cube root of 11 is between 2 and 2.25. For Newton's method, we need a function which we will take to be $f(x) = x^3 - 11$. If $f(\alpha)= 0$, then $\alpha$ is a cube root of 11. $f'(x) = 3x^2$ and $f''(x) = 6x$. To keep our bounds simple, we will bound the derivatives on the interval $[2,3]$ which yields $M = 18$ and $N = 12$. So the factor in front can be bounded by $\frac{18}{24}=\frac{3}{4}$. The initial interval has length less than $1/4$ and we know the $0$ is in $(2, 2.25)$ by the Intermediate Value Theorem. We therefore have error estimates for the first few guesses of $\frac{3}{4}(\frac{1}{4})^2 < .047$, $(\frac{3}{4})^3(\frac{1}{4})^4 < .00165$, $(\frac{3}{4})^7(\frac{1}{4})^8 < 2.04\times 10^{-6}$, and $(\frac{3}{4})^{15}(\frac{1}{4})^{16} < 3.12 \times 10^{-12}$. That is, we will have exhausted a typical calculator precision after four iterations. 

Let's do it. Our iteration formula is $x_{n+1} = x_n - \frac{x_n^3 -11}{3 x_n^2}$ 
\begin{enumerate}
\item $x_0 = 2$, $|\alpha - x_0| < \frac{1}{4}$
\item $x_1 = 2 - \frac{8- 11}{12} = \frac{9}{4}$, $|\alpha - x_1| < .047$
\item $x_2 = \frac{9}{4} - \frac{ 729/64 - 11 }{243/16} = \frac{2162}{972}\approx 2.2243 $, $|\alpha - x_2| < .00165$
\item $x_3 = \frac{1894566349}{851880969} \approx 2.2239801 $, $|\alpha - x_3| < 2.04\times 10^{-6}$
\item $x_4 = \frac{20400964697239818757748038397}{9173177756288897151620391507} \approx 2.22398009056931625 $, $|\alpha - x_4| < 3.12 \times 10^{-12}$
\end{enumerate}

The last two computations were done with WolframAlpha.  WolframAlpha also reports the cube root of 11 as approximately 2.22398009056931552. 

As one can see, Newton's Method converges quickly, but the fractions are quite complex. With the mediant method, the fraction $25022/11251$ is a closer approximation than $x_3$ above while using far fewer digits. One should also keep in mind that the computations for the mediant method requires nothing more than multiplication and addition with the added bonus that the fractions are always in reduced form. It is also worth mentioning that this iteration did start with an interval in common with the mediant method that we did previously. 

\subsection{Geometrically Computing \texorpdfstring{$\sqrt{2}$}{sqrt2}}

In this section, we are going to explore computing the square root of 2 with right triangles whose sides are Pythagorean triples and whose legs are almost identical. This will allow us to get rational intervals that bound the square root of 2. 

As a simple and not great example, consider the classic $3-4-5$ right triangle. By dividing the sides by 3, we get a $1-\frac{4}{3}-\frac{5}{3}$ right triangle whose hypotenuse has length $\frac{5}{3} = 1 \frac{2}{3}$. Dividing by 4 leads us to a $\frac{3}{4}-1-\frac{5}{4}$ with hypotenuse  length $\frac{5}{4} = 1.25$. Notice that $1.25 < \sqrt{2} < 1.\bar{6}$.  The ordering occurs in this way because we can imagine lining up the $1-1-\sqrt{2}$ triangle with the other two that have a side whose length is 1. One of the triangles is taller and the other is shorter than a height of 1. The hypotenuse lengths correspond to that height ordering. 

How do we find good Pythagorean triples for this? It just so happens that a well known rational parametrization of the unit circle is  $r(t) = (\frac{1-t^2}{1+t^2}, \frac{2t}{1+t^2})$.  If we have a rational $t$ such that those two coordinates are nearly equal, then we will end up with triangles that will be good approximations for the square root of 2. The $\sqrt{2}$-Yes interval generated by a given $t$ is $\frac{1+t^2}{1-t^2}:\frac{1+t^2}{2t}$. The ordering of that interval will change based on whether $t$ is greater than or less than the $t$ which provides equality of the two coordinates. 

Let's call the $t$ that does that $T$. If $T$ were rational, the mediant process would then provide a singleton solution to it. It is, of course, not rational. We can use Newton's method to produce good rational approximations of $T$ by using the function $f(t) = 1-t^2 - 2t$, but we will demonstrate using mediants and the intermediate value theorem procedure. To start off with, $f(0) = 1 > -2 = f(1)$. The mediant between them is $\frac{1}{2}$ and $f(\frac{1}{2}) = -\frac{1}{4}$. So we select $\frac{0}{1}:\frac{1}{2}$ for our next interval. By the way, $r(\frac{1}{2}) = (\frac{3}{4}, 1) / \frac{5}{4} = (\frac{3}{5}, \frac{4}{5})$ which leads to the $3-4-5$ right triangle we discussed above.  The next interval has mediant $\frac{1}{3}$ and $f(\frac{1}{3}) = \frac{2}{9}$ leading to selecting $\frac{1}{3}:\frac{1}{2}$. We also have $r(\frac{1}{3}) = ( \frac{8}{9}, \frac{2}{3}) / \frac{10}{9} = (\frac{8}{10}, \frac{3}{5})$ which is the same triangle again. Notice how we shifted the direction of the interval selection and the legs flipped the ordering as well.

Our next interval's mediant is $\frac{2}{5}$ with $f(\frac{2}{5}) = \frac{1}{25}$ which tells us to select the interval $\frac{2}{5}:\frac{1}{2}$. And then $r(\frac{2}{5}) = (\frac{21}{25}, \frac{4}{5})/\frac{29}{25} = (\frac{21}{29}, \frac{20}{29})$ leading to $\frac{29}{21}:\frac{29}{20} \approx 1.38: 1.45$ as the interval approximation for the square root of 2 using the right triangle $20-21-29$. The next mediant is $\frac{3}{7}$ and we would select the interval $\frac{2}{5}:\frac{3}{7}$ leading to a mediant of $\frac{5}{12}$.  Note that because we switched directions, we end up with the a scaled version of the same right triangle; here it is $40-42-58$. 

We can continue in this fashion, alternating the interval selection every two times as the positive solution to $1 - t^2 - 2t = 0$ is $\sqrt{2}-1$ with continued fraction $[0;\bar{2}]$ which does tell us to flip the selection direction every two selections. It also tells us that adding $1$ to the $t$ values gives us direct approximations to the square root of 2. 

If we did want to use Newton's method on this to quicken the pace, we would be computing the formula $x_{n+1} = x_n - \frac{ t^2  + 2t -1}{2t + 2}$. If we started with $\frac{2}{5}$, then the next one would be $\frac{29}{70}$. Computing out its right triangle, we get $4059-4060-5741$ with estimates $\frac{5741}{4059}\approx 1.4143: \frac{5741}{4060} \approx 1.41403$. 

\subsection{The Continued Fraction of \texorpdfstring{$\pi$}{pi}}

We can do the mediant method using the intermediate value process for computing out the continued fraction of $\pi$. Let us assume that we can compute the sign of $\sin(x)$ for any rational $x$ and also assume that we know that $\sin(x)$ has exactly one zero in the interval $3:4$. This zero is, of course, $\pi$. Then the oracle of $\pi$ restricted to intervals in $3:4$ is the set of intervals $a:b$ such that $\sin(a):0:\sin(b)$ is true. That is, we can compute the Yes or No of intervals by testing out the sign of sine. 

If we start out with the interval  $3:4$ and use the mediant process, we will be working with a Farey pair. The first mediant is $\frac{7}{2}$. Note the signs of sine are  $\sin(3) \col +$, $\sin(4) \col -$, $\sin(7/2) \col -$. So our next interval is $3:\frac{7}{2}$ with mediant $\frac{10}{3}$ whose sine value is negative. As we proceed, we generate the mediants $\frac{13}{4}$, $\frac{16}{5}$, $\frac{19}{6}$, and finally $\frac{22}{7}$. All of these had negative sine values and thus we were choosing the left sub-intervals. We did this 7 times so the continued fraction before the switch is $[3;7]$. The mediant where we are making the switch is $\frac{25}{8}$. It is represented by both $[3; 7, 1]$ and $[3, 8]$.  Because $\sin(\frac{25}{8}) > 0$, we choose the right subinterval represented by $[3; 7, 1]$. We will now select the right subintervals, namely those with $\frac{22}{7}$ as the endpoint. As we proceed, we will find we do this 15 times before switching to left subintervals. This leads to $[3; 7, 15] = \frac{333}{106}$ and $\frac{333}{106}:\frac{22}{7}$ is the interval of containment before we start selecting the left subintervals again.

This is a method, perhaps a slow one, which allows one to compute the continued fraction of any solution to an equation of the form of $f(x) =0$ where the sign of $f$ changes near the solution. For an alternative direct method of computation using a variant of Newton's Method, see \cite{shiu95}.\footnote{I found this reference at an excellent MathStackExchange overview of computing out continued fractions: \url{https://math.stackexchange.com/questions/716944/}.  It also gives some references for doing arithmetic with continued fractions which we demonstrate in Section \ref{sec:con-frac}.} 

We could do something similar for $e$, where we look at whether the function $\ln(x)$ is greater or less than 1 to determine if we switch the intervals. The basic starting point will be $2:3$ leading to a sequence of mediants being $\frac{5}{2}, \frac{8}{3}, \frac{11}{4}, \frac{19}{7}$ where the last one is accurate to within $0.004$ with the bounding interval giving a bound of $0.036$. 

The viewpoint of oracles seems naturally to lead into mediants and continued fractions.

\section{Alternative Definitions}\label{sec:others}

It is useful to compare our oracle approach with other common approaches and some nearby alternatives to this. 

What are some good properties of a definition of real numbers? This is subjective, of course, but I was guided by the following: 

\begin{itemize}
    \item Uniqueness. Given a target real number, there should be only one version of that in the real number system and its form should be indicative of what the number is. 
    \item Reactive. This is a key feature. Real numbers generally have an infinite flavor to them. It was important to me to not pretend we could present the infinite version of that, but rather to present a method of answering queries. 
    \item Rational-friendly. Ideally, rationals would be easily spotted, treated reasonably, and arithmetic with them would be easy to do. 
    \item Supportive. The definition should be in line with and supportive of standard practice of numbers. In particular, it should conform to how numbers are used in science, applied mathematics, and numerical analysis. 
    \item Delayed Computability. The definition should not require computing out approximations just to  state the number, but the definition should facilitate such computations. 
    \item Arithmeticizable. It should feel like the arithmetic laws are approachable and computable. That is, one can compute out the operations to any given level of precision and be confident in the result.
    \item Resolvability. We have concrete examples of real numbers whose fundamental nature is unknown. Does the approach give language or a setup that can respect that? 
\end{itemize}

The Oracle approach fits the first five of these quite easily. The last two are pretty subjective and perhaps the best way forward on those is to contrast them with the other definitions. 

Much of what follows was heavily inspired by NJ Wildberger's excellent videos, such as ``Real numbers as Cauchy sequences don't work!''\footnote{\url{https://www.youtube.com/watch?v=3cI7sFr707s}}

\subsection{Infinite Decimals}

This is a natural and old attempt at defining real numbers and is the approach of early mathematics education. It originated in the 1500s with Simon Stevin. 

The idea is to write the decimal expansion of irrational numbers as we do with rationals, but we can never complete the decimal form and can only produce a finite amount of digits. 

Ordering is very easy with this presentation as long as we can go as far as needed for numbers that are different. Establishing equality by comparing digits is not directly possible since we cannot write out infinitely many digits. 

We can view this as describing intervals whose width is a power of 10. For example, writing $1.41$ can be taken as $1.41:1.42$. See Section \ref{sec:decimals} for the subtleties and feasibility of this approach. 

Let us run through our criteria. 

\begin{itemize}
    \item Uniqueness. It has the issue of trailing 9's. Otherwise, there is only one representative. 
    \item Reactive. This can be viewed as answering the question of ``What is a decimal approximation up to the $n$-th place?'' where $n$ is some given natural number. It is even possible for some approximations to just produce the desired $n$-th place decimal value. The standard presentation, however, is misleading. As an example, $\sqrt{2} = 1.4142...$ gets presented as if the whole decimal expansion is there but is hidden due to space limitations. 
    \item Rational-friendly. Sort of. Rationals are the ones with repeating decimal expansions and can therefore be spotted. Those that terminate, however, should technically have 0's appended which is just odd. The arithmetic is not tenable. Compare multiplying $\tfrac{1}{9}*\tfrac{1}{9} = 0.\bar{1} * 0.\bar{1}$. Actually try that multiplication in decimal form. As one continues on, one has to carry (modify) digits many places away. Even multiplying 1 by itself in the form of $.\bar{9}*.\bar{9}$ takes a bit of effort. 
    \item Supportive. We certainly use numbers in decimal form to do computations. But we need to have error bars added, that is, we need to write these as decimal intervals. 
    \item Delayed Computability. This is a definition that requires computing it out, impossibly so. 
    \item Arithmeticizable. Not that easy, as indicated even in the simple case of the rationals multiplying. The most straightforward way is probably an interval or limiting kind of application of the arithmetic operators on the partial decimal approximations. Note that without explicit error bars being given, the decimal shorthand of an interval approximation being read off by the number of decimal digits ceases to apply as we do the arithmetic operations since the intervals expand in size. This is presumably the origin of the significant figure rules. 
    \item Resolvability. We can state how much we know of a number with the decimals and leave it as an error bar. But it would not be particularly convenient to write out $2^{68}$ 0's in the number associated with the Collatz conjecture. 
\end{itemize}

Infinite decimals are a natural attempt and it is a very common presentation of what we know of a number. But it is often not a convenient form and strongly suggests intervals instead. 

The lack of error control is problematic with the decimal expansion. For example, assume we want to multiply the real numbers $\sqrt{2} \approx 1.414$ and $\sqrt[3]{48}\approx 3.634$. Let's take the truncated versions of 1.4 and 3.6 with the implicit understanding of having the interval $1.4:1.5$ and $3.6:3.7$ as the intervals containing these numbers, respectively. When we multiply the representatives, we get $1.4*3.6 = 5.04$, but when we multiply $1.5*3.7$ we get $5.55$. The spread has increased and the best we can say is that $\sqrt{2} \sqrt[3]{48}$ is in in the interval $5.04:5.55$ leading to $5$ as the truncated answer if we follow the truncation convention. If instead of viewing the decimal expansion as representing the story so far and instead viewing it as having been rounded, then the numbers represent $1.35:1.45$ and $3.55 :3.65$ which leads to the interval $4.7925: 5.2925$ and thus for a rounded version, we again find $5$ as the representative answer with the increased uncertainty of $4.5:5.5$ as the implied interval. We can, of course, repeat this with more digits. For example, doing the rounded version with three digit approximation, we get $1.4135:1.4145\otimes 3.6335:3.6345 = 5.13595225:5.14100025$ which leads to a rounded version of $5.14$ representing the interval $5.135:5.145$. To track precision properly, we need to have intervals as part of the decimal arithmetic.

\subsection{Nested Intervals}

One can think of expanding the concept of infinite decimals as being a sequence of nested intervals where the length goes down by a tenth at each level. We could generalize this to be a more arbitrary sequence of nested and shrinking intervals. From the Oracle point of view, we could use the mediant method to define such a sequence of intervals. A sequence of such intervals is a fonsi that gives rise to an Oracle. These are closely related concepts. 

Let us run through our criteria. 

\begin{itemize}
    \item Uniqueness. This clearly fails. We can have two entirely distinct nested interval sequences describing the same real number in addition to arbitrarily changing a given sequence (cut out half of them, double their lengths, ...)
    \item Reactive. Not at all. The sequence of intervals is given. We could recast this as a function that, given an $n$, we generate the $n$-th nested interval based on what came before. We would probably want a function of length that gives us a shorter, nested interval from what came before. A choice external to the number must be made to decide which is the $n$-th interval. 
    \item Rational-friendly. The rationals are those whose nested intervals converge to a rational number. There does not seem a particularly clear property that establishes this. Depending on the definition, we could have a finite nested interval sequence that ends in the singleton $q:q$ if it is allowed. 
    \item Supportive. On a practical level, we do like shrinking intervals. But it is not generally predefined intervals. Mostly, it is intervals that are generated when working and we would want to know that the given interval has a non-zero intersection with all the nested intervals. 
    \item Delayed Computability. The sequence needs to be specified. Which sequence to provide is not intrinsic to the number and thus this does require the work of computation.  
    \item Arithmeticizable. Interval arithmetic works. If we tried to build in a constrained size, such as the $n$-th interval has to be no longer than $1/n$ in length, then the arithmetic would become a chore to manage. 
    \item Resolvability. The smallest intervals we can inspect in the sequence will tell us the resolution we have on a number. 
\end{itemize}

To address the uniqueness issue, one could pursue an equivalence class approach as one does with Cauchy sequences, as discussed below, with the same difficulties. But one could also take the nested intervals as inspiration and expand to have all the intervals that contain one of the nested intervals. This is a maximal fonsi which is equivalent to the set of all Yes-intervals for an oracle. 

There is also the issue, similar to the Cauchy sequence, though less severe, that the nesting intervals can be quite wide for a very large portion of the sequence.  For example, we could have a nested interval sequence which has an initial trillion intervals that are all wider than the known universe.

\subsection{Cauchy Sequences}

A Cauchy sequence is a set indexed by natural numbers such that for any given rational $\varepsilon > 0$, we can find an index such that all later elements of the sequence will be within $\varepsilon$ of each other. We can consider these as the real numbers. 

We can apply one of our algorithms to a given Oracle to generate a Cauchy sequence that represents that real number. Given a Cauchy sequence, we can generate an Oracle by the rule that an interval is a Yes if it contains all of the tail of the Cauchy sequence. See Section \ref{sec:cauchy}.

Let us run through our criteria. 

\begin{itemize}
    \item Uniqueness. This fails utterly here. We can say that a Cauchy sequence represents a real number. But there are infinitely many such sequences. So then we can consider an equivalence class of them, but this then becomes a very different kind of object. In addition, we have the problem that the initial part of the sequence can be anything. Given an equivalence class, we can expect that the finite portion of all Cauchy sequences will look the same. We could require that each term must be within the latter ones by a given precision based on the index, such as $|a_n - a_m| < \tfrac{1}{n}$ whenever $n < m$, but this makes the arithmetic portion of this more difficult in addition to actually computing such a sequence in practice may require more work for no practical gain. 
    \item Reactive. Given a desired precision, we can ask the Cauchy sequence for the $n$ and then for an element of the sequence. The given element is chosen by a mechanism not directly essential to the real number. 
    \item Rational-friendly. Rationals can have a constant sequence which is different. But one can also have a sequence for an irrational which is constant for a trillion to the trillion terms and then starts changing. The arithmetic between two constant sequences is easy, but the representative of the rational may or may not be constant in which case there is no difference from the irrationals.   
    \item Supportive. This is used by analysts in many theoretical arguments. For numerical work, we may be generating a sequence of approximate values that get closer and closer. But we need to pair the sequence value with an $\varepsilon$ to have it be actually useful. That is, to be useful, we need an interval lurking around. 
    \item Delayed Computability. If working with one sequence, then it needs to be specified. Which sequence to provide is not intrinsic to the number and thus this does require the work of computation.  For the equivalence class, one needs to somehow specify some criteria for being in the equivalence class. It is often done by having a specific sequence and then including all the ones that are equivalent to that sequence. 
    \item Arithmeticizable. This is the arithmetic of the individual sequence elements. This gets a little messy with the equivalence classes. If we try to weed out the initial garbage by specifying some specific sequence of $\varepsilon$'s to satisfy, then the arithmetic operations become more difficult to handle since we need to play around with ensuring that constraint.  
    \item Resolvability. This is difficult. One can produce a Cauchy sequence for the unknown numbers up to a point, but a finite segment of the Cauchy sequence is not useful without the $N$, $\varepsilon$ information which is really specifying an interval. That is to say, we are producing an interval in which all future sequence members need to be within.  
\end{itemize}

Cauchy sequences are very appealing as a next step past the decimal expansions, which themselves can be viewed as a very nice Cauchy sequence where each step along $n$ leads to a $\frac{1}{10}$th shrinking of the future variability. But they are a collection of somewhat arbitrary choices in the representation of a number. 

\subsection{Fonsis}

A natural progression from the nested interval sequences and Cauchy sequences is the Family of Overlapping, Notionally Shrinking Intervals (fonsi) which we have used throughout this paper. It is a reasonable generalization of these approaches and has some reasonable generalizations. This is a favored version of some constructivists as in \cite{bridger} and \cite{bridges}. 

But it suffers from many of the same problems as these other approaches. Namely, it requires non-canonical choices which then complicates the discussion. The constructivists, not particularly enjoying equivalence classes, define a real number as a fonsi and then define an equality operation which is the equivalence relation. They do not suggest that two fonsis are equivalent, but rather that they are equal as real numbers. The main difference is that they want to avoid talking about the set of all fonsis representing a given number and simply handle it on a case-by-case basis as it arises. 

\begin{itemize}
    \item Uniqueness. The fonsi approach does not give a unique representation at all. The maximal fonsi is unique and taking that to be the real number would make it unique, but from the perspective of constructivists, that would require typically doing an infinite amount of work granted by fiat. 
    \item Reactive. This is not very reactive. If we wanted to test an interval, it is insufficient to ask if it is in the fonsi and insufficient to ask if it intersects a finite number of elements in the fonsi. One needs to know if it intersects infinitely many elements of the fonsi and that the collection must include arbitrarily small intervals. 
    \item Rational-friendly. There is the option of having a rational singleton as the fonsi or as part of a fonsi. But it is not required to be there. This is part of the advantage of a fonsi as a computational tool, but seems lacking as a definition. 
    \item Supportive. It allows for more flexible uses than having a nested sequence and is in line with notions of measurements. 
    \item Delayed Computability. One has to specify the fonsi. Since there are many fonsis per real number, something additional to the number needs to be specified suggesting we need to do some computation.
    \item Arithmeticizable. This is what fonsis seem really good at doing. This is what our oracles rely on in the arithmetic of oracles. 
    \item Resolvability. This is in line with the resolvability discussion of oracles. Essentially, we find the smallest intervals in the fonsis that we can produce and that is the resolvability of the two. If we have two disjoint elements of the two fonsis, then we know they are distinct. 
\end{itemize}

Fonsis are very close to our oracle approach and are excellent tools for dealing with real numbers. But they seem to be much closer to a tool than a full definition.
  
\subsection{Dedekind Cuts}

The idea of a Dedekind cut is to divide the rational numbers into two pieces, one of which is below the real number and one which is above. The cut is where the real number is. If the cut happens to be a rational number, then one has to decide which set it ought to be in. The recasting below puts the rational in its own space. 

The approach of Dedekind cuts is a common construction of real numbers in beginning analysis courses. They have a very nice convincing example of the square root of 2 and the ordering, based on subsets, is quite nice. The arithmetic gets messy in detail, but conceptually it is not problematic. 

To align it more with our approach of Oracles, we would recast the set aspect into having a rule $T$ which decides whether a given rational number is less than, equal to, or greater than, the target real number. It can be codified by giving a -1, 0, or 1,  respectively, something in line with how one might code these kind of inequality tests in a programming environment. 

To convert to intervals, we see that below the cut is the set of lower endpoints of intervals that contain the real number while above the cut is set of the upper endpoints. Given a lower and upper bound on the cut, we can proceed with the algorithms we have already discussed and use $T$ to decide whether the new middle point is a lower bound, the number itself, or an upper bound. 

We can also view the Separation Property as being very similar to the $T$ function. If we had a $T$ function, then we could use it to answer whether to accept the left interval ($T(c) = -1$), the right interval ($T(c)=1)$, or accept the singleton $(T(c) = 0$). We could also use the existence of an interval and the Separation property to construct the $T$ function.

The arithmetic for our version is fairly straightforward. The new rule for the sum  $z = x+y$ is that a given rational $s$ is less than $z$ if $s$ can be written as $p+q$ for two rationals that are less than $x$ and $y$, respectively. It is greater if we can find two rationals summing to it that are respectively greater. Equality is a little tricky unless we are actually adding two known rationals. Otherwise, we need to work to find a gap one way or the other. 

Let us run through our criteria. 

\begin{itemize}
    \item Uniqueness. For each cut, we have a unique representative. This is based on deciding to exclude the rational number if the Dedekind cut represents a rational number. In essence, we are taking the representatives of $0.\bar{9}$ and excluding $1$ as representing itself. 
    \item Reactive. The standard presentation of the cut sets is not reactive. Technically, one would need to specify the set entirely. This works for the square root of 2, but is more difficult for something like $\pi$. Reformulated as above, one can ideally compute it out for any given rational that we wish to ask about the relation to the real number. 
    \item Rational-friendly. The standard presentation is awkward with rationals and does not highlight them. If we were take the viewpoint of the ternary function $T$ above, then rationals are exactly the $T$ functions where $0$ is in the range. 
    \item Supportive. Somewhat. Producing the Dedekind cut is not something typically needed or done, but figuring out whether one is below or above a given target is certainly useful and not an entirely wasted effort. 
    \item Delayed Computability. The standard definition does not seemingly require computation to define a number but neither does it particularly help computation. This is certainly true of computing roots where the inequality can be informatively stated. For something like $\pi$, there is no clean definition of the inequality other than having to compute lower bound approximations and being under those approximations to be in the set. One can use the IVT style setup to define $\pi$ as the set of $x$ that are less than 3 or, if $x$ is between 3 and 4, then add in the $x$ such that $\sin(x) < 0$. 
    \item Arithmeticizable. This gets a little messy for negatives and multiplication since directions get reversed. The actual computable action in our presentation is roughly equivalent to the Oracle arithmetic, but the standard presentation demands the whole set be produced which is not possible.
    \item Resolvability. It is obscured, particularly in the standard formulation. For our formulation, it basically suggests there is a gap between the less thans and the greater thans. This more or less suggests using the interval approach of our paper to get into this. 
\end{itemize}

The idea of Dedekind cuts, properly formulated, is a solid candidate for constructing the reals, but it feels slightly tangential from the main desire of what we want to know with a real number. It feels like the remains after someone tore apart the intervals of interest. 

The $T$ function certainly comes very close to our rule $R$, but, as detailed in \cite{taylor23metric}, we can generalize oracles easier than an order-based definition. In particular, if we use the Two Point Separation property, then oracles generalize to any metric space. 

It is also useful to point out that while the Separation property can be resolved with the help of the $T$ function, we do have instances of the Separation property being answered without one. Namely, the process behind the Intermediate Value Theorem is very much one of considering Yes or No on intervals and not a question of whether the given point is less than or greater than the target. Indeed, if we have multiple possible answers in a given interval, then a $T$ would not exist a priori. But given any interval to test on and an internal separation point, the oracle would yield an answer. 

As an example, we can modify the Thomae's function $t$ to be, for the rationals in reduced form $\frac{p}{q}$, $\frac{(-1)^p}{q}$. We still would have $0$ as the value on the irrationals and we would have continuity just the same on the irrationals. Given $1:2$ as a starting interval for the IVT, we would have $0$ is between $t(2) = -1$ and $t(1) = 1$. If we take the mediant, we are looking at $t(\frac{3}{2}) = \frac{-1}{2}$ and so we choose the interval $1:\frac{3}{2}$. That mediant's value is $t(\frac{4}{3}) = \frac{1}{3}$ and so we choose $\frac{4}{3}:\frac{3}{2}$ leading to a mediant of $\frac{7}{5}$. And so on. This will converge to some neighborly oracle (irrational) whose value is $0$. But it is hard to see how we could construct a $T$ function directly.

The oracle framework handles this with ease in contrast to the Dedekind cuts.

\subsection{Minimal Cauchy Filters}

I obtained this from \cite{weiss2015reals} which does not seem to be officially published but can be found on the arXiv under the title ``The reals as rational Cauchy filters''.\footnote{\url{https://arxiv.org/abs/1503.04348v3}} A filter is very similar to what we have used here. It is a collection of sets with pairwise intersections being a part of the collection and that any set that contains a set in the filter is also in the filter. A minimal filter is one which does not contain any other filter. A Cauchy filter is one which has arbitrarily small intervals in it. The paper goes through constructing the real numbers as the collection of all minimal Cauchy rational filters. 

A useful case to focus on is that of a rational number in this viewpoint. The maximal filter of $q$ is the one consisting of all sets that contain $q$. The minimal Cauchy filter is generated by the base of $q_{\varepsilon}$ intervals, namely, the intervals centered at $q$ with a length of rational $\varepsilon>0$.

One could liken the approach of the oracles as taking the minimal Cauchy filter and restricting that to only sets which are intervals. We do use intervals more general than the $q_{\varepsilon}$ intervals in part to be able to use the Separating property which generally would not result in a $q_\varepsilon$ interval. 

In any event, it should be clear how to map the two. Given an Oracle, the Yes intervals are the base for the filter which generates the real number (excluding the singleton if it is rational). One would need to verify that it is a minimal Cauchy filter. Given a minimal Cauchy rational filter, we generate an Oracle on intervals by a Yes being given if the interval is in the filter. We add in the singleton if it happens to be a rational real. 

Let us run through our criteria. 

\begin{itemize}
    \item Uniqueness. The filter is unique with no need for equivalence classes. But for every interval containing the real number, we have infinitely many sets that are added to that interval to generate the various sets. This is analogous to the arbitrary changing of the head of a Cauchy sequence. In particular, it would be very difficult to take a random set from a filter and figure out anything useful to say about which real number we were talking about. The ``garbage'' portion is away from the real number in question, but it would be hard to determine from a randomly chosen member of the filter which real number was under consideration. 
    \item Reactive. We can recast this as a query setup, namely given a set, we can say Yes or No depending on if it is in the filter or not. Unfortunately, given the almost random nature of the sets, it can be difficult to even present the set to be asked about. In contrast, rational intervals require just two fractions to ask about it. 
    \item Rational-friendly. Rationals are singled out by being the filters with a non-free core, that is, all the sets in the filter have $q$ as an element. The arithmetic, however, is not particularly improved. In particular, singleton sets $q$ are not actually part of the filter as that would generate the filter of all sets containing $q$. This means that it is slightly unnatural to focus on the singleton arithmetic though one can always do that and then generate the $\varepsilon$ intervals from that which form the base. 
    \item Delayed Computability. It does not require explicit computations, but it also does not seem to help much with computing out approximations. In particular, the bisection and mediant methods do not seem to be easily used on an arbitrary element of the filter. 
    \item Supportive. This has a bit of the core idea of wanting to say ``the number is in there'', but similar to the Cauchy sequences, it gets derailed by the large variety of useless set baggage that gets brought in with the filters. 
    \item Arithmeticizable. The arithmetic laws are easy to state, basically, being nothing more than the set generated by applying the operator to all of the elements. Unfortunately, this simplicity of statement does not translate into something easy to compute since we have to deal with essentially arbitrary sets. If one focuses on the base of the filter, namely, $\varepsilon$ intervals, then that is essentially the same as the arithmetic in our approach. 
    \item Resolvability. It does not seem to be particularly useful in dealing with uncertainty in what the real number is. It would be presumably appealing to $0_{\varepsilon}$ intervals and stopping when we got to the level of our current knowledge. We could talk about the smallest sets that the two filters have in common. 
\end{itemize}

Filters are generally concerned with arbitrary sets. It is not clear what advantage this has over just being concerned with intervals. If we restrict our attention to only sets that are intervals, then the minimal Cauchy filter approach and our approach would largely coincide, with the exception of the singletons.

\subsection{Continued Fractions}

An interesting and quite useful construct with real numbers is that of continued fractions. We have already touched upon them in terms of the mediants. It is possible to view real numbers as continued fractions. 

\begin{itemize}
    \item Uniqueness. Continued fractions are unique for irrationals. For rationals, there are two of them, corresponding to the two descendent trees rooted at that rational in the Stern-Brocot tree. 
    \item Reactive. This can be seen as a process generating approximations to the real number. One can ask for a representative at a given level.  
    \item Rational-friendly. Rationals are the ones with finite expansions though they do have two representatives. 
    \item Supportive. Continued fractions are certainly used in computational settings. The expansions are not, however, intuitively useful by themselves. The first part being the integer is quite useful. Measurements would generally have little reason to be converted into continued fractions. 
    \item Delayed Computability. This needs to be computed out as its definition. Having the continued fraction can facilitate other representations such as decimals. 
    \item Arithmeticizable. This is hard with continued fractions. Ordering is very easy as it is just comparing on the expansions. But doing the arithmetic operations  using continued fractions requires some effort. See Section \ref{sec:con-frac}.
    \item Resolvability. Comparing real numbers can be done by comparing the expansions.
\end{itemize}

Continued fractions are a great tool and have some nice properties as a definition, but fundamentally it seems specialized to certain needs. They do not seem to generalize to metric spaces. 

\subsection{Other Constructs}

The survey paper \cite{ittay-2015} informs the rest of our analysis of other constructions. 

Many of them are sequences of sums or products, which are rather interesting different representations, but they can be viewed as specialized cases of Cauchy sequences. They bear significant resemblances to the infinite decimal approach as they generally involve stand-ins for digits in the construction of each term, but they generally avoid the carry problem. They all seem to suffer from the idea that as a general scheme for representing real numbers, they are not generally useful except in particular cases. They may complicate some arithmetic operation or, more commonly, the ordering relation. They also may just take a great deal more of computation to accomplish without necessarily having any advantage. Some of them also create non-unique representatives and thus require an equivalence class. Most of them avoid the garbage header associated with Cauchy sequences. 

They generally seem quite amenable to producing intervals in which the number can be seen.

The survey paper is an enjoyable and enlightening read for those interested in other constructions of the reals. 

\subsection{Contrasting an Actual Computation}

Earlier, we looked at $\sqrt{2} (e + \pi) $ as done out with oracles. It can be instructive to see how we might approach that computation in each of the different options. We will be brief with most of the options as there is not much to say on some of them. We will go in-depth into the decimal version and continued fraction version. 

\begin{enumerate}
\item Decimals. This is the standard arithmetic. It suffers from the carry problem if thinking about decimals exactly. For approximations, it suffers from the expanding nature of interval arithmetic. See below for some of the issues. 
\item Nested Intervals. One has to choose the nested interval representations for $\sqrt{2}$, $e$, and $\pi$. Then we can just do the indicated operations. This can look similar to what was done with oracles, except there is no ending point. It really implies that we need to have formulas for each element of the sequence and then combine those formulas. It would then need to be shown that they are nested and the lengths are shrinking to 0. 
\item Cauchy Sequences. In some ways, this is similar to the nested intervals in that we need to pick the Cauchy sequences for these numbers and then combine these sequences. This is assuming that we are looking at particular Cauchy sequences. If we are trying to do this with whole equivalence classes of Cauchy sequences, then it is unclear what one could really compute. We listed explicit summations for $\pi$ and $e$. For $\sqrt{2}$, we can use the double factorial sum $\sqrt{2} = \sum_{k=0}^\infty (-1)^{k+1} \frac{(2k-3)!!}{(2k)!!} =
1 + \frac{1}{2} - \frac{1}{2\cdot4} + \frac{1\cdot3}{2\cdot4\cdot6} - \frac{1\cdot3\cdot5}{2\cdot4\cdot6\cdot8} + \cdots = 1 + \frac{1}{2} - \frac{1}{8} + \frac{1}{16} - \frac{5}{128} + \frac{7}{256} + \cdots$. One would then add the sums of $\pi$ and $e$ and multiply that sum by the expression above, doing partial sums to get concrete representations. The multiplication adds some complications though less than if one was doing division or higher powers. 
\item Fonsis. This is similar to the nested intervals and the oracles. It is more flexible spanning the possibilities of a singleton to the maximal family of oracles. It is the primary ingredient to doing oracle arithmetic. 
\item Dedekind Cuts. It is unclear what would be useful. One could take representatives from each of the sets and do the arithmetic, but to be useful one ought to do it for those outside the set to create upper bounds. This is basically creating bounding intervals. 
\item Minimal Cauchy Filters. This is doing arbitrary set arithmetic and rather unclear what would be sufficient. One could focus on fonsis and get the results, but something which is fundamentally of this kind is unclear. 
\item Continued Fractions. This is where continued fractions have severe impediments to being a good definition of reals from a conventional point of view. But this difficulty was overcome in the 1970s by Bill Gosper. See the dedicated subsection below for references and how to do a computation with it. 
\item Other Constructs. Use the given representatives and do the operations as was indicated with the Cauchy sequences. 
\end{enumerate}

\subsubsection{Decimals}\label{sec:decimals}

For decimals, we have a few difficulties. For example, if we multiply the decimal version of $\frac{1}{9} = 0.\bar{1}$ with itself, we should end up with $\frac{1}{81} = 0.\overline{012345679}$. For oracles, we can just multiply the fractional representatives with no difficulties as it is a singleton. For decimals, to do it with decimals, we multiply these two infinite strings of 1's which leads to adding up arbitrarily large numbers of 1's. This creates carrying problems. Mysteriously, it will come out to the pattern described above, but it is not clear how. Norm Wildberger has a nice video exploring this.\footnote{\url{https://youtu.be/YYnYatWV-pU?t=2198}}

Another difficulty is what exactly a decimal represents. If one writes $1.41\ldots$, then what interval is one implying by that and is it maintained by the arithmetic of these shorter representatives? For the rest of this section, the symbol $=$ is more of an approximately equal. We will explore four different interpretations of what this could imply:

\begin{itemize}
\item Short Rounding. The interval for $x$ is $a.d_1d_2d_3\ldots d_n \pm 0.5*10^{-n}$. $\sqrt{2} = 1.41$ represents saying that $\sqrt{2}$ is in the interval $1.405:1.415$.
\item Truncation. The interval for $x$ is $a.d_1d_2d_3\ldots d_n: a.d_1d_2d_3\ldots (d_n+1)$. $\sqrt{2} = 1.41$ represents saying that $\sqrt{2}$ is in the interval $1.41:1.42$. 
\item Long Rounding. The interval for $x$ is $a.d_1d_2d_3\ldots d_n \pm .1*10^{-n}$. $\sqrt{2} = 1.41$ represents saying that $\sqrt{2}$ is in the interval $1.40:1.42$.
\item Big Uncertainty in Last Digit. The interval for $x$ is $a.d_1d_2d_3\ldots d_n \pm 5*10^{-n}$. $\sqrt{2} = 1.41$ represents saying that $\sqrt{2}$ is in the interval $1.36:1.46$. 
\end{itemize}

We will now redo the previous computation\footnote{Another computation done along these lines is in the overview paper \cite{taylor23over} which explores $x = \frac{e-\sqrt{2}}{\pi}$. Being division with a divisor greater than 1, it makes the interval computation more interesting though a little more contractive.} of $x = \sqrt{2}(e+\pi)$  using $\sqrt{2} = 1.41$,  $e=2.72$, $\pi = 3.14$, the single value computation yields $u = 8.2626$. To about that precision, $x$ is $8.2871$. This does suggest that the rough computation needs to be rounded or truncated to avoid giving a false precision. In what follows, two intervals for each method will be produced. The first interval is the interval level that the default interpretation would contain both the short computation result as well as the longer one. The second interval will be the interval that gets produced by doing interval arithmetic and then the appropriate representative to produce something that contains that interval will also be given. 

\begin{itemize}
    \item Short Rounding. $8.26$ represents the interval $8.255:8.265$ which contains $u$, but not $x$. $8.3$ represents $8.25:8.35$ and does contain both $x$ and $u$. Computing the implied interval, the result is $1.405(2.715 + 3.135):1.415(2.725 + 3.145)  = 8.21925:8.30605$. This does include $x$, but to translate it into a single number with the implied spread covering this entire interval would yield $8$ because $8.30605$ is not contained in $8.2$'s interval while $8.21925$ is not contained in $8.3$'s interval under short rounding. The implied interval for $8$ is $7.5:8.5$.
    \item Truncation. $8.2$ would work as the interval of $8.2:8.3$ contains both $x$ and $u$. Truncating $u$ to $8.26$ would yield the interval $8.26:8.27$ which does not contain $x$. The interval computation is $1.41(2.72+3.14):1.42(2.73+3.15)  = 8.2626:8.3496$. We do have $x$ in this interval and $8.2$ also covers $x$ and $u$, what we do not have is covering half the computed interval. For that, we need to use $8$ as the representative. Thus, as before, $8$ is the single number representative; it implies $8:9$ as the interval. 
    \item Long Rounding. Since $x-u > 0.02$, we need to use $8.3$ to cover it with the implied interval $8.2:8.4$. We could also use $8.2$ covering $8.1:8.3$. So either version will contain both $x$ and $u$. As for the interval computation, $1.40(2.71+3.13): 1.42(2.73+3.15) = 8.176: 8.3496$. Neither $8.2$ nor $8.3$ will have that whole interval contained in it. So $8$ is to be chosen implying $7:9$. 
    \item Big Uncertainty in the Last Digit. There are several representatives that can cover both $x$ and $u$. $8.26$ leads to the interval $8.21:8.31$. The interval computation is $1.36(2.67+3.09): 1.46(2.77+3.19) = 7.8336:8.7016$. The representative $8.3$ leads to $7.8:8.8$ which just barely covers the computed interval.  
\end{itemize}

The problem with decimal arithmetic is that the imprecision grows making it difficult to use a single representative without sacrificing claims of precision. The last two options for interpretation also imply multiple reasonable representatives for a given decimal and level of precision. 

How about calculators? The calculator used here will be a TI-84. The calculator reports  $e = 2.718281828$, $\pi = 3.141592654$, and $\sqrt{2} = 1.414213562$. The calculator seems to do short rounding to produce decimals, for example, if expanding $\pi $ just a little further, the result is  $3.14159265358$ which the calculator rounds up. Computing out the decimals, the calculator reports $8.287113966$ with the implication that the result should be in the interval $8.2871139655:8.2871139665$.  WolframAlpha reports it as $8.287113966317$ so it is indeed in that interval. 

This success is because the calculator is computing more digits than is shown. If the directly displayed digits were used for the numbers, then the result is $8.287113964$ which implies the interval $8.2871139635:8.2871139645$ and this is not an interval that contains the quantity being computed. 

In practice, the extra hidden precision allows simple computations to not worry about the growing imprecision of the results from multiple arithmetic operations. It may become an issue for massive computations, but that is what numerical analysis handles, often with intervals. The proposal here is to bring that viewpoint to a more general audience, trying to make it natural to incorporate interval containment to have accurate answers with known precision for all of our real number needs.

\subsubsection{Continued Fractions}\label{sec:con-frac}

Arithmetic can be done for continued fractions in the continued fraction form. The two sources I found for this are slides for a talk\footnote{See \url{https://perl.plover.com/yak/cftalk/} as well as what seems to be an unpublished version of Bill Gosper's exploration of these ideas and more: \url{https://perl.plover.com/yak/cftalk/INFO/} including an interesting idea about continued logarithms.} and a more recent blog post\footnote{\url{https://srossd.com/posts/2020-09-18-gosper-1/}}. 

The basic idea is that the form $$z = \frac{a +bx + cy + dxy}{e + fx + gy + hxy}$$ is closed under substitutions of the form $x = q + \frac{1}{u}$ and $y = p + \frac{1}{v}$. Indeed, making those substitutions and multiplying by $uv$ , we get 
$$z = \frac{d + (c+d p) u + (b+dq)v + (a+bp + cq + dpq)uv}{h + (g+hp)u + (f+hq)v + (e+fp+ gq + hpq)uv}$$

Depending on the operation, we can start with the various forms (any non-specified variable is 0): addition [$b=c=e=1$], subtraction [$b=e=1, c=-1$], multiplication [$d=e=1$], and division [$a=1,h=1$]. Then we just expand, iterating through the continued fraction expansion of $x$ and $y$. It can be more prudent to do $x$ and $y$ as separate steps, using some conditions on the coefficients to decide which one to get an input from. See the above references for details. Ultimately, that choice is an optimization and one can easily choose multiple paths to the same outcome. 

We do this until the floored versions of $\frac{a}{e}, \frac{b}{f}, \frac{c}{g}, \frac{d}{h}$ all agree in which case that is the next number in the continued fraction expansion of $z$. We then reduce $z$ by $z = r + \frac{1}{w} = \frac{m}{n}$ where $m$ and $n$ are the numerator and denominator of our standard form above. Working towards getting $w$, we first have $\frac{1}{w} = \frac{m}{n} - r = \frac{m-nr}{n}$. We then obtain $w$ by flipping to get $w = \frac{n}{m-nr}$. The final form is 
$$w =  \frac{e + fx + gy + hxy}{(a-er) + (b-fr)x + (c-gr) y + (d-hr)xy}$$

This is how we extract the continued fraction of our result. We do the extraction whenever the ``corners'' agree after flooring. The corners are where we put in 0 and $\infty$ in for $x$ and $y$. The number must be between the corners and so if they are all floored to the same number, then the number must have that as the integer part.

Implementing this for $\sqrt{2}(e + \pi)$, we start with the addition.\footnote{Unfortunately, I ended up using $e$ in the formulas as a variable and here we use it for the special symbol $e$. There does not seem any real danger of confusion, however, so we will leave them both.} We should note here that $e = [2; 1, 2, 1, 1, 4, 1, 1, 6, ...]$, $\pi = [3; 7, 15, 1, 292, 1, 1, 1, 2, \ldots]$, and $\sqrt{2} = [1; \bar{2}]$. We use a notation of subscripts to indicate the remainder up to that term. For example, $\pi = 3 + \frac{1}{\pi_1} =  3 + \frac{1}{7 + \frac{1}{\pi_2}}$. 

Addition starts with $z = e + \pi =  (2 + \frac{1}{e_1} ) + (3 + \frac{1}{\pi_1}) = \frac{e_1 + \pi_1 + 5 e_1 \pi_1}{e_1 \pi_1}$. Switching to a more compact version for the coefficients with the third row being the corners, i.e., 
$$\bigg\langle\begin{smallmatrix}
  a & b & c & d\\
  e & f & g & h \\
  [\frac{a}{e}] & [\frac{b}{f}] & [\frac{c}{g}] & [\frac{h}{d}] 
\end{smallmatrix}\bigg\rangle$$ we can begin with the representation of $\frac{x+y}{1}$ and expand from there: 
\begin{multline*}
\big\langle\begin{smallmatrix}
  0 & 1 & 1 & 0\\
  1 & 0 & 0 & 0\\
  0 & \infty & \infty & 0
\end{smallmatrix}\big\rangle
\xrightarrow{2, 3}
\big\langle\begin{smallmatrix}
  0 & 1 & 1 & 5\\
  0 & 0 & 0 & 1\\
  0 & \infty & \infty & 5 
\end{smallmatrix}\big\rangle
\xrightarrow{1, 7}
\big\langle\begin{smallmatrix}
  5 & 36 & 6 & 43\\
  1 & 7 & 1 & 7\\
  5 & 5 & 6 & 6 
\end{smallmatrix}\big\rangle
\xrightarrow{2, 15}
\big\langle\begin{smallmatrix}
  5 & 36 & 6 & 43\\
  1 & 7 & 1 & 7 \\
  6 & 6 & 5 & 5 
\end{smallmatrix}\big\rangle
\xrightarrow{1, 1}
\big\langle\begin{smallmatrix}
  1847 & 1969 & 2498 & 2663\\
  318 & 339 & 424 & 452 \\
  5 & 5 & 5 & 5 
\end{smallmatrix}\big\rangle
\\
\xrightarrow[5]{}
\big\langle\begin{smallmatrix}
  318 &339 & 424  & 452\\
  257 & 274 & 378 & 403 \\
  1 & 1 & 1 & 1
\end{smallmatrix}\big\rangle
\xrightarrow[1]{}
\big\langle\begin{smallmatrix}
  257 &274 & 378  & 403\\
  61 & 65 & 46 & 49 \\
  4 & 4 & 8 & 8
\end{smallmatrix}\big\rangle
\xrightarrow{1, 292}
\big\langle\begin{smallmatrix}
  403 & 11805 & 677 & 198319 \\
  49 & 14354 & 114 & 33395 \\
  8 & 8 & 5 & 5 
\end{smallmatrix}\big\rangle
\xrightarrow{}\cdots
\end{multline*}
The numbers above the arrow are the next terms in the continued fraction of the inputs while the numbers underneath are the numbers being extracted to represent the output. Here we have concluded $[5; 1, \ldots]$ as the first two terms of the continued fraction. We know the next term is either 5, 6, 7, or 8. It is, in fact, 6. 

To continue with our computation, we would next do the multiplication. As we do the multiplication, we will be pulling in terms from the sum. If we need more terms, we can go back to the sum and continue that process. In that way, these chains of arithmetic can be done just-in-time. 

The computation for the multiplication is as follows, where to get the first output term, we use the sum's continued fraction of $[5; 1, 6, 7, \ldots]$. We start with the representation of $\frac{xy}{1}$ and expand from there: 
\begin{multline*}
 \big\langle\begin{smallmatrix}
  0 & 0 & 0 & 1\\
  1 & 0 & 0 & 0\\
  0 & 0 & 0 & \infty
\end{smallmatrix}\big\rangle
\xrightarrow{1, 5}
\big\langle\begin{smallmatrix}
  1 & 5 & 1 & 5\\
  0 & 0 & 0 & 1\\
  \infty & \infty & \infty & 5 
\end{smallmatrix}\big\rangle   
\xrightarrow{2, 1}
\big\langle\begin{smallmatrix}
  5 & 6 & 15 & 18\\
  1 & 1 & 12 & 2\\
  5 & 6 & 7 & 9 
\end{smallmatrix}\big\rangle  \xrightarrow{2, 6}
\big\langle\begin{smallmatrix}
  18 & 123 & 42 & 287\\
  2 & 14 & 5 & 35  \\
  9 & 8 & 8 & 8 
\end{smallmatrix}\big\rangle  
\\
\xrightarrow{2, 7}
\big\langle\begin{smallmatrix}
  287 & 2051 & 697 & 4981\\
  35 & 250 & 84 & 600  \\
  8 & 8 & 8 & 8 
\end{smallmatrix}\big\rangle 
\xrightarrow[8]{}
\big\langle\begin{smallmatrix}
  35 & 250 & 84 & 600  \\
  7  & 51 & 25 & 181  \\
  5 & 4 & 3 & 3 
\end{smallmatrix}\big\rangle 
\end{multline*}

Thus after some laborious computations, we have the first term is 8. 

This procedure can do the arithmetic. But just as decimals have a carry problem, there is an analogous problem that can arise. Consider the case of multiplying the square root of 2 by itself. It should produce 2 whose continued fraction is just [2]. That means in the multiplication, the first extraction should terminate the computation. In fact, we will never be able to extract anything for certain: 
\begin{multline*}
 \big\langle\begin{smallmatrix}
  0 & 0 & 0 & 1\\
  1 & 0 & 0 & 0\\
  0 & 0 & 0 & \infty
\end{smallmatrix}\big\rangle
\xrightarrow{1, 1}
 \big\langle\begin{smallmatrix}
  1 & 1 & 1 & 1\\
  0 & 0 & 0 & 1\\
  \infty & \infty & \infty & 1
\end{smallmatrix}\big\rangle
\xrightarrow{2, 2}
 \big\langle\begin{smallmatrix}
  1 & 3 & 3 & 9\\
  1 & 2 & 2 & 4\\
  1 & 1 & 1 & 2
\end{smallmatrix}\big\rangle
\xrightarrow{2, 2}
 \big\langle\begin{smallmatrix}
  9 & 21 & 21 & 49\\
  4 & 10 & 10 & 25\\
  2 & 2 & 2 & 1
\end{smallmatrix}\big\rangle
\xrightarrow{2, 2}
 \big\langle\begin{smallmatrix}
  49 & 119 & 119 & 289\\
  25 & 60 & 60 & 144\\
  1 & 1 & 1 & 2
\end{smallmatrix}\big\rangle
\end{multline*}

The corners will always have the 1's and 2's meaning we can never be confident of our extraction. If we do an extraction, the process will try to correct itself with negative numbers. 

The oracles avoid this problem of undecidability for any intervals whose endpoints do not include 2. One can always find smaller inputs to the arithmetic to narrow the interval. But if an interval endpoint is 2, then all one can do is shrink the other endpoints. This is how the mediant process would handle this situation in which the first value to try is the answer, but we cannot establish it as such. Indeed, if we started out squaring the $\sqrt{2}$-Yes interval of $\frac{41}{29}:\frac{17}{12}$ to produce the $[(\sqrt{2})^2=T]$-Yes interval of $\frac{1681}{841}:\frac{289}{144}$, it would not distinguish $1:2$ or $2:3$, both of which are Yes intervals, but not seen as such from this interval representative. But it does say that $1:3$ is a Yes interval. Using $\frac{2}{1}$ to generate mediants of $1:2$ and $2:3$, we get candidates $\frac{3}{2}$ and $\frac{5}{2}$ to test in some fashion. If we use the endpoint $2$, as we typically do, it would not be decidable. But we can look at $\frac{3}{2}:\frac{5}{2}$ and see that it is a Yes-interval while both $1:\frac{3}{2}$ and $\frac{5}{2}:3$ are No-intervals. We can continue, testing and getting Yes for intervals such as $\frac{5}{3}:\frac{7}{3}, \frac{7}{4}:\frac{9}{4}, \ldots, \frac{1 + 2k}{1+k}:\frac{3+2k}{1+k}$. These will be seen to be Yes-intervals as long as $\frac{1+2k}{1+k} < \frac{1681}{841} < \frac{289}{144} < \frac{3+2k}{1+k}$. The first lasts until $k=839$ while the second lasts until $k=142$. So we can confidently say that the square of the square root of 2 is in the interval $\frac{284}{143}:\frac{287}{143}$. To get more precision, we need to extract a shorter interval for the square root of 2. 

The main point of this is that with intervals, we can make progress and make definite statements. Both continued fractions and decimals do not have mechanisms on their own to precisely get past these points of indecision. The interval approach mitigates the difficulty of uncertainties and oracles give us the maximal flexibility in taking advantage of that mechanism.

\section{Reflections and Generalization}

In this section, we reflect on the oracles, starting with a critique, then we consider some extensions, and finish with brief descriptions of some generalizations.

\subsection{Critiquing Oracles}

Let us apply our subjective criteria to our oracle definition. 

\begin{itemize}
    \item Uniqueness. For each real number, there is only one oracle that represents it. We have chosen a maximal representation of the number in terms of intervals that contain it. While it contains many intervals of enormous size that we would never be interested in, it also provides a mechanism for starting from any initial Yes interval and going down to an interval of a length we do care about, at least in theory. 
    \item Reactive. The choice of presenting this as an oracle was specifically to have this flavor. An alternative perspective of having a set that contains all the intervals that ``contain'' the real number, feels less reactive and draws the attention to many intervals that we would not care about. The oracle approach highlights exactly the intervals we care about as those are the ones being asked about. 
    
    We also do not need to make arbitrary choices. In a Cauchy sequence, for example, each value is chosen from an interval consisting of other equally good choices. For the oracles, no choice was made. The choice is made by the questioner. 
    \item Rational-friendly. Rationals are exactly those oracles that contain a singleton. That is very distinguishable in theory. In practice, it can be hard to decide whether the oracle is rooted or not. Arithmetic with rational oracles is no different than arithmetic with rationals. If we have the singletons, we can operate entirely on those singletons. Even if we are operating with an irrational and a rational, we can still use the singleton version to minimize the computations involved. Finally, for rational oracles, the mediant method, strongly supported by the oracle point of view, will stop at the rational if we are able to get an answer from the oracle about a given singleton. If not, then the process will still lead to a shrinking interval converging on it.
    \item Supportive. Intervals are very much needed in working with measurements and this puts intervals as a primary focus. 
    \item Delayed Computability. This fits perfectly for an oracle. One can specify the rule for intervals and then iterate an algorithm based on that. There is a need to come up with at least one Yes interval to satisfy the existence property, but this is generally not difficult to do as the interval can be quite large if need be. 
    \item Arithmeticizable. We can easily do the interval arithmetic and we have bounds for forcing the narrowing of the intervals. Answering an oracle's question does require a bit of figuring out how to narrow the intervals and then try to produce an interval which is either contained in the one we care about or disjoint from it. That is, our arithmetic is a little indirect for answering the question, but it has mechanisms for producing an answer. 
    \item Resolvability. The interval nature of an oracle is giving exactly how well we have resolved a number. 
\end{itemize}

If we think of real numbers as something that we can never hold in our hands, but can only tell bits about based on their shadows, it feels that the oracle approach has the most complete and direct shadow.\footnote{This is reminiscent of the ancient parable trying to describe an elephant from localized touching. The oracle is roughly the equivalent of having a mental model of what an elephant is to guide the interpretation and integration of the localized touching.} Dedekind cuts can be thought of as taking the lower endpoints of Yes intervals. It feels as if they are minimizing the information needed.  Nested sequence of intervals can be thought of the result of taking a particular pathway through an oracle line of questioning. Cauchy sequences can be very roughly thought of as taking centers from the Yes-intervals along that pathway. Infinite decimals can be thought of in a similar way, but with a much more constrained pathway. Continued fractions are also a pathway, but they are a distinguished pathway generated by the mediant process. Other approaches with representing the numbers as sums can be seen in a  similar light. Filters can be thought of as dressing up the intervals with extra elements.

Oracles do provide easy ways of discussing the completion properties. For Dedekind cuts, the least upper bound is trivial to define while the Cauchy sequences having a limit is rather more difficult to specify with the cuts. For the Cauchy sequence approach, Cauchy sequences are trivially part of the definition while the least upper bound requires coming up with a sequence, probably through an in-between process such as bisection. For oracles, both end up being very easy to understand. The least upper bound Yes-interval is one with an upper endpoint on the interval which is also an upper bound of the given set while the lower endpoint of the interval is a number less than or equal to an element of the set. The rule for defining the limit of the Cauchy sequence is to say Yes to intervals that contain the tail of the sequence, possibly adding in a singleton.  

Oracles are a balanced approach to real numbers. 

\subsection{Oracle vs Fonsi}

The idea of using rational intervals to represent real numbers is not new. From Bachmann, a sequence of nested shrinking intervals was contemplated. More recently, the idea of a fonsi defining a real number became popular among constructivists though they do not use this name.

The concept of a family of overlapping notionally shrinking intervals seems to be more in line with how we typically think of a real number in contrast to the oracle style. A maximal fonsi is equivalent to the set of Yes intervals for an oracle.  So why did we choose our oracle definition rather than the fonsi definition as our primary definition? 

The basic answer is that the oracle definition, particularly the separating property, is exactly what we need to construct further approximations. It feels more in line with how we would actually explore a given real number as we explored with the bisection and mediant methods. The family of shrinking intervals feels as if we ought to already have this whole construct and that we should complete it given something like an infinite series. All of this seems distracting from how real numbers are generally used. We cannot list out the explicit details of a real number. What we can do is to get in the neighborhood of a real number and use that neighborhood in deducing further information. Oracles focus on exactly the information we want in our customary use of a number. It is a very nice fact that we can get oracles from a fonsi as that greatly aids the initial starting points for the real number inputs into a calculation, but that is just the beginning of our journey. 

Another reason to go with the oracles is that it suggests that we want some kind of rule or algorithm for determining whether an interval is a Yes interval. It does not demand all the intervals at once, but, rather, it demands that we be able to produce an answer when we are given an interval. This makes it more in line with what we can do as humans and highlights where the difficulties could be. There are oracles, as we have seen, for which we cannot give a definitive answer yet with our current knowledge. But the oracle approach gives us a tool, the resolvability, to make that limitation known. 

A fonsi, on the other hand, promotes the idea of having an actual infinite set of them.  It obscures what we actually know versus what we assume we know. It also obscures the role of how to choose the intervals. Fonsis are, however, a very convenient mental model when trying to visualize what we are talking about and very useful with creating new oracles from existing oracles. 

Fundamentally, it feels that it is theoretically possible to have a rule that does say Yes or No for any given interval. Furthermore, with such a rule, we could be able to use the number quite effectively, particularly with the narrowing down methods, such as bisection or the mediant method. With a fonsi, we could never really have an infinite collection of intervals. At best, it could be a mechanism that given an $\varepsilon$, we get a Yes interval of that size. But we can get that from an oracle without pretending that we actually have this full set. And there are infinitely many ways to produce an interval for a given $\varepsilon$; that choice goes beyond what is encoded by the real number. 

Oracles avoid much in the way of irrelevant claims and choices.

\subsection{Unique Representative of Equivalence Classes for Numbers}

In constructing number systems, numbers can often be thought of as an equivalence class often with a distinguished representative:

\begin{enumerate}
    \item For natural numbers, the equivalence class for a number is the class of sets with that number of objects in the set. The equivalence relation is that of bijection. The distinguished representative is ``inductively'' built as the set that contains the previous numbers. Thus $0$'s representative is the empty set, $1$ is $\{0\}$, $2$ is the $\{0,1\}$, and so forth. 
    \item For integers, the fundamental object is a pair of natural numbers with the first element representing a positive amount while the second represents a negative amount. The equivalence relation is that of being related by the addition of the same quantity to both elements, i.e., $(a,b) \sim (c,d)$ exactly when $a+d = b+c$. The unique representative is the one in which at least one of the elements is $0$. Thus, $0$'s representative  is $(0,0)$, $1$ is $(1,0)$ and $-1$ is $(0,1)$
    \item For rationals, these are also pairs in which the first element is an integer, called the numerator, while the second is a positive natural number, called the denominator. The equivalence is based on multiplication. In particular,  $(p,q) \sim (r,s)$ exactly when $p*s = r*p$. The unique representative is the one with the smallest denominator. This leads to $0$ being represented by $(0, 1)$, $-2$ as $(-2,1)$, and $\frac{9}{27}$ as $(1, 3)$. 
\end{enumerate}

The typical real number definitions lack this. Cauchy sequences are usually given as equivalence classes, but there is no canonical representative except for rational numbers for which the representative would be the constant Cauchy sequence of the rational. Perhaps the closest canonical Cauchy sequence for an irrational number would be the continued fraction representatives. For Dedekind cuts, there are no equivalence classes. For many of the other representations, we can roll them up into a fonsi as we can with Cauchy sequences. For a fonsi, the equivalence relation is that two fonsis are equivalent if given any pair of intervals from the two fonsis, the intervals intersect. Alternatively, the union of the two fonsis needs to be a fonsi. The canonical representative for every real number in the realm of the fonsis is the maximal fonsi which is, of course, the set of Yes intervals for the oracle associated with the fonsis.

Having equivalent representatives can be useful for flexibility in computations. Fonsis are very useful. But having a single well-defined representative also gives a clear way to communicate. The maximal fonsi serves this role, but it is most compactly described by the oracle presentation. 

\subsection{Other Dense Sets}

This setup also allows us to consider other foundational sets different from the rationals, but that are still dense in the real numbers. 

The easiest to handle are those that are fields themselves. This would be the algebraic extension fields such as adding $\sqrt{2}$ as a primitive. As long as the extension is closed under arithmetic, then the intervals defined by the elements of the field can properly participate in the arithmetic to define the oracle arithmetic. 

Other dense sets, such as rationals whose denominators are powers of 2, appropriate for binary computations, are not closed under arithmetic. We can still use them for representing intervals and numbers, but the arithmetic of those intervals no longer produces intervals drawn from that set. In particular, the reciprocal of, say, the interval $\frac{3}{4}:1$ is $\frac{4}{3}:1$ which does not have a power of 2 denominator.  

Other definitions, such as nested intervals, fonsis, and Cauchy sequences are compatible with these notions by specifying what the relevant elements are allowed to be. Specific representations are generally not compatible with this, generally needing arbitrary rationals to express themselves, such as with continued fractions. Dedekind cuts and minimal Cauchy filters are compatible, but it is less clear with those sets what the impact would be. 

There is something tangibly different with oracles versus the other definitions in that this informs which singletons are allowed, impacting both exact arithmetic and functions. As detailed elsewhere, functions can be modelled as oracles and in that context, different choices change where discontinuities are allowed. 

This foreshadows which definitions of real numbers can be extended easily and meaningfully to general metric spaces.

\subsection{Extended Reals}

We can also extend the Oracles to include an extended version of the real numbers that includes $\pm \infty$. We need to include infinite intervals in our definition and with an understood updating of our existence clause to include being 1 on an unbounded interval.

For unbounded intervals, we could write that as $a:$ or $a:\infty$ for all rationals greater than or equal to $a$ and write $:a$ or $-\infty:a$ for all rationals less than or equal to $a$. The special interval $-\infty:\infty$ consists of all rationals and is a Yes interval for all oracles in this extended framework. 

If we changed the Existence property to the assertion that $R(-\infty:\infty)=1$, we could then define the Oracle of $-\infty$ as the rule $R(:a) = 1$ for all rational $a$ and 0 otherwise. The Oracle of $\infty$ is similarly defined as $R(a:)=1$  for all rational $a$ and 0 otherwise. 

Arithmetic of these intervals should largely work out as one would expect. Certain combinations, such as $\infty-\infty$, do not lead to useful results. 

What this allows us to do is to then include infinite limits. Indeed, the definition of $\lim_{n\to \infty} f(n) = \infty$ would mean that the collection of intervals that contain the tail of $f$ would coincide with the collection of intervals $a:$. If we were considering $\lim_{x \to \infty} f(x) = L$, we would interpret this as saying that the collection of intervals that contain the image of $a:$ under $f$, would be the oracle of $L$, which could be infinite.  Finally, $\lim_{x \to \alpha} f(x) = \infty$ would be saying that the image of the intervals of $\alpha$, possibly excluding $f(\alpha)$, are all contained in the intervals of $a:$. 

As a final note, we can link this to our mediant method by noting that $\frac{1}{0}$ could be viewed as a stand-in for $\infty$ in the process while $\frac{-1}{0}$ is a stand-in for $-\infty$. Taking the mediant of those two numbers leads to $\frac{0}{0}$ which we swiftly push into the form $\frac{0}{1}$. We can then decide which of the two infinite intervals to choose from using our oracle functionality. If it is the left one, we know it is a negative number. If it is the right one, we know it is positive. We can then proceed to compute the integer portion and then the mediant expansion from there as detailed in Section \ref{sec:mediant}. If we are unable to figure out the sign, then we can shrink our intervals towards $0$ until we hit the limits of our computational abilities. 

What we have explored here is one compactification of the reals. Another would be to adjoin a single point at infinity. We could do this by considering the oracle that says yes to anything of the form $-\infty:b \cup a:\infty$. This is going beyond rational intervals just a bit. Unlike completion, compactification can be done in various ways. For further explorations of this, see \cite{taylor23metric}.

\subsection{Generalizations to Topological Spaces}

The idea of oracles can be generalized to other spaces. In \cite{taylor23metric}, I plan to explore how oracles interact with four different kinds of topological arenas: metric spaces, general topological spaces, topological groups, and topology based on general lines that give a notion of betweeness for points on the same line \cite{maudlin}. For each space, we have a different notion of what the relevant replacement for closed intervals is. The term container is used to replace the term interval.

For metric spaces, the containers are closed balls. The Separation Property fails to be relevant, but we can use the Two Point Separation Property. It can be shown that the set of oracles of a metric space form a completed metric space with the induced metric. 

There are more general topological spaces which oracles may be used in. For those, we may be able to use closed sets. We seem to lose any sense of convergence, but we can modify our approach to do Stone-\v{C}ech compactifications. This is inspired by our exploration of the extended real-line. 

For topological groups, if we use oracles to produce new spaces, does the group structure always go along for the ride? 

A more recent theory of topology is that based on linear structures; see \cite{maudlin} to learn how to view lines as the primitive structures to base topology on and straight lines as the basis for differentiable manifolds. We can use the closed neighborhoods which are essentially stand-ins for closed balls without the metric. Note that much of the commonly understood terminology is redefined so as to extend it usefully to finite and discrete spaces. 

All four of these different approaches are entwined in the completion of the rational line to the reals and it is an interesting project to see what does or does not carry over. 

\subsection{Function Oracles}

The oracles are equivalent to the axiomatic real numbers and we could, therefore, just do the usual story of pointwise defined functions being maps from oracles (reals) to oracles. But that not only discards the idea and usefulness of oracles, but it also gives up an opportunity to have the theory be based more in line with how we actually compute functions. When we compute $e^{\sqrt{2}}$, for example, we might use the Taylor polynomial for $e^x$ which is a partial approximation, but the input of $\sqrt{2}$ is also a partial approximation. While the easiest approach is to simply take an output number that hopefully represents something close enough, a more precise approach is to develop some notion of saying ``this interval maps into this other interval.'' This is what functional oracles will do as detailed in \cite{taylor23funora}. They allow one to make clear, precise statements about results that are inherently imprecise. As we have throughout, we do allow intervals to be rational singletons. 

With a function oracle, there is a fundamental difference between rationals and irrationals in this setup. Indeed, we will eventually derive that the function oracles represent pointwise functions which are continuous at the irrationals in their domain, but possibly discontinuous at the rationals.

A natural structure to look at as providing a basis for functions is that of rational rectangles. The base will represent the valid inputs and the wall will represent the possible outputs. The family of rectangles will have the property that we can narrow the range of the outputs with a possible narrowing of the input range. A basic requirement will be that if an oracle $\alpha$ is in the domain of the function with supposed value $\beta$ at that point, then, given an $\alpha$-Yes interval $a:b$ and a $\beta$-Yes interval $c:d$, there is a Yes rectangle contained in $a:b \times c:d$ such that the base is $\alpha$-Yes and the wall is $\beta$-Yes. 

We could imagine that the rectangle $a:b \times c:d$ just needs to be an $(\alpha, \beta)$-Yes rectangle for a given $\alpha$ and $\beta$. This would be essentially the way to recover the usual set of functions that we discuss in real analysis as we would have each collection of rectangles tied to a given pair. But we have in mind the desire that these Yes rectangles contain all the function values for the inputs contained in the base of the rectangle. This implies that over the base, the rectangle is a bounding rectangle for our function. 

For rationals, we do allow singleton bases which enables them to be independent of the surrounding regions. For irrationals, we are effectively requiring them to be in line with their neighbors. The model to have in mind is the difference between discrete probability (probability mass function) and a continuous probability (probability density function). Just as we can mix these in probability, we can have that here. The functions that get defined by this should be continuous on the irrationals, but can be discontinuous on the rationals. Thomae's function\footnote{\url{https://en.wikipedia.org/wiki/Thomae\%27s_function}} is an example of this. 

The paper is a sketch of the foundations for function oracles as a basis of analysis. We do establish the arithmetic of function oracles, including the exponential operator of oracles raised to an oracle power. After we explore continuity in the context of function oracles, we then also discuss composition which, even in its definition, is strongly linked to continuity questions in this framework. The main continuity result is that function oracles are equivalent to pointwise functions continuous everywhere on its domain except possibly at the rational numbers. We also give a few examples of how common functions are modeled with function oracles. We also explore using the generalizations to other topological spaces as a basis for function oracles and what structures we might need for some of the analysis results to generalize.

\section{Conclusion}

We have given a new definition of real numbers in terms of oracles. We established that these are, indeed, the real numbers with all of the necessary properties proven. We gave some examples of arithmetic with them. We discussed some methods for obtaining good rational representations for such objects.  We compared and contrasted oracles with other definitions of real numbers. 

The advantage of oracles is that it should be quite approachable to understanding the definition, the goal, and the use of these objects. The basic idea is that we want to know whether a given interval contains the real number. Interval representations then become very natural as are the bisection and mediant methods for obtaining new intervals. 

Oracles also push the idea of interval manipulation which is something that would be very useful in making numbers more understandable. Though I have no scientific evidence, my experiences in teaching over the past two decades is that a useful precision can help much in the confusion of students learning material. In the current state of education, it is highly problematical to talk about what the square root of 2 actually is. Students know what it should be used for and know a few of its first digits, but there is the mystical notion of it being an infinite string of digits. An interval approach, even if it is simply using the decimal intervals, can bring a more comfortable understanding. Specific suggestions for basing mathematics education on rationals will be presented in \cite{taylor23edu}. Also, I am working on the website RatMath\footnote{\url{https://ratmath.com}}, under current development, which will be an online calculator and exploration tool for rational mathematics. 

Oracles avoid the ambiguity and work of choosing which representative to use for the Cauchy sequence definition. In contrast to Dedekind cuts, oracles have a wide, immediate practical purpose. Some of the other approaches such as nested intervals, ultrafilters, and Cauchy sequences, feel as if one has to choose between a definition that is too thin or one which has been maximized to such an extent that the core identity of the number has been entirely lost in the noise. I feel that oracles are at just that correct level of maximality to avoid non-uniqueness and yet not so large that one loses sight of the number. 

The specific approach advocated here emphasizes a dynamic approach to the number. The person asking about the number ought to have a particular purpose that can then be met by the tools available. The other approaches often have a shadow of that in the definition, such as the $\varepsilon$ in a nested interval sequence, but the presentation pushes that aside. This approach, by centering the activity of asking, highlights the dynamism even while we acknowledge the equivalence to a more static view via the maximal fonsi.  

Finally, this approach highlights the core difficulty of a real number, namely, that we do not have a direct way of knowing an irrational number. We are seeking information about a real number, but we can never have it all precisely in our hands as we do with rational numbers. We just know what neighborhood it is in. It is as if we are looking for a person and we know enough details to rule out billions of people, but we cannot narrow it down to just a single person. It is a merit of this approach that it brings this out as the central issue of dealing with real numbers as well as giving us the tools to do the asking. 

Oracles facilitate computation without requiring it. It is a map to the real number, one which goes beyond asserting existence, allowing us to travel towards it, but not suggesting that we could ever reach it. 

\appendix

\section{\texorpdfstring{$n$}{n}-th Power Inequality Facts}\label{app:A}

This is a place to collect some technical facts that are commonly known about $n$-th powers, but that we would like to have collected proofs of for completeness.

\begin{lemma}
$x^n$ is monotonic for $x>0$. That is, if $0 < a < b$, then $0 < a^n<b^n$.
\end{lemma}

While this is very basic, we rely on monotonicity of $x^n$ repeatedly so it seems prudent to include it. The form here is also useful to highlight. 

\begin{proof}
$b^n-a^n= (b-a)\sum_{k=0}^{n-1} b^k a^{n-1-k}$. Since both $a$ and $b$ are positive, the sum is positive. The sign is therefore determined by $b-a$. If $b>a$, then $b^n-a^n > 0$ as was to be shown. 
\end{proof}

This also demonstrates that $ n m^{n-1} (b-a) \leq  b^n -a^n \leq n M^{n-1} (b-a)$ for $0 < m < a< b< M$ by replacing the placing the products in the sum with $m^{n-1}$ and $M^{n-1}$, respectively. 

If $ a< b< 0$, then $-a > -b > 0$ and $(-a)^n > (-b)^n$. If $n$ is even, then $(-a)^n = a^n$ and we have that it is monotonically decreasing. If $n$ is odd, then $(-a)^n = -a^n$, and therefore $a^n < b^n$ which tells us it is monotonically increasing. 

\begin{lemma}\label{app:lesser}
Let $r \geq 0 $ and $q > 0$ be rational numbers such that $r^n < q$. Then there exists a rational number $s$ such that $r < s$ and $s^n < q$.
\end{lemma}

The basic idea is to find $N$ for $s = r + \tfrac{1}{N}$ such that $s^n < q$. We use the completely rational binomial theorem.  

\begin{proof}
Define $a = q - r^n$. Define $N =  \max(1,\tfrac{3}{a}n r^{n-1}, \tfrac{3}{a}(r+1)^n)$.  Take $s = r + \tfrac{1}{N}$. Then $s^n = (r+ \tfrac{1}{N})^n = r^n + \tfrac{n r^{n-1}}{N} + \sum_{k=2}^{n} \binom{n}{k} \tfrac{r^{n-k}}{N^k}$. We can factor out a $\tfrac{1}{N}$ in the sum\footnote{Actually, we could do a $\frac{1}{N^2}$, but that does not seem needed here.} and, since $N \geq 1$, we have $\tfrac{1}{N^i} \leq  1$ for any natural number $i$. Thus, $\sum_{k=2}^{n} \binom{n}{k} \tfrac{r^{n-k}}{N^k} < \tfrac{1}{N} \sum_{k=2}^{n} \binom{n}{k} r^{n-k}$  But that sum is part of the expansion of $(r+1)^n$ and is therefore bounded by it since those are all positive terms thanks to $r$ being positive. Thus, we have $s^n < r^n + n \tfrac{r^{n-1}}{N} + \tfrac{ (r+1)^n }{N}$.  By definition, we have $N > \tfrac{3}{a} n r^{n-1}$ implying $\tfrac{ n r^{n-1}}{N} < \tfrac{a}{3}$. We also have $N > \tfrac{3}{a} (r+1)^n$ implying $ \tfrac{(r+1)^n}{N} < \tfrac{a}{3}$. Therefore $s^n < r^n + \tfrac{2 a}{3} < q$. Since $r< r + \frac{1}{N} = s$, we have shown our result. 
\end{proof}

\begin{lemma}\label{app:greater}
Let $r \geq 0 $ and $q > 0$ rational numbers such that $r^n > q$. Then there exists a rational number $s$ such that $r > s$ and $s^n > q$.
\end{lemma}

The trick here is to consider $(r-\tfrac{1}{N})^n$ instead. 

\begin{proof}
Define $a = r^n - q$. Define $N =  \max(1,\tfrac{3}{a}n r^{n-1}, \tfrac{3}{a}(r+1)^n)$.  Take $s = r - \tfrac{1}{N}$. Then $s^n = (r- \tfrac{1}{N})^n = r^n - \tfrac{n r^{n-1}}{N} + \sum_{k=2}^{n} \binom{n}{k} \tfrac{ (-1)^{k} r^{n-k}}{N^{k}}$. We can factor out a $\tfrac{1}{N}$ in the sum and, since $N \geq 1$, we have $\tfrac{1}{N^i} \leq 1$ for any natural number $i$. Since we are looking to prove $s^n > q$, making the expression smaller is fine. If we replace any positive terms in the sum with negative terms, we will make it smaller. So $s^n > r^n - \tfrac{n r^{n-1}}{N} - \sum_{k=2}^{n} \binom{n}{k} \tfrac{r^{n-k}}{N^{k}}$. As before, $\sum_{k=2}^{n} \binom{n}{k} \tfrac{r^{n-k}}{N^{k}} < \tfrac{1}{N} \sum_{k=2}^{n} \binom{n}{k} r^{n-k}$.  But that sum is part of the expansion of $(r+1)^n$ and is therefore bounded by it since those are all positive terms thanks to $r$ being positive. Thus, we have $s^n > r^n - n \tfrac{r^{n-1}}{N} - \tfrac{ (r+1)^n }{N}$.  By definition, we have $N > \tfrac{3}{a} n r^{n-1}$ implying $ \tfrac{-n r^{n-1}}{N} > -\tfrac{a}{3}$. We also have $N > \tfrac{3}{a} (r+1)^n$ implying $ -\tfrac{(r+1)^n}{N} > -\tfrac{a}{3}$. Therefore $s^n > r^n - \tfrac{2 a}{3} > q$. Since $r> r - \frac{1}{N} = s$, we have shown our result. 
\end{proof}

\section{Detailed \texorpdfstring{$e$}{e} Computations}\label{app:e}

These are supporting computations for Section \ref{sec:e}. While not required in the sense that these computations exist elsewhere, it is nice to give the flavor of some of the work that can be involved in doing it in the more explicit fashion we favor here. 

For this section, let $S_n = \sum_{i=0}^n \frac{1}{i!}$, $a_n = (1+\frac{1}{n})^{n}  $, and $b_n = (1+\frac{1}{n})^{n+1}$

We will use the Arithmetic-Geometric Mean Inequality in the form that states that the $n$-power of the average of the sum of $n$ numbers is greater than the product of those $n$ numbers, assuming at least one of the numbers is different from the others. 

While this is a very standard result, for completeness, we include the proof of the Arithmetic-Geometric Mean Inequality. See Wikipedia for a variety of different proofs.\footnote{\url{https://en.wikipedia.org/wiki/Inequality_of_arithmetic_and_geometric_means}}

\begin{proposition}[Arithmetic-Geometric Mean Inequality]
    Let $\{x_i\}_{i=1}^n$ be a given set of numbers. Then $\big(\frac{\sum_{i=1}^n x_i}{n}\big)^n \geq  \prod_{i=1}^n x_i$ with equality only if all of the $x_i$ are equal to one another. 
\end{proposition}

\begin{proof}
    Let $A = \frac{\sum_{i=1}^n x_i}{n}$. If there exists $x_j \neq A$, then we must have a partner to $x_j$, say $x_k$, such that $A$ is between $x_j$ and $x_k$.\footnote{For example, if $x_1 > A$ and $x_k \geq A$ for all other $k$, then $\sum_{i=2}^n x_i \geq (n-1) A$ and the final addition of $x_1$ would lead to strict inequality: $x_1 + \sum_{i=2}^n x_i > n A$ but $n*A = \sum_{i=1}^n x_i$ by definition of $A$.} The trick is to replace $x_j$ with $A$ and $x_k$ with $x_j + x_k - A$. Since their sum is just $x_j + x_k$, then sum is unchanged. The product, however, is changed as $A(x_j + x_k - A) = Ax_j + A x_k - A^2 > x_j x_k$. The inequality follows from the fact that if $x_j > A > x_k$, then $(x_j - A)(A-x_k) > 0$ as both factors are positive, but expanding out, we get $A x_j + Ax_k - A^2 - x_j x_k > 0$. Moving the last term to the other side, we get our product inequality. 

    Iterating this process on our finite collection of numbers, we gradually replace all of the elements that are not equal to $A$ with elements that are with the final replacement also managing to replace the smaller element with $A$.\footnote{The last step has the form $x_j + x_k + (n-2)A = n A$ leading to $x_j +x_k = 2A$ which is also $x_j + x_k - A = A$.} Since $A$ never changes, but the product gets larger, the original product must be smaller. We also see that equality of the power of the average with the product can only be achieved when all the $x_i$ are equal.  
\end{proof}

The next two lemmas come from \cite{mend}.

\begin{lemma}
  $ a_n < a_{n+1}$ 
\end{lemma}

\begin{proof}
Consider the collection of $n+1$ numbers consisting of $1$ and $n$ copies of $1+1/n$. We need to compute the average and product to apply the AGMI. The sum is $1 + n* (1 + \frac{1}{n} ) = n+2$. If we take the average of the $n+1$ terms and raise it to the $n+1$ power, we have $(\frac{n+2}{n+1})^{n+1} = (1 + \frac{1}{n+1})^{n+1} = a_{n+1}$. The product of the numbers is $1*(1+\frac{1}{n})^n = a_n$. The AGMI then is exactly the statement of $a_n < a_{n+1}$.
\end{proof}

\begin{lemma}
 $b_{n+1} < b_n$
\end{lemma}

\begin{proof}
This is similar to the previous proof but we use the collection $1$ and $n$ copies of $1 - \frac{1}{n}$. The sum is $1 + n*(1-\frac{1}{n}) = n$. The average of the $n+1$ numbers raised to the $n+1$ power is then $(\frac{n}{n+1})^{n+1}$. Since the reciprocal of $\frac{n}{n+1}$ is $\frac{n+1}{n} = 1+\frac{1}{n}$, we have the powered average is the same as $\frac{1}{b_n}$. The product of the terms is $(1-\frac{1}{n})^n$. But $1 - \frac{1}{n} = \frac{n-1}{n}$ has reciprocal $\frac{n}{n-1} = \frac{ (n-1) + 1}{n-1} = 1 + \frac{1}{n-1}$. Therefore, the product is the same as $\frac{1}{b_{n-1}}$. The AGMI then tells us that $\frac{1}{b_n} > \frac{1}{b_{n-1}}$. Since both terms are positive, we can reciprocate, flipping the inequality leading to $b_{n-1} > b_n$ for all $n$ which is the same statement as $b_{n+1} < b_n$ for all $n$. 
\end{proof}

The reference for the next two lemmas is \cite{rudin}, page 64. The second lemma differs from that of Rudin's presentation in that Rudin uses a limiting infimum argument to avoid computing the explicit values necessary for achieving the desired inequality. 

\begin{lemma}\label{lem:ansn}
$a_n < S_n$
\end{lemma}

\begin{proof}
The binomial expansion of $a_n$ is $\sum_{i=0}^n \frac{1}{i!} \prod_{j=1}^{i-1} (1-\tfrac{j}{n})$ (for $i=0$ and $i=1$, we take the product to be 1). By replacing each factor in each product with 1, we are making the expression larger. When done to completion, we have turned the sum into $S_n$. So $a_n < S_n$.
\end{proof}

\begin{lemma}\label{lem:snam}
For a given $m$, there exists an $n$ such that $S_m < a_n$.   
\end{lemma}

\begin{proof}
 We want to find $n$ such that $\sum_{i=0}^m \frac{1}{i!} < (1+\frac{1}{n})^n$. Our $n$ will be larger than $m$ so we can truncate the binomial expansion at $m+1$: $\sum_{i=0}^{m+1} \frac{1}{i!} \prod_{j=1}^{i-1} (1-\tfrac{j}{n})$, making a smaller quantity which is okay since we want this to be an upper bound. By choosing $n$ large enough, the products should be near enough to 1 that the $\frac{1}{(m+1)!}$ term should be large enough to make the expanded sum larger than the sum to $m$. 

 Specifically, we choose $\varepsilon < \frac{1}{(m+1)! S_{m+1}  }$ and from this we choose\footnote{The  following inequality is equivalent to, and generated from, $(1-\tfrac{m}{n})^m > 1 - \varepsilon$. Also, $\ln(x) \approx m (\sqrt[m]{x} -1)$. This suggests the inequality can be viewed as $n > \frac{m}{1 - \sqrt[m]{1-\epsilon}} \approx \frac{m}{-\ln(1 - \epsilon)/m} \approx \frac{m^2}{\varepsilon} > m^2 (m+1)! S_m$. So $n$ is quite large for large $m$.}  $n > \frac{m}{ 1 - \sqrt[m]{1-\varepsilon}}$. With these choices, let us prove that $S_m < a_n$. 
 
 With our choices, we claim that $\prod_{j=1}^{i-1} (1-\tfrac{j}{n}) > 1-\varepsilon$ for all $i \leq m+1$. Observe that $1-\tfrac{j}{n} > 1 -\tfrac{m}{n}$ for all $j < m$. Then rearranging the inequality we have for choosing $n$, we have $(1 - \tfrac{m}{n})^m > 1 - \varepsilon $. By possibly replacing factors of 1 with $1-\tfrac{m}{n}$, we can view all of the products as being greater than $(1 - \tfrac{m}{n})^n > 1-\varepsilon$.  Thus, we have 
 $a_n >  \sum_{i=0}^{m+1} \frac{1}{i!} \prod_{j=1}^{i-1} (1-\tfrac{j}{n}) > (1-\varepsilon) S_{m+1} = S_m + \frac{1}{(m+1)!} - \varepsilon S_{m+1}$.  Now, $\varepsilon$ was chosen so that $\varepsilon < \frac{1}{(m+1)! S_{m+1}}$ which is equivalent to $- \varepsilon S_{m+1} > -\frac{1}{(m+1)!}$. We therefore have $a_n > S_m + \frac{1}{(m+1)!} - \varepsilon S_{m+1} > S_m + \frac{1}{(m+1)!} - \frac{1}{(m+1)!} = S_m$. We have established our result. 
\end{proof}

\section{Direct Oracle Construction of the Supremum}\label{app:sup}

We are given a non-empty set of oracles $E$ with an upper bound $M$, meaning that if $x \in E$ then $x$ has the property $x < M$. We define the set $U$ to be the set of oracles that are upper bounds of $E$, namely, $y \in U$ if $y \geq x$ for every $x$ in $E$.  Thus, $M \in U$.

We want to create the oracle for the least upper bound of $E$. We define the rule $R$ such that $a\lt b$ is to be a Yes interval exactly when $a$ is a lower endpoint for a Yes-interval of an oracle in $E$ and $b$ is an upper endpoint for a Yes-interval of an oracle which is an upper bound of $E$. We also include $c:c$ as a Yes singleton if $c \geq x$ for all $x \in E$ and $c \leq y$ for all $y\in U$.

We claim that this satisfies the properties of an oracle and we call this the Oracle of $\mathrm{sup} E$, the supremum of $E$ or least upper bound. This oracle is also the greatest lower bound of $U$.

Let us establish the properties: 

\begin{enumerate}
    \item Consistency. Assume $R(a \lte b)=1$ and $c\lte d$ contains $a\lte b$. Then there exist oracles $x \in E$ and $y \in U$ with $x$-Yes interval $a\lte A$ and $y$-Yes interval $B\lte b$. By consistency of $x$ and $y$ being oracles, $c\lte A$ is an $x$-Yes interval and $B\lte d$ is a $y$-Yes interval. Thus, $R(c\lte d) = 1$.
    \item Existence. By assumption, there is an $x \in E$ and an upper bound $M$ in $U$. Let $a\lte b$ be an $x$-Yes interval and $c\lte d$ be an $M$-Yes interval. Then $R(a\lte d) = 1$ by definition.  
    \item Separating. Let $R(a\lt b)=1$ and $a < c< b$. We need to show that either $R(c:c)=1$ or $R(a:c) \neq R(c:b)$. 

    We proceed by cases.
    \begin{enumerate}
    \item There exists an oracle $y$ in $U$ such that $c > y$. Namely, there exists an $m \lte n$ $y$-Yes interval such that $c > n$.  This implies $c$ is an upper endpoint of a Yes-interval of $y$, specifically $m \lte c$. Thus, $R(a \lte c) = 1$ by definition. Also, $c$ cannot be a lower endpoint of a Yes interval for an  $x \in E$ since $y \geq x$ for all $x \in E$ implying all the lower endpoints of such $x$ are less than or equal to $y$'s upper endpoints and $c$ is strictly greater than them. Thus, $R(c \lte b) = 0$. 

    \item There exists an oracle $x$ in $E$ such that $c < x$. Namely, there exists an $m \lte n$ $y$-Yes interval such that $c < m$. This implies $c$ is a lower endpoint of a Yes-interval of $x$, specifically $c \lte n$. Thus, $R(c \lte b) = 1$ by definition. Also, $c$ cannot be an upper endpoint of a Yes interval for a $y \in U$ since $x \leq y$ for all $y \in U$ implying all the upper endpoints of such $y$ are greater than or equal to $x$'s lower endpoints and $c$ is strictly greater than them. Thus, $R(a \lte c) = 0$.

    \item $c \geq x$ for all $x \in E$ and $c \leq y$ for all $y \in E$. Then $R(c:c) = 1$ by the definition. 

     \end{enumerate}
    
    \item Rooted. For $R(c:c)=1$, we would need to have that $c \leq y$ for all $y$ in $U$ and $c \geq x$ for all $x$ in $E$. Assume we had another rational such that $R(d:d) = 1$. Then we also have $d \geq x$ for all $x$ in $E$ and $d \leq y$ for all $y$ in $E$. Since both $c$ and $d$ are upper bounds of $E$, we have that each are in $U$. So $c \leq d$ and $d \leq c$. This implies that $c = d$. 
    \item Closed. Assume $c$ is in all Yes intervals. By the definition, $R(c:c) = 1$ unless either there exists an $x \in E$ such that $x> c$ or there exists a $y \in U$ such that $y < c$. If there is such an $x$, then let $a\lte b$ be an $x$ Yes-interval such that $c < a$; this exists by the definition of oracle inequality. Therefore, $a:B$ is a Yes-interval of $\mathrm{sup} E$, where $B$ is an upper endpoint of any $M$-Yes interval. This is a Yes-interval which does not contain $c$. So we have ruled out this case. The other case goes similarly.
\end{enumerate}

We also have to establish that for every $x \in E$ and $y \in U$, we have $x \leq \mathrm{sup} E \leq y$. Let us assume this is not true. Then either there exists $x > \mathrm{sup} E$ or $y < \mathrm{sup} E$. Both cases proceed similarly. Let us 
assume $ x > \sup E$. Take $a \lte b$ to be an $x$-Yes interval such that there exists $c \lte d$, a $\sup E$-Yes interval with $b > c$; this is the meaning of being greater than for oracles. But then $b$, when paired with an upper endpoint from an upper bound, forms a $\sup E$ Yes-interval which is disjoint from $c:d$. This is not allowed for an oracle. Hence, there is no such $x$. Similarly for $y < \sup E$. 

\section{Obtaining a Rational in the Mediant Process}\label{app:med}

We claim that if we are looking for $\frac{e}{f}$ in the interval $\frac{a}{b} \lt \frac{c}{d}$, then the mediant process will produce $\frac{re}{rf}$ where $r$ is some non-negative integer and we take $e$, $f$ to be coprime. 

We shall establish this in two main steps. The first step is to apply the solution to some linear equations to establish the existence of certain integers that form a weighted mediant. The second step is to argue that a integrally weighted mediant can also be realized.  We will rely on the specialized Farey process applied to this more general setting. 

But first, let us explore with some examples. 

If we have two rationals $\frac{a}{b} < \frac{c}{d}$ such that $bc - ad = 1$, then they are a Farey pair and the mediant process preserves this condition. In this case, the mediant process leads to the rational number being achieved in this process and it is in reduced form. This is what Theorem 1 in \cite{richards} states explicitly. The argument is that such Farey pairs always produce a mediant which is the closest rational to any number in the interval with the smallest denominator. Given that, if a rational is the target of the process, then it must be achieved as it will be its own best approximation at that level. 

This theorem does not directly apply to non-Farey pairs. If $bc-ad > 1$,\footnote{This is the only other case since $\frac{a}{b} < \frac{c}{d}$ implies $bc - ad > 0$ and this is an integral quantity.} every rational in the interval will be obtained by a mediant process, but it will not be in reduced form. In general, one will have a scaling factor of $bc-ad$ for the form though it can be less as the example of $\frac{7}{8}< \frac{11}{12}$ with a target of $\frac{9}{10}$ demonstrates; $bc-ad = 4$, but the process obtains $\frac{18}{20}$ instead of $\frac{36}{40}$. Note that there are some common factors involved here. In contrast, if we used the same interval but wanted to converge to $\frac{8}{9}$, we would end up with the fraction $\frac{32}{36}$ as the sequence of mediants $\frac{18}{20}, \frac{25}{28}, \frac{32}{36}$ establishes. Another random example is  $\frac{1}{4} : \frac{19}{9}$ with target $\frac{2}{1}$. The scaling factor is $67$ and the mediants do indeed become $\frac{2 \cdot 67}{1 \cdot 67} = \frac{134}{67}$  ($\frac{20}{13}, \frac{39}{22}, \frac{58}{31}, \frac{77}{40}, \frac{96}{49}, \frac{115}{58}, \frac{134}{67} = 2$).

To establish our claim, we want to find non-negative integers $m$, $n$, and $r$, with $m$ and $n$ coprime. When we form our first mediant, we get $\frac{a+c}{b+d}$. As we proceed with the processes, we will in general have the form $\frac{ma + nc}{mb + nd}$ for integers $m$ and $n$. One can prove this by seeing that taking the mediant of two numbers of this form will produce another number of that form. Since we start off with two numbers in this form, specifically, $m=1, n=0$ for $\frac{a}{b}$ and $m=0, n=1$ for $\frac{c}{d}$, we will stay with this form. 

Therefore, we need to find non-negative integers $m$ and $n$ such that $\frac{ma+nc}{mb+nd} = \frac{e}{f}$. This is equivalent to finding non-negative integers $m$ and $n$
satisfying the equations
\begin{align*}
    am + cn &= re & \\
    bm + dn &= rf & 
\end{align*}
This has solution, $m = \frac{r}{bc-ad} (cf-de)$ and $n= \frac{r}{bc-ad} (be-af)$. When $\frac{r}{bc-ad}$ is an integer, then $m$ and $n$ will be integers. We also need these to be non-negative. Since $\frac{a}{b} < \frac{e}{f} < \frac{c}{d}$, we have that $cf>de$ or $cf - de > 0$ and we have $be>af$ or $be - af > 0$. Thus, all the quantities are positive assuming we take $r$ to be as well. 

The next question is whether we can actually achieve any combination of $m$ and $n$ from the mediant process. The answer is no, we cannot. The final condition we need is that $m$ and $n$ are coprime. This is what prevents us from claiming that we can achieve multiple forms of the rational number from this process, something we cannot do. In fact, the requirement of it being coprime is what changes the conclusion of the example above with the target of $\frac{9}{10}$. While we will establish why we need the coprimeness, let us first show that we can ensure that there do exist $m$ and $n$ satisfying these equations that are coprime. 

Let $cf-de = st$ and $be-af = su$ where $t$ and $u$ are coprime and $s, t, u$ are integers.  Then mutliplying the first equation by $a$, the second equation by $c$, and then combining, we end up with $e (bc -ad) = s (at  + cu)$. Multiplying the first equation by $b$ and the second by $d$ leads to $f (bc - ad) = s(bt + du)$. This means $s$ divides both $e(bc-ad)$ and $f(bc-ad)$. Since $e$ and $f$ are coprime, this means that $s$ must divide into $bc-ad$. If we take $r = \frac{(bc-ad)}{s}$, then we have that $m=t$ and $n=u$ implying they are corpime. 

The final part of our proof is to establish that we can construct a mediant process given coprime $m$ and $n$. This is where we use Theorem 1 from \cite{richards}. The form $\frac{ma + nc}{mb+nd}$ for coprime $m$ and $n$ is obtained in the exact same way as $\frac{m}{n}$ is in the Farey process where $\frac{a}{b}= \frac{0}{1}$ and $\frac{c}{d}=\frac{1}{0}$ are the starting points. Since the Farey process always results in a fraction without common factors, we have to have $m$ and $n$ coprime for this to work.  

We need to establish comparative ordering for fractions of the form $\frac{ma + nc}{mb + nd}$. So let's consider such a fraction with $m,n$ as one pair and $p,q$ representing another pair. The numerator of their difference is 
\begin{flalign*}
(ma+nc)(pb+qd) & - (pa+qc)(mb+nd)  & \\
 & = pmab + pnbc + qmad + qncd -  mbpa - mqbc -  ndpa -ndqc & \\
& =  pnbc -  ndpa + qmad  - mqbc  + (pmab -  mbpa) + ( qncd  -ndqc) & \\
 & = pnbc -pnad +qmad- mqbc & \\
 & = pn(bc-ad) + qm(ad-bc) & \\ 
 &= (pn-qm)(bc-ad) & 
\end{flalign*}
The denominator is $(mb+nd)(pb+qd)$.  Since we assume $m,n,p,q$ are non-negative, and we can take the denominators $b$ and $d$ to be positive, and we have $bc-ad > 0$ due to $\frac{a}{b} < \frac{c}{d}$, we can see that the difference of these two forms is entirely based on the computation of $pn - qm$. This implies that the sequence of choices about which subinterval to choose will always be decided in the same way. Specifically, $pn - qm > 0$ if and only if $\frac{p}{q} > \frac{m}{n}$.

Furthermore, we always start with $m=0, n=1$ and $m=1, n=0$. The next mediant will always be $m=1, n=1$. After that, we need to start comparing with our target. Let's say it is $\frac{Ma + Nc}{Mb + Nd}$. Then if $\frac{M}{N} > \frac{1}{1} = 1$, then we will go right and be looking at the interval $\frac{a+c}{b+d} \lte  \frac{c}{d}$ leading to the next mediant being $m=1, n=2$. If $\frac{M}{N} < 1$, then we go left and look at the interval $\frac{a}{b} \lte \frac{a+c}{b+d}$ leading to the next mediant being $m=2, n=1$. We proceed in a similar fashion, comparing the target $\frac{M}{N}$ to the current $\frac{m}{n}$ and deciding which interval to select from that. When we achieve equality, we stop the process.

What we have just established is that we can start with any interval $\frac{a}{b} \lte \frac{c}{d}$ and use the Stern-Brocot tree to obtain any rational in that interval by mapping $\frac{m}{n}$ to $\frac{ma + nc}{mb + nd}$. This does not require the original endpoints to be in reduced form. Having them in different scaled versions leads to different mappings of the intervals, but the structure of the tree is always present. 

Let us consider the simple interval $[1, 2]$ with target $\frac{3}{2}$. For this, the target is $m=1, n=1$ and happens after the first step. But now let's change the endpoint $2$ to $\frac{6}{3}$. This is the same rational, but it changes the $m$ and $n$ for the target. In particular, since $6*1-3*1 = 3$, we expect the target to be $\frac{9}{6}$ with $m=3, n=1$, generating the intervals $\frac{1}{1} :\frac{6}{3}$,  $\frac{1}{1} :\frac{7}{4}$, $\frac{1}{1} :\frac{8}{5}$, $\frac{1}{1} :\frac{9}{6}$. We can view this as we initially weighted the right endpoint more and so we start farther to the right and thus we need to go left a few times to compensate. 

In summary, the mediant process of approximating a real number given a starting interval will always terminate when the real number is a rational in the rational interval. There is no condition on what the interval is or the form of the endpoints. All we require is that the target rational number is in the initial interval and then the process will terminate.\footnote{This is assuming that we have also tested the endpoints. If we did not and start with the mediant and the endpoint is the target, then the process will converge towards it, never switching. For example, $[0; n]$ is $\frac{1}{n}$ and is the sequence of approximations to $0$ if we start with $0:1$ and compare only the mediant intervals of $0:\frac{1}{2}$ and $\frac{1}{2}:1$. As $n$ goes to infinity, we get ever better approximations of $0$, but we never reach $0$.}

\medskip

\normalem 

\printbibliography

\end{document}